\documentclass[11pt,oneside]{amsart}
\usepackage{amsmath, ifthen, amsfonts, amssymb, srcltx, amsopn, xcolor, enumerate, amsthm, comment}

\usepackage[linktocpage=true]{hyperref}

\usepackage[all]{xy}
\usepackage{pb-diagram, pb-xy}
\usepackage{overpic}
\usepackage{tikzsymbols}
\usepackage{tikz-cd}
\usepackage{ circuitikz }
\usepackage{glossaries}

\usepackage[T1]{fontenc}
\usepackage{stmaryrd}

\usepackage[usehighlevels]{alnumsec}
\alnumsectionlevels{1}{section,subsection,subsubsection,paragraph}
\alnumsecstyle{aaaa}
\surroundarabic[{.}][]{}{}
\otherseparators{1}
\alnumsecstyle{aaaa}

\newcommand{\nocontentsline}[3]{}
\let\oldaddcontentsline\addcontentsline
\newcommand{\tocless}[2]{%
  \let\addcontentsline\nocontentsline
  #1{#2}
  \let\addcontentsline\oldaddcontentsline}

\newcommand{\showcommentsbox}{yes}

\newsavebox{\commentbox}
{\ifthenelse{\equal{\showcommentsbox}{yes}}%
{\footnotemark
        \begin{lrbox}{\commentbox}
        \begin{minipage}[t]{1.25in}\raggedright\sffamily\tiny
        \footnotemark[\arabic{footnote}]}
{\begin{lrbox}{\commentbox}}}%
{\ifthenelse{\equal{\showcommentsbox}{yes}}%
{\end{minipage}\end{lrbox}\marginpar{\usebox{\commentbox}}}
{\end{lrbox}}}

\newcounter{ax}
\newtheorem{thm}{Theorem}[section]

\newtheorem{lem}[thm]{Lemma}

\newtheorem{cor}[thm]{Corollary}
\newtheorem{conj}[thm]{Conjecture}

\newtheorem{prop}[thm]{Proposition}

\newtheorem{thmi}{Theorem}

\newtheorem{cori}[thmi]{Corollary}

\newtheorem*{exceptCor}{Theorem~\ref{exceptCor}}

\newtheorem*{lem:unbranched-implies-one-Nielsen}{Theorem~\ref{lem:unbranched-implies-one-Nielsen}}

\theoremstyle{definition}
\newtheorem{defn-inner}[thm]{Definition}
\newenvironment{defn}{%
  \pushQED{\qed}%
  \begin{defn-inner}%
}{%
  \popQED
  \end{defn-inner}%
}

\newtheorem{exmp-inner}[thm]{Example}
\newenvironment{exmp}{%
  \pushQED{\qed}%
  \begin{exmp-inner}%
}{%
  \popQED
  \end{exmp-inner}%
}

\newtheorem{notation-inner}[thm]{Notation}

\newtheorem*{thm*}{sketch of Proof of Theorem \ref{\maptor}}

\newtheorem{claim}{Claim}

\newtheorem{claim*}{Claim}

\newtheorem*{theorem*}{Theorem}

\newtheorem{question}[thm]{Question}
\newtheorem{example}[thm]{Example}

\newtheorem{remark-inner}[thm]{Remark}
\newenvironment{remark}{%
  \pushQED{\qed}%
  \begin{remark-inner}%
}{%
  \popQED
  \end{remark-inner}%
}

\DeclareMathOperator{\Fix}{Fix}

\DeclareMathOperator{\kernel}{ker}
\DeclareMathOperator{\image}{im}
\DeclareMathOperator{\rank}{rk}

\DeclareMathOperator{\Aut}{Aut}
\DeclareMathOperator{\Out}{Out}

\DeclareMathOperator{\stabilizer}{Stab}

\newcommand{\neb}{\mathcal N}

\newcommand{\factor}[2]{{\raise0.7ex\hbox{$#1$} \!\mathord{\left/ {\vphantom {#1 {#2}}}\right.\kern-\nulldelimiterspace}\!\lower0.7ex\hbox{${#2}$}}}

\newcommand{\field}[1]{\mathbb{#1}}
\newcommand{\integers}{\ensuremath{\field{Z}}}

\newcommand{\naturals}{\ensuremath{\field{N}}}
\newcommand{\reals}{\ensuremath{\field{R}}}

\newcommand{\GG}{\ensuremath{\field{G}}}

\newcommand{\boundary}{{\ensuremath \partial}}

\makeatletter

\newcommand{\Rmnum}[1]{\mathbf{{\expandafter\@slowromancap\romannumeral #1@}}}

\makeatother
\DeclareMathOperator{\Isom}{Isom}

\let\oldmarginpar\marginpar
\renewcommand\marginpar[1]{\-\oldmarginpar[\raggedleft\footnotesize #1]{\raggedright\footnotesize #1}}

\newcommand{\tsh}[1]{\left\{\kern-.7ex\left\{#1\right\}\kern-.7ex\right\}}

\newcommand{\co}{\colon}

\newcounter{enumitemp}

\newcommand{\calC}{\mathcal C}
\newcommand{\gate}{\mathfrak g}

\newcommand{\hhsconst}{E} 

\usepackage{mathabx}

\newcommand{\edges}{\mathbf{Edges}}
\newcommand{\vertices}{\mathbf{Vertices}}

\newcommand{\nest}{\sqsubseteq}

\newcommand{\orth}{\bot}
\newcommand{\transverse}{\pitchfork}

\newcommand{\out}{\ensuremath{\mathrm{Out}(\mathbb{F}) } }
\newcommand{\aut}{\ensuremath{\mathrm{Aut}(\mathbb{F}) }}

\newcommand{\maptor}{\Gamma}

\newcommand{\F}{\ensuremath{\mathbb{F} } }
\newcommand{\bigset}{\mathrm{Big}}

\begin{document}
\title{Characterizing hierarchically hyperbolic free by cyclic groups}

\author{Eliot Bongiovanni}\thanks{The first author was partially supported by NSF grant DMS-1745670.}
    \address{Department of Mathematics \\ Rice University \\ Houston, Texas}
    \curraddr{Department of Mathematics \\ University of Michigan \\ Ann Arbor, Michigan}
    \email{\url{eb40@rice.edu}, \url{bongio@umich.edu}}
    \urladdr{https://eliotbongiovanni.com}
\author{Pritam Ghosh}\thanks{The second author was supported by the Ashoka University faculty research grant.}
 \address{Department of Mathematics \\
Ashoka University\\
  Haryana 131029, India}
  \email{\url{pritam.ghosh@ashoka.edu.in}}
\author {Funda G\"ultepe}\thanks{The third author was partially
supported by NSF grant DMS-2137611.}
\address{Department of Mathematics and Statistics\\
 University of Toledo\\
 Toledo, OHIO}
\email{\url{funda.gultepe@utoledo.edu}}
\urladdr{https://sites.google.com/view/fundagultepe}
\author{Mark Hagen}
\address{School of Mathematics, University of Bristol, Bristol, UK}
\email{\url{markfhagen@posteo.net}}
\urladdr{https://www.wescac.net/}
\date{\today}

\begin{abstract}
We algebraically characterize free by cyclic groups that have coarse medians, and prove that this is equivalent to the a priori stronger properties of being colorable hierarchically hyperbolic groups and being quasi-isometric to CAT(0) cube complexes.  Our algebraic characterization involves a condition on intersections between maximal virtually $F_n\times \mathbb Z$ subgroups, which we call being \emph{unbranched}.  We also characterize hierarchical hyperbolicity of $\Gamma=F_n\rtimes_{\phi}\mathbb Z$ in terms of a property of completely split relative train track representatives of $\phi\in\mathrm{Out}(F_n)$ that we call \emph{excessive linearity}, a slight refinement of the \emph{rich linearity} condition for relative train track maps introduced by Munro and Petyt. 
\end{abstract}

\maketitle

\tableofcontents

\section{Introduction}\label{sec:intro}
Let $\F$ be a free group of finite rank $\geq 3$ and let $\Aut(\F)$ be the group of automorphisms of $\F$. The quotient group
$\out=\aut/\mathrm{Inn}(\F)$ is  the group  of outer automorphisms of $\F$.  A group $\maptor $ is \emph{free by cyclic} if it fits in a short exact sequence
   \[ 1\to \F \to \maptor  \to {\langle \phi \rangle } \to 1.\]
 In other words  $\maptor=\F\rtimes_\phi\integers$
 is the pullback of $\mathbb Z=\langle \phi \rangle $ to $\aut$,  where $\phi\in \out$. 
 $\maptor$ is the \emph{mapping torus} of the outer automorphism $\phi$, following the terminology of fundamental groups of mapping tori surface homeomorphisms, and has a presentation
 \[  \maptor=\F\rtimes_\phi\integers=\langle \F, t\mid \{txt^{-1}= \Phi(x):x\in \F\} \rangle, \]
where $\Phi\in\Aut(\F)$ represents $\phi$; the group $\maptor$ depends on $\phi$ but not on the choice of $\Phi$. 

An outer automorphism $\phi\in \out$ is of one of three dynamical types according to its growth rate, which captures how the lengths of (conjugacy classes of) words in $\mathbb{F}$ grow under repeated applications of $\phi$.
We say that $\phi$ is \emph{exponentially growing}, \emph{polynomially growing}, or \emph{finite order}, respectively, if for some conjugacy class $[w]$ of an element $w\in \mathbb F$, the word length (for any fixed generating set of $\F$) of $\phi^{i}([w])$ grows exponentially, or every $[w]$ grows polynomially, or every $[w]$ has bounded length.
The geometry of free by cyclic groups varies with this dynamical type: For example, when $\phi$ is exponentially growing and atoroidal, 
$\maptor$ is hyperbolic \cite{Br-00}, and more generally, $\phi$ is (non-atoroidal)  exponentially growing if and only if $\maptor$ is hyperbolic relative to subgroups with polynomial growth \cite{Ghosh, DahmaniLi}.

Comparatively little is clear about the geometry of free by cyclic groups with polynomial growth.
Some act freely on CAT(0) cube complexes \cite{Button:tubular-cubes,HagenWise:new-polynomial}, which suggests some hyperbolic features (since any hyperbolic free by cyclic group is cocompactly cubulated \cite{HagenWise:irred, HagenWise:general, DahmaniKrishnaMutanguha}). However, it is unknown exactly when there is a free action on a CAT(0) cube complex, and in general such actions cannot be cocompact \cite{Gersten,MunroPetyt}. 
In fact, there are free by cyclic groups (in fact, geometric ones) with only linear growth that are CAT(0) but not virtually cocompactly cubulated; see Example \ref{exmp:not-via-specialness}, which uses \cite{HagenPrzytycki}.
Free by cyclic groups with polynomial growth are \emph{thick} in the sense of \cite{BDM} by \cite{Macura:detour,Hagen:thickness}, and hence are not hyperbolic relative to any collection of proper subgroups, using \cite{BDM}.  Yet, they exhibit some coarse nonpositive curvature properties: For example, they famously satisfy a quadratic isoperimetric inequality \cite{BridsonGroves}.

This paper identifies the missing piece of the puzzle: Under a straightforward algebraic condition related to their linearly growing components, polynomial free by cyclic groups---and hence all free by cyclic  groups satisfying the condition---are hierarchically hyperbolic groups and, moreover, are quasicubical (i.e. quasi-isometric to a CAT(0) cube complex). We refer to this condition as being \emph{unbranched}.
Roughly speaking, a group is unbranched if it is impossible for three pairwise-incommensurable maximal $\F\times\integers$ subgroups to intersect along a non-cyclic subgroup.  
Our main theorem incorporates these and more conditions from existing literature, showing that they are in fact all equivalent.
We prove this in Section \ref{sec:proof}; Figure \ref{fig:main-theorem-proof} illustrates the implications in the proof and highlights the contributions made by this paper.

\begin{thmi}\label{thm:main}
Let $\F$ be a finite-rank free group (of rank at least $2$), let $\phi\in\Out(\F)$, and let $\maptor=\F\rtimes_\phi\integers$ be the mapping torus.  Then the following are equivalent:
\begin{enumerate}[(I)]
    \item \label{item:unbranched-blocks} $\maptor$ is \emph{unbranched} (see Definition \ref{defn:block-unbranched}).
    
    \item\label{item:coarse-median} $\maptor$ admits a coarse median, in the sense of \cite{Bowditch:coarse-median}.

    \item \label{item:virtual-hhg} $\maptor$ is \emph{virtually} a colorable HHG, in the sense of \cite{DMS}.

    \item \label{item:actual-hhg} $\maptor$ is a colorable HHG.

    \item \label{item:no-2-RBF} $\maptor$ has no QI--embedded dimension--$2$ richly branching flat (Definition \ref{2RBF}). 

    \item \label{item:QI-to-cube-complex} $\maptor$ is quasi-isometric to a finite-dimensional CAT(0) cube complex.
\end{enumerate}
Moreover, the coarse median in \eqref{item:coarse-median} has finite rank if it exists.
\end{thmi}

In other words, Theorem \ref{thm:main} characterizes \emph{all} hierarchically hyperbolic free by cyclic groups, regardless of growth rate, answering a problem of \cite{Banff}\footnote{This problem list, compiled from discussions at the 2024 BIRS conference \emph{Advances in Hierarchical Hyperbolicity}, can be accessed at \url{https://www.wescac.net/HHG_Question_Session-dec-2024.pdf}.}. 
We note that the case $\rank(\F)=2$ already follows from \cite{Ghosh,DahmaniLi} and \cite{HRSS}; also, the ``unbranched'' condition holds automatically in rank $2$.

This paper further shows that the relatively hyperbolic free by cyclic groups characterized in  \cite{Ghosh, DahmaniLi} are hierarchically hyperbolic and quasicubical provided the peripheral subgroups are, and this in turn just depends on linear growth subgroups. 

\begin{cori}\label{cor:RH} 
There is a finite set $\{G_i\}_{i=1}^r$ of subgroups of $G$ such that each $G_i$ is isomorphic to the mapping torus of a linear-growth automorphism of a finitely generated subgroup of $\F$ and $G$ is hierarchically hyperbolic if and only if each $G_i$ is hierarchically hyperbolic if and only if each $G_i$ is unbranched.

In particular, relatively hyperbolic free by cyclic groups are hierarchically hyperbolic if and only if the peripheral subgroups are unbranched.
\end{cori}

\begin{proof}
Use Proposition \ref{prop:relhyper} (which is from \cite{Ghosh}), Theorem \ref{thm:main}, and \cite[Thm. 9.1]{HHSII} to verify the second assertion, and reduce the first assertion to the case where $G=\F\rtimes_\phi\integers$ with $\phi$ polynomially-growing.  By Theorem \ref{thm:main}, we can pass to a finite index subgroup and assume that $G$ has an acylindrical cyclic hierarchy as in Proposition \ref{prop:cyclic-hierarchy}; the terminal vertex groups provide the $G_i$ from the statement, and $G$ is unbranched if and only if the $G_i$ are, by Proposition \ref{prop:unbranched-cyclic-splitting}, and we conclude by applying Theorem \ref{thm:main} to $G$ and each $G_i$.
\end{proof} 
\noindent Similarly, we also show that hierarchical hyperbolicity of $G$ can be characterized by examining linearly-growing strata in a completely split relative train track representative of (a power of) $\phi$, in Theorems \ref{exceptCor} and \ref{lem:unbranched-implies-one-Nielsen}.

Hierarchically hyperbolic spaces/groups (HHS/G) were introduced in \cite{HHSI, HHSII}  to study cubical groups, mapping class groups of surfaces of finite type, and other natural examples. Hierarchical hyperbolicity uses an axiomatic framework involving a collection of projections to hyperbolic spaces, which enables studying the geometry of the space by working in the hyperbolic ``coordinate spaces''.  This reveals  features reminiscent of hyperbolicity in the whole space. The class of HHSs includes 
mapping class groups and many CAT(0) cube complexes (including all universal covers of compact special cube complexes as in \cite{HaglundWise:special}), along with Gromov hyperbolic spaces (trivially, though importantly), extensions of finitely generated Veech groups \cite{Bongiovanni:Veech}, Teichm\"{u}ller space with any of the usual metrics, many Artin groups \cite{HHSI,HMSxl}, and more.  On the other hand, $\Out(\F)$ is not an HHS/HHG for isoperimetric inequality reasons \cite{HHSII}.

Theorem \ref{thm:main} involves some subtleties about HHGs, and says that a free by cyclic group is either not even a hierarchically hyperbolic \emph{space}, or is an HHG in a particularly strong way. Colorability (Definition \ref{defn:colourable}) does not hold for all HHGs \cite{Hagen:colour}, and is required in order to get the conclusion about quasi-isometries to cube complexes.
Although being an HHS is a quasi-isometry invariant of a space, for a group $G$ to be an HHG, one needs not only that any Cayley graph is an HHS, but also that the hierarchically hyperbolic structure is $G$--invariant in an appropriate sense.  Hence, being an HHG is not only not a quasi-isometry invariant property of a group, it is not even a commensurability invariant \cite{PetytSpriano}, so the implication \eqref{item:virtual-hhg}$\implies$\eqref{item:actual-hhg} in Theorem \ref{thm:main} has actual content.

\begin{figure}[!ht]
    \centering
    \resizebox{0.9\textwidth}{!}{
    \begin{circuitikz}
    \tikzstyle{every node}=[font=\large]
    \node [font=\LARGE] at (2.75,12.75) {\ref{item:actual-hhg}};
    \draw [->,double distance=3pt, >=Stealth] (8.5,12.25) -- (9.75,11)node[pos=0.3,right]{\phantom{aa}\cite[Theorem B]{Pet}};
    \node [font=\LARGE] at (8,12.75) {\ref{item:virtual-hhg}};
    \node [font=\LARGE] at (10.25,10.5) {\ref{item:QI-to-cube-complex}};
    \node [font=\LARGE] at (12.5,8.25) {\ref{item:coarse-median}};
    \node [font=\LARGE] at (14.75,6) {\ref{item:no-2-RBF}};
    \node [font=\LARGE] at (8,8.25) {\ref{item:unbranched-blocks}};
    \draw [->,double distance=3pt, >=Stealth] (11.75,8.25) -- (8.75,8.25)node[pos=0.5,below]{};
    \draw [->,double distance=3pt, >=Stealth] (8,9) -- (8,12)node[pos=0.5,left]{Lemma \ref{lem:unbranched-implies-HHG}};
    \draw [->,double distance=3pt, >=Stealth] (14,6) -- (8.75,7.75)node[pos=0.4,below]{Lemma \ref{lem:coarse-median-implies-unbranched-blocks}\phantom{aaaaaaaaaaaaa}};
    \draw [->,double distance=3pt, >=Stealth] (10.75,10) -- (12,8.75)node[pos=0.3,right]{\phantom{aa}definition of coarse median};
    \draw [->,double distance=3pt, >=Stealth] (13,7.75) -- (14.25,6.5)node[pos=0.3,right]{\phantom{aa}\cite{MunroPetyt}};
    \draw [->,double distance=3pt, >=Stealth] (7.5,12.5) -- (3.25,12.5)node[pos=0.5,below]{Proposition \ref{prop:remove-virtual}};
    \draw [->,double distance=3pt, >=Stealth] (3.25,13) -- (7.5,13)node[pos=0.5,above]{\small{(trivially)}};
    \end{circuitikz}
    }
    \caption{Sketch of proof of Theorem \ref{thm:main}.}
\label{fig:main-theorem-proof}
\end{figure}

\begin{proof}[Sketch of Proof of Theorem \ref{thm:main}]
   Given $\maptor=\F\rtimes_\phi\integers$  as in the statement of Theorem \ref{thm:main} (where we identified free by cyclic group with the fundamental group of its associated mapping torus), using Proposition \ref{prop:relhyper}, which restates \cite[Thm. 3.5]{DahmaniLi} or \cite[Thm. 3.15]{Ghosh}, we reduce to the case where $\phi$ is polynomially growing. When $\phi$ is polynomially growing  with degree $d\geq 2$, we use the cyclic hierarchy (Proposition \ref{prop:cyclic-hierarchy}) to split $\maptor$ into a \emph{hierarchy} of  graph of group decompositions where the edge groups are cyclic and at the bottom of the hierarchy vertex groups are mapping tori of free group automorphisms of polynomial growth rate at most $1$. For these latter vertex groups we use a graph of groups  splitting of their mapping tori  where the edge groups are $\mathbb Z \times \mathbb Z$ (Proposition \ref{prop:split-over-tori}). 

For linearly growing mapping tori, when $\maptor$ is unbranched, graph of groups decomposition is \emph{admissible} (\cite{MunroPetyt}), and hence HHG by \cite{HRSS}. Hence when $\maptor$ is unbranched with polynomial growth $d\geq 2$, we prove that the cyclic hierarchy $\maptor$ splits into is an \emph{acceptable}  graph of groups decomposition of $\maptor$ (Definition \ref{defn:acceptable-cyclic-HHG-graph}) where the ending vertex groups are HHG by admitting admissible decomposition. Then using combination Theorem of \cite{HHSII} we show in Lemma \ref{lem:combination-theorem} that the total group of a acceptable graph of groups decomposition admits a HHG structure with some additional properties. Hence, when $\maptor$ is unbranched it is quasi-isometric to a finite-dimensional CAT(0) cube complex (\cite{Pet}) and is in particular coarse median of finite rank.  

This sketches the proof of $\eqref{item:unbranched-blocks}\implies\eqref{item:virtual-hhg}$\label{lem:unbranched-implies-HHG} and summarizes core ideas of the proof of Theorem \ref{thm:main}.  For the rest of the sketch we refer the reader to the chart given in Figure \ref{fig:main-theorem-proof}.
\end{proof}

Additionally, our results imply that a free by cyclic group satisfying any of the equivalent conditions given in Theorem \ref{thm:main} also satisfies previously unknown properties. For example, there is a \emph{distance formula} (\cite[Thm. 4.4, 4.5]{HHSII}) which measures, in the sense of Masur and Minsky \cite{MM2}, the distance between points in terms of projection distances into the coordinate sets of the HHS structure. Another consequence of our results is the \emph{local-to-global Morse property} ([Question 1.2,\cite{RST-MLTG}]) for free by cyclic groups that are included in our theorems. We remark that local-to-global Morse property for \emph{all} free by cyclic groups is recently proven in \cite{KudPetyt} by different methods. 

Here are consequences of Theorem \ref{thm:main} that, to our knowledge, have not been proved for free by cyclic groups by other methods:

\begin{cori}\label{cor:intro-consequences}
Let $\maptor$ be a free by cyclic group.  If $\maptor$ is unbranched, then  the following hold:
\begin{enumerate}[(i)]
    \item $\maptor$ acts freely and coboundedly on an injective metric space.\label{item:injective-space}

    \item $\maptor$ is \emph{semihyperbolic} in the sense of \cite{AB}.\label{item:semihyperbolic}

    \item $\maptor$ acts geometrically on a metric space whose asymptotic cones are all CAT(0).\label{item:asymp-cat}

    \item Every asymptotic cone of $\maptor$ is bilipschitz to a \emph{real cubing} in the sense of \cite{CRHK}.\label{item:cone-median}
\end{enumerate}
\end{cori}

\begin{proof}
Given Theorem \ref{thm:main}, items \eqref{item:injective-space} and \eqref{item:semihyperbolic} follow from \cite[Cor. H, Prop. 1.1]{haettel_coarse_2020}, item \eqref{item:asymp-cat} from \cite[Thm. A]{DMS2}, and item \eqref{item:cone-median} from \cite[Thm. C]{CRHK}.  Alternatively, item \eqref{item:semihyperbolic} also follows from Theorem \ref{thm:main} and \cite[Cor. C]{DMS}.
\end{proof}

%

We next discuss a version of Theorem \ref{thm:main} phrased in terms of relative train tracks on graphs, which topologically represent outer automorphisms of the free group \cite{BFH-00,BestvinaHandel:ttm}.   
Relative train track theory is extremely powerful and has undergone various upgrades. We work with \emph{completely split improved relative train track maps (CTs)}, introduced in \cite{FeighHandel:recognition}; see Section \ref{subsec:IRTT}.  

\begin{thmi}\label{exceptCor}
Let $\phi\in\Out(\F)$ be an infinite order element.  Then the following are equivalent, where $k\geq 1$ is a constant depending only on $\rank(\F)$: 
\begin{enumerate}
    \item $\F\rtimes_\phi\integers$ is a colorable hierarchically hyperbolic group. 
    \item Given any $\phi^k$--invariant free factor $[F^1]$ such that $\phi^k|_{F^1}$ has at most linear growth, no CT map representing this restriction has \emph{excessive linearity}.
    \label{item:no-exceptional} 
\end{enumerate}
\end{thmi}

\emph{Excessive linearity} (Definition \ref{defn:excessive-linearity}) is a condition on CT representatives of linearly growing automorphisms that is very closely related to the \emph{rich linearity} used in \cite{MunroPetyt} as an obstruction to the existence of coarse median structures on mapping tori. Question 2 of \cite{MunroPetyt} asks whether rich linearity characterizes mapping tori without coarse median structures; this is not quite the case (Example \ref{exmp:not-rich}), but Theorem \ref{exceptCor} shows that Munro-Petyt's question does not need to change much to have a positive answer.  
Excessive linearity can be read from a CT map, but has an interpretation in terms of \emph{principal lifts} of $\phi$ to $\Aut(\F)$ (Remark \ref{rem:principal-lifts}).  Excessive linearity is closely related to  the existing notion of \emph{exceptional paths} (see \cite[p. 42-43]{HM-20}) in the graph supporting a CT representative of the automorphism, but excessive linearity is a strictly more general property than having an exceptional path. 

Any automorphism which is rotationless unipotent polynomially growing (UPG) is represented by a CT map, perhaps after passing to a positive power depending only on $\F$.
In the linear case, we show another equivalence between excessive linearity and the mapping torus failing to be unbranched.
We combine this result, stated below, with Theorem \ref{thm:main} to prove Theorem \ref{exceptCor}.

\begin{thmi}\label{lem:unbranched-implies-one-Nielsen}
Suppose $\phi\in\Out(\F)$ is a rotationless UPG automorphism of at most linear growth.  Then the following statements are equivalent: 
\begin{enumerate}
    \item The mapping torus $\maptor$ of $\phi$ is not unbranched. \label{maptorunbranched}
    \item Any CT map $f:\mathbb G\to\mathbb G$ representing $\phi$ has excessive linearity.\label{onelinearedgepercycle} 
    \end{enumerate}
\end{thmi}

Theorem \ref{lem:unbranched-implies-one-Nielsen} is proved by analyzing the tree of cylinders splitting of $\maptor$ from the proof of \cite[Thm. 2.7]{DahmaniTouikan}. 
It is of independent interest because it enables one to check systematically from a CT representative whether the mapping torus is hierarchically hyperbolic.  It then follows from \cite{feighn-handel:algorithmic-constructions} that there is an algorithm that decides, given an element of $\Out(\F)$, whether the mapping torus is hierarchically hyperbolic. 

Relative train track theory plays a direct role in the proofs of Theorems \ref{exceptCor} and \ref{lem:unbranched-implies-one-Nielsen} only.
The proof of Theorem \ref{thm:main} is based on a collection of structural results about mapping tori---the relative hyperbolicity results from \cite{DahmaniLi,Ghosh}, the cyclic hierarchy from \cite{Macura:detour}, and the JSJ-type results from \cite{AndrewMartino:splitting,DahmaniTouikan}---along with tools from \cite{MunroPetyt} and \cite{HHSII}. 

Finally, a special feature of our proof of Theorem \ref{thm:main} is that we obtain an explicit hierarchically hyperbolic structure in which the coordinate spaces are either bounded or are one of several well-understood types summarized as follows.
The non--$\nest$--maximal spaces are quasi-trees (many of them Bass-Serre trees of some well-known splitting), and in particular the $\nest$--minimal spaces are quasi-lines.
In the case of relatively hyperbolic groups, the $\nest$--maximal space is the coned-off Cayley graph.


\tocless{\subsection{Further questions}}{\label{subsec:questions}}
We now propose some questions.

\subsubsection{(Non)-existence of Cannon-Thurston maps} Hierarchically hyperbolic groups are compactified using the hierarchical  boundary (\emph{HHS boundary}) from \cite{DHSbound}.  One description is as follows.  Any two points in an HHS/HHG are connected by a \emph{hierarchy path} \cite[Thm. 4.4, Thm. 4.5]{HHSII}, which is 
a uniform quasigeodesic projecting to a uniform unparametrized quasigeodesic in the hyperbolic spaces. Imitating 
the constructions of the Gromov boundary and the visual boundary of a CAT(0) space, a boundary for HHS/HHGs was constructed in \cite{DHSbound, DHScorr} by taking as boundary points 
asymptotic classes of “hierarchy rays” emanating from a
fixed basepoint and having a ``well-defined gradient''.    For hyperbolic groups, the HHS and Gromov boundaries coincide \cite[Thm. 4.3]{DHSbound}. We suggest exploring HHS boundaries of mapping tori satisfying the conditions in Theorem \ref{thm:main}.

For hyperbolic groups $H\leq G $ \emph{the Cannon–Thurston map} is a continuous map
$\widehat \iota: \hat H\rightarrow \hat G$ that extends the inclusion $\iota: \mathcal C(H) \hookrightarrow \mathcal C(G)$ where $\mathcal C (H)$ and $\mathcal C(G)$ are Cayley graphs for $H$  and $G$, respectively, given by choosing a finite generating set for $H$ and extending it to a finite  generating set for $G$; such an extension may or may not exist. For hyperbolic free by cyclic groups,  Cannon--Thurston maps exist by work of Mitra (\cite{MitraCT}), and by work of Bowditch (\cite{BowRH}), a Cannon--Thurston map for suitable subgroups $\F \triangleleft G$ of a relatively hyperbolic group $G$ also exists. However, there are examples of relatively hyperbolic free by cyclic groups for which the Cannon--Thurston map does not exist (\cite{BGGGS}.) We ask:%

\begin{question}\label{BdryCor} Let $\maptor$ be an unbranched mapping torus and let $H<\maptor$ be a HHG subgroup. When does the inclusion map $H \hookrightarrow  \maptor=\F\rtimes_\phi\mathbb Z$  extend to a continuous map $ \boundary_{HH}H \longrightarrow \boundary_{HH}\maptor$ between HHG boundaries? 
\end{question}

When $H<\maptor $ is hierarchically quasiconvex, the answer is positive by \cite{DHSbound}. This case includes those where: $H$ is one of the peripheral polynomial mapping tori, in the relatively hyperbolic case; $H$ is the centralizer of some element; $H$ is a maximal abelian subgroup; $H$ is one of the vertex groups in the cyclic hierarchy from Proposition \ref{prop:cyclic-hierarchy}, etc.

When $H$ is a free subgroup, then $\boundary_{HH}H$ is the Gromov boundary.  In many cases, $H$ is hierarchically quasiconvex, so the answer is positive. However if $H$ is a free \textbf{normal} subgroup, e.g. $H=\F$, then we suspect the answer is positive only if $\phi$ is atoroidal (the converse follows from \cite{MitraCT} since hierarchical boundaries of hyperbolic groups are just the Gromov boundaries). We suspect that the answer is negative if $H$ is free normal and $\maptor$ is not hyperbolic, using arguments similar to \cite{BGGGS}. In fact we conjecture:

\begin{conj}\label{conj:no-cannon-thurston}    
 Let $G$ be a HHG and  and $H <G$ an infinite hyperbolic
normal subgroup of infinite index. Then the Cannon--Thurston map does not exist for $(H,G)$ unless either $G$ is hyperbolic or $H$ is hierarchically quasiconvex in some HHG structure on $G$.
\end{conj}

\subsubsection{Asymptotic cones}  Another question concerns Corollary \ref{cor:intro-consequences}.\eqref{item:cone-median}.  Recall that this says that asymptotic cones of unbranched free by cyclic groups are real cubings, since this is true of all HHGs.  Real cubings are a special class of median spaces.  We ask:

\begin{question}\label{question:no-median}
Let $\maptor$ be a free by cyclic group that is not unbranched, and hence contains a $2$--RBF.  Can the proof of \cite[Thm. 4.3]{MunroPetyt} be adapted to show that no asymptotic cone of $\maptor$ is bilipschitz equivalent to a median metric space?  
\end{question}

The answer is likely positive; one should bound the rank of such a median structure (using something like \cite[Cor. 3.7]{Haettel:lattices}) and then imitate the argument from \cite[Thm. 4.3]{MunroPetyt}.

\subsubsection{Other nonpositive curvature properties}
Semihyperbolicity was introduced by Alonso and Bridson in \cite{AB}, and is of interest, for instance, because it is a necessary condition for the existence of biautomatic structures, and because semihyperbolicity is sufficient for solving the conjugacy problem.  As far as we know, semihyperbolic and biautomatic free by cyclic groups have not been classified, but non-biautomatic examples exist \cite{BradyBridson}. 

Semihyperbolicity is a consequence of other (coarse) nonpositive curvature properties.  To begin with, semihyperbolicity follows from the existence of a proper, cocompact action on a CAT(0) space.  A classical example due to Gersten \cite{Gersten} shows that even linear-growth mapping tori need not be CAT(0), and in Example \ref{exmp:gersten}, we examine the JSJ decomposition to see that this example fails to be unbranched.

On the other hand, Example \ref{exmp:cat(0)-non-hhg} shows that $\maptor$ can be CAT(0) without being unbranched/an HHG.  There are many such examples.  This is not surprising: as a general matter, the existence of CAT(0) structures on a group says little about hierarchical hyperbolicity.  However, CAT(0) \emph{cubical} structures often give rise to HHG structures \cite{HHSI,HagenSusse} (and also always give biautomatic structures \cite{NibloReeves:biautomatic}).  But Example \ref{exmp:not-via-specialness} illustrates that Theorem \ref{thm:main} gives HHGs that are not even virtually cubical.  All of this motivates:

\begin{question}\label{question:non-cat0-hhg}
Does there exist $\phi\in\Out(\F)$ such that $\F\rtimes_\phi\integers$ is unbranched (and hence an HHG), but does not act geometrically on a CAT(0) space?
\end{question}

Corollary \ref{cor:intro-consequences}.\eqref{item:asymp-cat} says that unbranched mapping tori act geometrically on asymptotically CAT(0) spaces, so the non-HHG CAT(0) examples show that the converse of that statement cannot hold (since CAT(0) spaces are asymptotically CAT(0)), while an answer to Question \ref{question:non-cat0-hhg} would determine whether or not Corollary \ref{cor:intro-consequences} can be strengthened by replacing asympotically CAT(0) spaces with CAT(0) spaces.

Recall that Corollary \ref{cor:intro-consequences} establishes semihyperbolicity for unbranched $\maptor$ in two independent ways, one of which is via the existence of a proper cocompact action on a proper \emph{coarsely injective space}, and hence a metrically proper cobounded action on a (not necessarily proper) injective space, using results of Haettel-Hoda-Petyt \cite{haettel_coarse_2020}.  Harry Petyt has informed us in personal communication that there are free by cyclic groups that are not coarsely injective, which gives a sense in which item \eqref{item:injective-space} is sharp.  

Stronger than coarse injectivity is the existence of a proper, cocompact action on a locally finite \emph{Helly graph}; the study of groups with this property was initiated in \cite{CCGHO:Helly}, where it is shown that groups admitting such an action enjoy strong forms of some of the properties discussed above: they act properly and cocompactly on proper injective spaces (strengthening coarse injectivity), and they are biautomatic (strengthening semihyperbolicity).  

\begin{question}\label{question:coarsely-helly}
For which free by cyclic $\maptor$ is there a proper cocompact action on a coarsely injective space?  On a Helly graph?  On a proper injective space?  Are there non-CAT(0) examples with such actions?
\end{question}

Interestingly, some of the aforementioned non-coarsely injective examples (which will appear in forthcoming work by Haettel-Hoda-Petyt) are CAT(0).  They also have examples that are Helly (hence biautomatic) and CAT(0), but are not coarse median (hence they do not satisfy any of the conditions in Theorem \ref{thm:main}).       

\tocless{\subsection{Plan of the paper}}{}
In Section \ref{sec:P} we give all the definitions we will need in the paper. 
Section \ref{sec:unbranched-blocks} elaborates on our algebraic obstruction \emph{unbranched blocks}, and Section \ref{sec:free-by-Z} details the graphs of groups decompositions  corresponding to free by cyclic groups on which our proofs based. Proof of Theorem \ref{lem:unbranched-implies-one-Nielsen} is also included in this section. 
Section \ref{sec:coarse-median-implies-unbranched} details how to go from coarse median to unbranched, and arguments in Section \ref{sec:build-HHG} prove how unbranched blocks imply hierarchical hyperbolic structure.  Section \ref{subsec:excessive-linearity} deals with excessive linearity and CT representatives.  Finally, in Section \ref{sec:proof} we put together lemmata from earlier sections to prove our remaining theorems.  Section \ref{sec:exs} contains examples. 

\tocless{\subsection{Acknowledgments}}{} We thank Zachary Munro and Harry Petyt for helpful discussions about \cite{MunroPetyt} and comments on an earlier version, Alessandro Sisto for a question motivating Proposition \ref{prop:remove-virtual}, Rob Kropholler for an informative discussion about CAT(0) structures, and Harry Petyt for explaining the (non)-coarsely injective examples. We are grateful to Jack Button for a correction in Example \ref{exmp:not-via-specialness}, Giorgio Mangioni for a discussion related to the linear case of Proposition \ref{prop:remove-virtual}, to Naomi Andrew, Monika Kudlinska, and Armando Martino for extremely helpful clarifications on the superlinear case of the same proposition (see Claim \ref{claim:naomi-monika-jp}, whose proof was explained to us by Naomi Andrew).  We finally thank Matt Durham and an anonymous referee for various helpful comments.  

\section{Preliminaries}\label{sec:P}

\subsection{Relative train track maps and CT maps}\label{subsec:IRTT}

We now summarise the parts of the theory of improved relative train track maps that we will use; see \cite{BFH-00,BFH:kolchin} for more details.

  Given a marked graph $\mathbb G$, a \textit{filtration}  of $\mathbb G$ is a strictly increasing sequence $\mathbb G_0 \subset \mathbb G_1 \subset \cdots \subset \mathbb G_k = \mathbb G$ of subgraphs $\mathbb G_r$ with no isolated vertices. The filtration is \textit{$f$-invariant} if $f(\mathbb G_r) \subset \mathbb G_r$ for all $r$, where $f:\mathbb G\to\mathbb G$ is a homotopy equivalence sending vertices to vertices and edges to paths.
 The \textit{$r^{\text{th}}$ stratum} is the subgraph $H_r$ consisting of $\mathbb G_r \setminus \mathbb G_{r-1}$ together with the  endpoints of each edge not in $\mathbb G_{r-1}$.
 
If $\lambda=1$ then $H_r$ is a \textit{nonexponentially growing (NEG) stratum}  whereas if $\lambda>1$ we say that $H_r$ is an \textit{exponentially growing (EG) stratum}, where $\lambda $ is the Perron-Frobenious eigenvalue associated to the stratum \cite{BestvinaHandel:ttm}.	

An automorphism $\phi\in\out$ can be represented by a  homotopy equivalence $f:\mathbb G\rightarrow \mathbb G$ that takes vertices to vertices and edges to edge-paths of a marked graph $\mathbb G$ with marking $\rho: R_n \rightarrow \mathbb G$, called a \textit{topological representative}. Topological representative $f$ preserve the marking of $\mathbb G$, in other words $\overline{\rho}\circ f \circ \rho: R_n \rightarrow R_n$ represents $R_n$. A nontrivial path $\alpha$  in G is a \emph{periodic Nielsen path} if $f^k_{\#}(\alpha)=\alpha$ for some $k$, where the smallest such  $k$ is called the \emph{period}.
When $k=1$, we simply call $\alpha$ a \emph{Nielsen path}; if in addition $\alpha$ is a (nontrivial) immersed cycle, we call $\alpha$ a \emph{Nielsen cycle}.

\subsubsection{Improved relative train track map (IRTT)} \label{IRTT}
Relative train track maps are topological representatives whose exponentially-growing strata satisfy certain extra conditions \cite{BestvinaHandel:ttm}. Since our automorphisms will also have non-exponentially-growing (NEG) strata, we will also add conditions on the zero strata and on the NEG strata. This will result in a topological representative $f$, called an \emph{improved relative train track (IRTT)}, satisfying Theorem \ref{thm:RTT}.

If $\phi \in \out$ is polynomially growing and  has unipotent image in $GL_n(\mathbb Z)$ (in other words $\phi$ induces a unipotent action on
on $H_1(\F, \mathbb Z)$); we call $\phi$  \emph{unipotent polynomially growing (UPG)}. 

\begin{thm} [Theorem 5.1.5 \cite{BFH-00}]\label{thm:RTT}
For every UPG outer automorphism
$\phi\in \out$ there is a  relative train track map $f:\mathbb G\rightarrow \mathbb G$ with a filtration $\mathbb G_0 \subset \mathbb G_1 \subset \cdots \subset \mathbb G_k = \mathbb G$ such that $f$ represents  $\phi$ and  $f$ has the following properties.
\begin{itemize}
\item Filtration gives a $\phi$--invariant factor system.
 \item Every periodic Nielsen path has period one.
\item Each $\overline{\mathbb G_i\setminus \mathbb G_{i-1}} $ is a single edge $E_i$ satisfying $f(E_i) = E_i \cdot u_i $ for some
closed (immersed) path $u_i$ (called \emph{suffix} of $E_i$) whose edges  are in $\mathbb G_{i-1}$. In other words, for all $k>0$,
\[ [f^k(E_i)] = E_i \cdot u_i \cdot  [f(u_i)]\cdots [f^{k-1}(u_i)]. \]
   
    \item For every vertex $v\in \mathbb G$, $f(v)$ is a fixed point. 
\end{itemize}
\end{thm}

The stratum $H_i=E_i$ is a \emph{linear stratum} if $u_i$ represents a conjugacy class in $\pi_1\mathbb G$ that is also represented by a Nielsen path. If $H_i=E_i$ is a linear stratum, and if  $u_i $ is a non-zero power of a cyclic permutation of  a Nielsen cycle $u$, we say that Nielsen cycle $u$ \emph{supports} the linear stratum $E_i$.  (We emphasize that a \emph{cycle} in $\mathbb G$ is a locally injective map $S\to \mathbb G$ of graphs, where $S$ is a graph homeomorphic to $S^1$.  In many papers on train tracks, this is called a \emph{circuit}; we are using the terminology from \cite{MunroPetyt}.)
\subsubsection{Completely split improved relative train track (CT) maps}
Further improvements of IRTT maps lead to the construction of \emph{completely split improved relative train track (CT)} maps. Below we summarize their properties that apply to maps representing UPG automorphisms. For the  full list of  properties we refer the reader to [Definition 4.7 \cite{FeighHandel:recognition}].   
\begin{lem} \cite[Definition 4.11, Lemma 4.42,Theorem 4.28]{FeighHandel:recognition}
There exists $k\geq 1$ depending only on $\F$, so that given any $\phi\in \, UPG(\F)$, there exists an IRTT map $f:\mathbb G\rightarrow \mathbb G$ representing $\phi^k$ (called a CT-map) which satisfies the following additional properties:
\begin{itemize}
\item(\textbf{Periodicity}) Each periodic edge and each periodic vertex is fixed.
\item(\textbf{Linear Edges}) For each linear edge $E_i$ there is a closed root-free Nielsen path $w_i$ such that $f(E_i)=E_iw^{d_i}_i$ for some $d_i\neq 0$. If  $E_i$ and $E_j$ are distinct linear
edges with the same axes then $w_i =w_j$ and $d_i\neq d_j$.
\item (\textbf{NEG Nielsen paths}) If the edges in an indivisible Nielsen path $\sigma$ that are in the highest stratum belong
to an NEG stratum then there is a linear edge $E_i$ with $w_i$ as in item ({\bf Linear Edges})
and $k\neq 0$ such that $\sigma=E_iw^k_i\overline {E_i}$.\qedhere
\end{itemize}
\end{lem} 
The following lemma summarizes an additional crucial property of CT--maps that we use repeatedly in this paper:

\begin{lem}[Lemma 4.21 \cite{FeighHandel:recognition}]\label{lem:suffix} Let $f: \mathbb G \rightarrow \mathbb G$ be a CT map. Then every NEG stratum $H_i$ is a single edge $E_i$. If $E_i$ is not contained in $Fix(f)$ then there is a non-trivial closed path $u_i\subset \mathbb G_i$ such
that  $f(E_i) = E_i\cdot u_i$. Moreover $u_i$ forms a cycle and the turn $(u_i, \overline{u_i})$ is legal.
\end{lem}

\subsection{Coarse geometry and coarse median spaces}

\subsubsection{Acylindrical action}
An acylindrical action on a space is one where the number of elements that simultaneously coarsely stabilize sufficiently distant points is uniformly bounded.

\begin{defn}[\cite{BOWtight}]\label{defn:acyl} The action of a group $G$ on a hyperbolic graph $X$ is \emph{acylindrical} if for all $r\geq 0$ there
exist $R, N >0$ such that for all vertices $a,b\in X$ with $d(a, b) \geq R $, there are at
most $N$ distinct $g\in G$  such that $d(a, ga)\leq  r$  and $d(b, gb) \leq r$.

An action is $k$-acylindrical if, for every pair of vertices that are at distance at
least $k+1$ the stabilizer of the pair is finite.
\end{defn}

We will use acylindricity in two contexts.  First, when $G$ is a relatively hyperbolic group and $\widehat G$ is the coned-off Cayley graph, the left-multiplication action of $G$ on $\widehat G$ is acylindrical.  Second, when $X$ is a tree, the two notions of acylindricity in Definition \ref{defn:acyl} are equivalent, and we will work with the latter one.

\subsubsection{Coarse median spaces}
Defined in \cite{Bowditch:coarse-median}, coarse medians are motivated by a \emph{centroid construction} capturing certain aspects of the large-scale “cubical”
structure of various naturally occurring spaces. 
A coarse median space
is a geodesic metric space equipped with a ternary “coarse median” operation,
defined up to bounded distance, and satisfying some natural axioms. Roughly
speaking, these require that any finite subset of the space can be embedded in
a finite CAT(0) cube complex in such a way that the coarse median operation
agrees, up to bounded distance, with the natural combinatorial median in such
a complex.

\begin{defn}[Median algebra]\label{defn:median-algebra}
Let $M$ be a set and $\mu:M^3\rightarrow M$ a ternary operation.  $(M, \mu)$ is a median algebra, if, for any $x,y,z,u, t \in M$,
\begin{itemize}
\item $\mu(x,y,z)=\mu(y,x,z)=\mu(y,z,x)$
\item $ \mu(x,x,y)=x$
\item $\mu(x,y, \mu(z,u,t))=\mu(\mu(x,y,z), \mu(x,y,u)), t)$\qedhere
\end{itemize}
\end{defn}
Intuitively, we
think of $\mu$ mapping points $x,y,z\in M$ in $M$ to a point “between $a,b$'' for each pair of distinct $a,b\in\{x,y,z\}$. For $x,y \in M$ the \emph{interval between $x$ and $y$} is, $[x,y]_{\mu}=\{s\in M: \mu(x,y,s)=s\}$ and we write $[x,y]$ if the choice of $\mu $ is clear. 

\begin{defn} A  median metric space is any metric space $(X, d)$
such that for all $x,y,z\in X$, there is a unique $m$ such that $d(x,y)=d(x,m)+d(m,y)$ and the analogous condition is also satisfied for $d(x,z)$ and $d(y,z)$.  Setting $\mu(x,y,z)=m$ makes $(X,\mu)$ a median algebra.
\end{defn}

\begin{defn} Let $(X,d_X)$ $(Y, d_Y)$ be metric spaces and $\mu_X, \mu_Y$ be ternary operators on $X,Y$ respectively. Then, a map $\rho: X\rightarrow Y$ is called $k$--quasi-median if, for all $x,y,z$ in $X$, 
$d_Y(\rho(\mu_X(x,y,z)), \mu_Y(\rho(x), \rho(y), \rho(z)))\leq k.$
\end{defn}

\begin{defn}  A coarse median space $(X,d)$ is a metric space with a \emph{coarse median}, which is a ternary operator $\mu: X\rightarrow X$ such that there is a sequence $\{s_n\}$ satisfying the following:
\begin{itemize}
\item $\mu $ is $s_0$--coarsely Lipschitz in each factor,
\item for each finite $A\subset X$ there is a finite median algebra $M$ and maps $\iota: M\rightarrow X$, $j: A\rightarrow M$ that are $s_{|A|}$--quasi-median, and $d_X(ij(a), a) \leq s_{|A|}$ for all $a\in A$. \qedhere
\end{itemize}
\end{defn}

The existence of a coarse median on a geodesic space is a quasi-isometry invariant.  Gromov hyperbolic spaces, mapping class groups and Teichm\"uller
spaces of compact surfaces, right-angled Artin groups and geometrically finite
Kleinian groups in any dimension are coarse median. Hierarchically hyperbolic spaces are also coarse median (\cite{HHSI,HHSII}). The notion is useful for establishing coarse rank bounds and quasi-isometric rigidity.

\begin{defn} A finitely generated group is \emph{coarse median} if some (hence any)
Cayley graph with respect to a finite generating set admits a coarse median.  A \emph{invariant coarse median} on a finitely generated group $G$ is a ternary operator $\mu:G^3\to G$ that is $G$--equivariant for the left-multiplication action, and also a coarse median.
\end{defn}

Not every coarse median group admits an equivariant coarse median, but every hierarchically hyperbolic group does.  In Theorem \ref{thm:main}.\eqref{item:coarse-median}, we mean the weaker (non-invariant) notion.  Hence the theorem says that (for free by cyclic groups), either there is an invariant coarse median, or there is no coarse median at all.

\begin{remark}\label{rem:quasicubical-coarse-median}
If a space $X$ is quasi-isometric to a CAT(0) cube complex, then, immediately from the definition, it is a coarse median space.    
\end{remark}

\subsection{Hierarchical hyperbolicity}
Heuristically, a hierarchically hyperbolic space structure is a collection of hyperbolic spaces forming a ``coordinate system'' for a larger space.
The notion originates in \cite{HHSI} and was given a simpler equivalent formulation in \cite{HHSII}. 
\begin{defn}[Hierarchically hyperbolic space structure]\label{defn:HHS}
Let $\hhsconst \geq 0$, and let $(\mathcal{X},d_\mathcal{X})$ be a $(\hhsconst,\hhsconst)$--quasi-geodesic space. A \emph{hierarchically hyperbolic space structure (HHS structure)} on $\mathcal{X}$ consists of
\begin{itemize}
    \item an index set $\mathfrak{S}$ for a collection of \emph{domains} $W$;
    \item a set of hyperbolic spaces $\left\{ (\calC(W), d_W) :  W \in \mathfrak{S} \right\}$;
    \item a set of projection maps $\left\{ \pi_W : \mathcal{X} \rightarrow 2^{\calC(W)} \right\}$; and
    \item nesting, orthogonality, and transversality relations so that each pair of domains is related by exactly one of these relations, 
\end{itemize}
satisfying the following axioms. Let $x,y \in \mathcal{X}$ and $U,V,W \in \mathfrak{S}$.
\begin{itemize}
     \item[(P)] Projection.
         \begin{itemize}
         \item[P1] \emph{Coarsely Lipschitz:} Each projection $\pi_W$ is $(\hhsconst,\hhsconst)$--coarsely Lipschitz.
         \item[P2] \emph{Coarsely well defined:} The image under the projection $\pi_W(x)$ is nonempty and has diameter bounded by $\hhsconst$. 
         \item[P3] \emph{Coarsely surjective:} The image of $\pi_W(\mathcal{X})$ is $\hhsconst$--dense in $\calC(W)$.
         \item[P4] \emph{Uniqueness:} For each $r \geq 0$ there is a $\theta$ (depending only on $r$) such that if $d_\mathcal{X} (x,y) \geq \theta$, then there is some $W \in \mathfrak{S}$ for which $d_W (\pi_W(x), \pi_W(y) ) \geq r$.
         \end{itemize}

     \item[(N)] Nesting. The index set $\mathfrak{S}$ is partially ordered by the nesting relation $\nest$. If $V \nest W$, say that $V$ is \emph{nested} in $W$. Nesting is reflexive, i.e., $V \nest V$. Denote $\mathfrak{S}_W^\nest := \left\{ V \in \mathfrak{S} : V \nest W \right\}$.
         \begin{itemize}
         \item[N1] \emph{Maximum:} The index set $\mathfrak{S}$ is either $\emptyset$ or contains a unique $\nest$--maximal element.
         \item[N2] \emph{Finite height:} Any set of pairwise--$\nest$--comparable elements has size at most $\hhsconst$.
         \item[N3] \emph{Relative projection:} For all $V,W \in \mathfrak{S}$ with $V \sqsubsetneq W$ there is a specified nonempty subset $\rho_W^V \subseteq \calC(W)$ with diameter at most $\hhsconst$ and a map $\rho_V^W: \calC(W) \rightarrow 2^{\calC(V)}$.
         \item[N4] \emph{Bounded geodesic image:} If $V \sqsubsetneq W$, then for all geodesics $\gamma \subset \calC(W)$ either $\rho_V^W(\gamma)$ has diameter at most $\hhsconst$ or $\gamma$ intersects the $\hhsconst$-neighborhood of $\rho_W^V$.
         \item[N5] \emph{Large links:} Let $N$ be the largest integer no greater than $E d_W(\pi_W(x),\pi_W(y))+E$. There exists $\left\{ V_1, \dots, V_N \right\}$ with $V_i \sqsubsetneq W$ such that, for all $U \sqsubsetneq W$, either $U \nest V_i$ for some $i$ or $d_U(\pi_U(x),\pi_U(y)) \leq E$.
          \end{itemize}

     \item[(O)] Orthogonality. The orthogonality relation $\orth$ is symmetric and antireflexive. Denote $\mathfrak{S}_W^\orth := \left\{ V \in \mathfrak{S} : V \orth W \right\} $.
        \begin{itemize}
        \item[O1] \emph{$\nest$--compatibility:} If $V \nest W$ and $U \orth W$, then $V \orth U$.
        \item[O2] \emph{Containers:} For $U \nest W$ with $\mathfrak{S}_W^\nest \cap \mathfrak{S}_U^\orth \neq \emptyset$, there exists $Q \in \mathfrak{S}_W^\nest$, called the \emph{container} of $U$ in $W$, such that $V \nest Q$ whenever $V \in \mathfrak{S}_W^\nest \cap \mathfrak{S}_U^\orth$. 
        \end{itemize}
        
    \item[(T)] Transversality. The transversality relation $\transverse$ is symmetric and antireflexive.
        \begin{itemize}
        \item[T1] \emph{Relative projections:} If $V \transverse W$, then there are specified nonempty sets $\rho_W^V \subseteq \calC(W)$ and $\rho_V^W \subseteq \calC(V)$ each of diameter at most $\hhsconst$.
        \end{itemize}

    \item[(C)] Consistency.  
        \begin{itemize}
        \item[C1] If $V \transverse W$, then $\mathrm{min}\left\{ d_W(\pi_W(x),\rho_W^V), d_V(\pi_V(x), \rho_V^W) \right\}\leq E$.
        \item[C2] If $V \sqsubsetneq W$, then $\mathrm{min}\left\{ d_W(\pi_W(x),\rho_V^W), \mathrm{diam}(d_V(\pi_V(x) \cup \rho_W^V(\pi_V(x)) \right\} \leq E$.
        \item[C3] If $U \nest V$ and either $V \sqsubsetneq W$ or $V \transverse W$ and $W \not \orth U$, then $d_W(\rho_W^V,\rho_W^U)\leq E$. 
        \end{itemize}

    \item[(R)] Partial realization. For any set of pairwise orthogonal domains $\left\{ V_j \right\}$ and any points $p_j \in \calC(V_j)$, there exists $x \in \mathcal{X}$ such that $d_{V_j} (x,p_j) \leq E$ for all $j$; and for each $j$ and each $W \in \mathfrak{S}$ with $V_j \sqsubsetneq W$ or $V_j \transverse W$, $d_W(x, \rho_W^{V_j}) \leq E$.\qedhere
\end{itemize}
\end{defn}

\begin{remark}
In \cite[Defn. 1.1]{HHSII}, it is only hypothesised that each $\pi_U$ has uniformly quasiconvex image, for convenience in describing HHS structures on certain subspaces.  However, in \cite[Sec. 1]{HHSII}, it is explained that one can always modify the $\mathcal CU$ to make the $\pi_U$ uniformly coarsely surjective, so Definition \ref{defn:HHS} identifies the exact same class of spaces as \cite[Defn. 1.1]{HHSII}, and we choose work with the more convenient coarsely surjective $\pi_U$ because in this paper there is no reason not to.
\end{remark}

The notion of a hierarchically hyperbolic group also appears in \cite{HHSI,HHSII}. The following formulation comes from \cite{PetytSpriano}; it is equivalent to the original by \cite{DHScorr}.

\begin{defn}[Hierarchically hyperbolic group]\label{defn:HHG}
A finitely generated group $G$ is called a \emph{hierarchically hyperbolic group (HHG)} if there exists a hierarchically hyperbolic structure $(G,\mathfrak{S})$ satisfying the following assumptions.
\begin{itemize}
    \item $G$ acts cofinitely on $\mathfrak{S}$, and this action preserves the three relations.
    \item For each $g \in G$ and each $U \in \mathfrak{S}$ there is an isometry $g: \mathcal{C}(U) \rightarrow \mathcal{C}(gU)$, and for all $h \in G$ these isometries satisfy $g \cdot h = gh$.
    \item For all $x, g \in G$ and $U \in \mathfrak{S}$ we have $g \pi_U(x) = \pi_{gU}(gx)$. If in addition $V \in \mathfrak{S}$ and $U \transverse V$ or $V \sqsubsetneq U$, then $g \rho_U^V = \rho_{gU}^{gV}$.
\end{itemize}
The notation $(G, \mathfrak{S})$ is meant to convey that $G$ plays the role of the quasigeodesic metric space $\mathcal{X}$ in the definition of a hierarchically hyperbolic space structure (Definition \ref{defn:HHS}): We take $G$ to be the discrete metric space equipped with the word metric.\qedhere
\end{defn}

\begin{defn}\label{defn:colourable}
$(G,\mathfrak S)$ is \emph{colorable} if $gU\transverse U$ or $gU=U$ for all $g\in G$ and $U\in\mathfrak S$.    
\end{defn}

Since $G$--translates of $U\in\mathfrak S$ cannot be properly $\nest$--related ($\nest$ is a partial order with bounded chains), colorability means exactly that $\mathfrak S$ can be partitioned into finitely many $G$--invariant colors such that elements of the same color are not orthogonal.

\begin{defn}[Hierarchical quasiconvexity, Definition 5.1, \cite{HHSII}]
    Let $(\mathcal X, \mathfrak S)$ be a hierarchically hyperbolic space. Then for some $k:[0,\infty)\rightarrow [0,\infty)$ $\mathcal Y\subset \mathcal X$ is $k$--\emph{hierarchically quasiconvex} (HQC) if:
    \begin{enumerate}
    \item For all $W\in \mathfrak S$, $\pi_W(\mathcal Y)$ is $k(0)$–-quasiconvex in the $\delta$--hyperbolic space $\mathcal C W$,
    \item For every $\kappa>0 $ and every point $x\in \mathcal X$ such that 
    $d_{W}(\pi_W(x), \pi_W(\mathcal Y))\leq \kappa $
    for all $W \in \mathfrak S$, we have $d(x, \mathcal Y)\leq k(\kappa)$.\qedhere
    \end{enumerate}
\end{defn}

\begin{remark}[HHS coarse medians and HQC sets]\label{rem:coarse-median}
If $(X,\mathfrak S)$ is an HHS, then \cite[Thm. 7.3]{HHSII} provides a coarse median $\mu:X^3\to X$.  Specifically, given $x,y,z\in X$, there is a coarsely unique point $\mu=\mu(x,y,z)$ with the property that $\pi_W(\mu)$ is $C$--close to the coarse centre of any geodesic triangle in the $E$--hyperbolic space $\mathcal CW$ with vertices $\pi_W(x),\pi_W(y),\pi_W(z)$, where $C=C(E)$.  This yields an equivalent formulation of hierarchical quasiconvexity that resembles the notion of convexity in a median space: $Y\subset X$ is HQC if there exists $Q$ such that $d(\mu(x,y,y'),Y)\leq Q$ for all $x\in X$ and $y,y'\in Y$; see \cite[Sec. 5]{RST}.

HHSs in fact enjoy a stronger form of the coarse median property given by the \emph{cubulation of hulls} from \cite[Thm. 2.1]{HHS:quasiflats}, along with various even stronger forms from \cite{DMS,DMS2,Durham:cubes} in the colorable case.  A more global form of the coarse median property, and the one used in this paper, is \cite[Thm. B]{Pet}, which says that any colorable HHG admits a quasi-median quasi-isometry to a finite-dimensional CAT(0) cube complex.
\end{remark}

\section{Unbranched blocks}\label{sec:unbranched-blocks}
\begin{defn}[Block, unbranched]\label{defn:block-unbranched}
Let $G$ be a group.  A subgroup $B\leq G$ is a \emph{block} if both of the following hold:
\begin{itemize}
    \item $B$ has a finite-index subgroup isomorphic to $\F\times\integers$, where $\F$ is a nonabelian free group.

    \item If $B'\leq G$ has a finite-index subgroup $\F'\times \integers$, where $\F'$ is nonabelian free, and $[B:B'\cap B]<\infty$, then $[B':B'\cap B]<\infty$. 
\end{itemize}

The group $G$ is \emph{unbranched}, or \emph{has unbranched blocks}, if for all blocks $B_1,B_2,B_3$ such that $[B_i:B_i\cap B_j]=\infty$ whenever $i\neq j$, the subgroup $B_1\cap B_2\cap B_3$ is virtually cyclic (possibly finite).
\end{defn}

  Observe:

\begin{lem}\label{lem:unbranched-index}
Let $G'\leq G$. If $[G:G']<\infty$, then $G'$ is unbranched if and only if $G$ is.
\end{lem}

\begin{proof}
Suppose $G'\leq G$ has finite index and let $B_1,B_2,B_3$ be blocks in $G$.  For $i\in\{1,2,3\}$, let $B_i'=G'\cap B_i$. Observe that $B_i'$ is a block in $G$.  Indeed, suppose that $B_i''\leq G'$ is virtually the product of $\integers$ with a nonabelian free group, and $B_i'$ is virtually contained in $B_i''$.  Then, viewing $B_i''$ as a subgroup of $G$, the fact that $B_i'$ has finite index in $B_i$, and $B_i$ is a block in $G$, together imply (using Definition \ref{defn:block-unbranched}) that $[B_i'':B_i'\cap B_i'']<\infty$.

This shows that $B_i$ is a block in $G$ if and only if $B_i'=B_i\cap G'$ is a block in $G'$.  Moreover, each of the intersections $B_i\cap B_j$ is commensurable with $B_i'\cap B_j'$, and $B_1\cap B_2\cap B_3$ is commensurable with $B_1'\cap B_2'\cap B_3'$.  Hence $G$ is unbranched if and only if $G'$ is.
\end{proof}

\begin{prop}\label{prop:unbranched-cyclic-splitting}
Let $G$ be a group splitting as a graph of groups such that each edge group is virtually abelian and the $G$--action on the Bass-Serre tree is acylindrical.  Then every block in $G$ fixes a unique vertex in $T$.  If every edge group is virtually cyclic, then $G$ is unbranched if and only if every vertex group is unbranched.
\end{prop}

\begin{proof}
Let $T$ be the Bass-Serre tree and let $B$ be a block in $G$.  Since $G$ acts acylindrically on $T$, Definition \ref{defn:block-unbranched} implies that $B$ must fix a point.  The assumption on edge-stabilizers means $B$ cannot fix an edge, so there is a unique vertex $v(B)\in T$ with $B\leq \stabilizer_G(v(B))$.

Now suppose edge stabilizers are virtually cyclic.  If $B_1,B_2$ are blocks and $B_1\cap B_2$ is non-(virtually-cyclic), then $v(B_1)=v(B_2)$, since $B_1\cap B_2$ fixes the geodesic from $v(B_1)$ to $v(B_2)$.  In particular, if $B_1,B_2,B_3$ are blocks witnessing that $G$ is not unbranched, then $v(B_1)=v(B_2)=v(B_3)$, so the corresponding vertex group $G_v$ contains $B_1,B_2,$ and $B_3$, and these subgroups are blocks in $G_v$ whose mutual intersection is not virtually cyclic.
\end{proof}

\begin{prop}\label{prop:relhyp-unbranched}
Let $G$ be a finitely generated group that is hyperbolic relative to a finite collection $\mathcal P$ of subgroups.  Then $G$ is unbranched if and only if each $P\in\mathcal P$ is unbranched.  If each $P\in\mathcal P$ is coarse median, then so is $G$.
\end{prop}

\begin{proof}
Let $\widehat G$ be a coned-off Cayley graph witnessing that $(G,\mathcal P)$ is a relatively hyperbolic pair.  Since $G$ acts on $\widehat G$ acylindrically \cite[Prop. 5.2]{Osin:acyl}, any block $B\leq G$ has bounded orbits, and it follows from relative hyperbolicity that it is conjugate into a unique $P\in\mathcal P$.  Since $\mathcal P$ is almost malnormal, distinct blocks have infinite intersection only if they belong to a common peripheral subgroup, so $G$ is unbranched exactly when each $P\in\mathcal P$ is.  If each $P\in\mathcal P$ is coarse median, then so is $G$, by \cite[Thm. 1.1]{Bowditch:invariance}.
\end{proof}

\begin{remark}\label{rem:blocks-in-graph-of-groups}
Proposition \ref{prop:unbranched-cyclic-splitting} says that each block lies in a unique vertex stabilizer.  The proof of Proposition \ref{prop:relhyp-unbranched} shows that, in a relatively hyperbolic group, any block is conjugate into a unique peripheral subgroup.
\end{remark}

\section{Decomposing mapping tori of free by cyclic groups}\label{sec:free-by-Z}
Let $\F,\phi,$ and $\maptor$ be as in Theorem \ref{thm:main}. The \emph{growth rate} of $\phi$ is 
\emph{polynomial of degree $d$} if, for all $g\in \F$, there exist 
$d'\in\{0,\ldots,d\}$ and $A,B>0$ such that $An^{d'}\leq \frac{\ell_S(\phi^n(g))}{\ell_S(g)}\leq Bn^{d'}$,
and $d$ is as small as possible.  If $d=1$, then $\phi$ has \emph{linear growth}, and if $d>1$, then $\phi$ has 
\emph{superlinear growth}.  If $\phi$ does not have polynomial growth, then $\phi$ is exponentially growing.  Every polynomially growing automorphism has a nonzero power that is UPG \cite[Cor. 5.7.6]{BFH-00}.  Fixing a representative $\Phi\in\Aut(\F)$ of $\phi$, the group $\maptor\cong\langle \F,t\mid \{tft^{-1}=\Phi(f):f\in \F\}\rangle$.   

We will use several structural results about mapping tori. The first is from \cite[Thm. 3.5]{DahmaniLi} or \cite[Thm. 3.15]{Ghosh}:

\begin{prop}[Relative hyperbolicity]\label{prop:relhyper}
There exists a finite collection $\{F_i\}_{i=1}^k$ of finite-rank subgroups of $\F$, a collection of elements $\{g_i\}_{i=1}^k\subset \F$, and a collection of positive integers $\{m_i\}_{i=1}^k$ such that the following holds.  Let $t_i=(tg_i)^{m_i}$ and let $\maptor_i=\langle F_i,t_i\rangle$.  Then: 
\begin{itemize}
    \item $\maptor$ is hyperbolic relative to the set $\{\maptor_i\}_{i=1}^k$ of peripheral subgroups.
    \item For each $i$, the subgroup $F_i$ is normal in $\maptor_i$.
    \item Letting $\Phi_i\in\Aut(F_i)$ be the automorphism induced by conjugation by $t_i$, the maps $F_i\hookrightarrow \maptor_i$ and $t_i\mapsto t_i$ induce an isomorphism $F_i\rtimes_{\phi_i}\langle t_i\rangle\to \maptor_i$.
\end{itemize}
Moreover, each $\phi_i$ represents a polynomial-growth element of $\Out(F_i)$.
\end{prop}

\begin{lem}\label{lem:separability}
Let $\{\maptor_i\}_{i=1}^k$ be the peripheral structure for $\maptor$ from Proposition \ref{prop:relhyper} and let $\maptor_i'\leq \maptor_i$ be a finite-index subgroup, for each $i$.  Then there exists a finite-index normal subgroup $\maptor'\leq \maptor$ such that $\maptor'\cap \maptor_i\leq \maptor_i'$ for all $i$.
\end{lem}

\begin{proof}
Fix $i\leq k$.  By Proposition \ref{prop:relhyper}, $\maptor_i$ is free by cyclic, so $\maptor_i'$ is also, and hence Proposition 4.2.1 of \cite{Kudlinska:thesis} implies that $\maptor_i'$ is separable in $\maptor$, which provides $\maptor(i)\leq_{fi}\maptor$ such that $\maptor(i)\cap \maptor_i\leq \maptor_i'$, so $\maptor'=\bigcap_{i=1}^k\maptor(i)$ is the required subgroup.
\end{proof}

Next, we recall the splittings used to study mapping tori of polynomial-growth automorphisms.  The first statement is for  superlinear growth, and comes from the topmost edges decomposition originating in work of Macura \cite{Macura:detour}; see also \cite{Hagen:thickness} and \cite[Prop. 2.5]{AndrewHughesKudlinska}.  Applying this splitting theorem recursively yields:

\begin{prop}[Cyclic hierarchy]\label{prop:cyclic-hierarchy}
Suppose that $\phi$ has polynomial growth rate $d\geq 2$.  Then $\maptor$ has a finite-index subgroup of the form $\maptor'=\F\rtimes_{\phi^k}\integers,\ k>0$, admitting a hierarchy with the following properties:
\begin{itemize}
    \item each graph of groups decomposition in the hierarchy has finite underlying graph;
    \item each graph of groups in the hierarchy has all edge groups isomorphic to $\integers$;
    \item the terminal vertex groups in the hierarchy are mapping tori of free group automorphisms of polynomial growth rate at most $1$, and each subgroup of $\maptor$ that is the mapping torus of a linear-growth automorphism is contained in one of the vertex groups.
    \item Each splitting in the hierarchy is obtained by cutting all of the topmost edges in the domain of an improved relative train track representative for the monodromy of the mapping torus being split (see e.g. \cite[Lem. 5.2]{KudlinskaValiunas}).
\end{itemize}
Moreover, each graph of groups decomposition in the hierarchy is $2$--acylindrical.
\end{prop}

A graph of groups decomposition of a group $G$ is \emph{$k$--acylindrical} if $G$ acts $k$--acylindrically on the Bass-Serre tree.  The acylindricity statement the proposition is \cite[Lem. 5.2]{KudlinskaValiunas}.  

In the linear case, we use a splitting over $\integers^2$ subgroups constructed in \cite{AndrewMartino:splitting,DahmaniTouikan} whose properties are summarized in the following proposition, which is almost exactly the same as \cite[Prop. 4.11]{HagenWise:new-polynomial} but which we restate here for the reader's convenience:

\begin{prop}[Linear mapping torus decomposition]\label{prop:split-over-tori}
Let $\phi$ be a linearly growing \emph{UPG} outer automorphism of the finite rank free group $\F$ and let $\maptor=\F\rtimes_\phi\langle t\rangle$ be its mapping torus. Then there is a finite connected graph $\Delta$ such that $\F$ and $\maptor$ split as graphs of groups with underlying graph $\Delta$, and the following all hold:
\begin{enumerate}
\item For each $v\in\vertices(\Delta)$, the vertex group $\F_v$ has finite rank, and $\maptor_v=\F_v\times \langle t_v\rangle$, where $t_v\in \F t$.  If $\rank(\F_v)>1$, then $\F_v\times\langle t_v\rangle$ is a block of $\maptor$.  If $\rank(\F_v)=1$, then $\F_v\times\langle t_v\rangle$ is a maximal $\integers^2$ subgroup of $\maptor$ .\label{item:linear-vertices}

\item For $e\in\edges(\Delta)$,  $\F_e=\langle r_e\rangle$ is a maximal cyclic subgroup of $\F$ and $\maptor_e=\F_e \times\langle t_e\rangle$.  Let $r_e$ also denote the image of $r_e$ in $\F_{e^-}$ and let $p_e$ denote its image in $\F_{e^+}$.  The image of $t_e$ in $\maptor_{e^-}$ is $t_{e^-}$ and the image of $t_e$ in $\maptor_{e^+}$ is $p_e^{k(e)}t_{e^+}$, where $k(e)\in\integers$.\label{item:linear-edges}

\item If $e\in\edges(\Delta)$ and $v\in\vertices(\Delta)$ is incident to $e$, then $Z(\maptor_v)\leq\image(\maptor_e\to \maptor_v)$.\label{item:linear-centre}  

\item The outer automorphism $\phi$ has a representative preserving the graph of groups decomposition $\Delta$ of $\F$, acting as a multitwist automorphism as in \cite[Sec. 6]{CohenLustig}.\label{item:linear-dehn}

 \item $\Delta$ is bipartite, and $v\in\vertices(\Delta)$ is \textbf{black} if $\F_v$ has rank $1$ and \textbf{white} otherwise, all edges are oriented from their black to their white vertex, and edge monomorphisms to black vertices are surjective; if $v$ is black, then $\F_v=\langle r_v\rangle$ for some infinite-order $r_v$, and $r_e$ is sent by the edge map to $r_v$ for each $e$ with $e^-=v$.\label{item:linear-bipartite}

\item The action of $\maptor$ on the Bass-Serre tree $T$ is $4$--acylindrical and minimal.  \label{item:linear-acylindrical}

\item Let $b\in\vertices(\Delta)$ be black and let $v,v'\in\vertices(\Delta)$ be white vertices joined to $b$ by distinct edges $e,e'$.  Then $\langle t_{b}r_e^{-k(e)}\rangle \times\langle t_{b}r_{e'}^{-k(e')}\rangle$ has finite index in $\maptor_b$, so at least one of $k(e),k(e')$ is nonzero.\label{item:linear-finite-index}

\item If $\tilde v\in\vertices(T)$ is white and $\tilde b,\tilde b'\in\vertices(T)$ are distinct black vertices adjacent to $\tilde v$, then $\stabilizer_\maptor (\tilde b)\cap\stabilizer_{\maptor }(\tilde b')$ is cyclic.\label{item:linear-cylic-intersection}

\item If $\rank(\F)\geq 2$, then there is at least one white vertex, and the blocks in $\maptor$  are precisely the stabilizers of the white vertices, i.e. the centralizers in $\maptor$  of the various $t_v$.\label{item:linear-blocks}
\end{enumerate}
\end{prop}

\begin{proof}
The proof is a collection of citations to results in the literature; see \cite[Prop. 4.11]{HagenWise:new-polynomial} for details.  The only addition is \eqref{item:linear-blocks}, which follows from \eqref{item:linear-bipartite} and Remark \ref{rem:blocks-in-graph-of-groups}.
\end{proof}

\begin{remark}[Comparison to graph manifolds]\label{rem:3-manifold-comparison}
In the 3-manifold case, the blocks are the virtually $\F_k\times \integers$ pieces corresponding to the Seifert pieces (a Seifert manifold with nonempty boundary has virtually $\F_k\times\integers$ fundamental group).  In the manifold case, the blocks are always unbranched because each JSJ torus has exactly two blocks attached to it.  For mapping tori of linear growth $\phi\in\Out(\F)$, a torus can have more than two blocks attached.  So, in the geometric case, Proposition \ref{prop:split-over-tori} treats the JSJ tori as black vertex groups, and the blocks are the white vertices. We will see that being unbranched means there are two blocks per ``JSJ torus'', as in the geometric case. We refer the reader to  Example \ref{exmp:cat(0)-non-hhg}. We also note that linear mapping tori have cyclic splittings resembling those in Proposition \ref{prop:cyclic-hierarchy}, but they are not acylindrical (see Example \ref{exmp:cat(0)-non-hhg}), which is why one needs Proposition \ref{prop:split-over-tori}.
\end{remark}

\begin{remark}[Nontrivial splitting of terminal vertex groups]\label{rem:cyclic-hierarchy-suffix}
Consider a UPG element $\phi$ and let $\maptor=\F\rtimes_\phi\langle t\rangle$.  Suppose that $d\geq 2$ and consider the topmost edges splitting from Proposition \ref{prop:cyclic-hierarchy}.  Let $e$ be an edge in the underlying graph of the graph of groups, and let $e^-,e^+$ be the initial and terminal vertex groups, so that $\maptor_{e^\pm}=\F_{e^\pm}\rtimes_{\phi_{e^\pm}} \langle t_{e^\pm}\rangle$.  The edge group is a cyclic group $\maptor_e=\langle t_e\rangle$.

As noted in the proof of \cite[Prop. 2.5]{AndrewHughesKudlinska}, the edge-homomorphism $\maptor_e\to\maptor_{e^+}$ is given by $t_e\mapsto w_et_{e^+}$, where the nontrivial conjugacy class $[w_e]$ in $\F$ grows polynomially with degree $d-1$ under iterations of $\phi$.  

This implies that $\phi_{e^+}$ has polynomial growth rate at least $1$.  If it is more than $1$, then Proposition \ref{prop:cyclic-hierarchy} applies to $\maptor_{e^+}$, giving a nontrivial action of $\maptor_{e^+}$ on a Bass-Serre tree $T_{e^+}$ (with cyclic edge stabilisers) where the element $w_et_{e^+}$ is a hyperbolic element.  If $\phi_{e^+}$ is linearly growing, then Proposition \ref{prop:split-over-tori} gives an action of $\maptor_{e^+}$ on a tree $T_{e^+}$, where again $w_et_{e^+}$ is a hyperbolic element.  In either case, the splitting of $\maptor_{e^+}$ provided by whichever of the two proposition applies is nontrivial, and $w_et_{e^+}$ has an axis in the Bass-Serre tree.

Now, in $\maptor_{e^-}$, it is possible that $\phi_{e^-}$ has growth rate $0$, and it is possible that $\F_{e^-}$ has rank $0$ or $1$.  By the previous discussion, these things can only occur if all edges of the underlying graph incident to $e^-$ are outgoing.  In any case, the edge homomorphism sends $t_e$ to $t_{e^-}$.  Now, $\maptor_{e^-}$ also has an action on a tree $T_{e^-}$: if $\phi_{e^-}$ has superlinear growth, it is the Bass-Serre tree given by Proposition \ref{prop:cyclic-hierarchy}, and if the growth is at most linear, it is the Bass-Serre tree from Proposition \ref{prop:split-over-tori}; the $0$--growth case is a special case of the latter, where $T_{e^-}$ is a single vertex.  In any case, $\langle t_{e^-}\rangle$ is a maximal cyclic subgroup of a vertex stabilizer in $T_{e^-}$.  Moreover, this vertex stabilizer is itself a mapping torus (of a possibly growth--$0$) automorphism of a (possibly rank--$0$ or $1$ free group) whose $\integers$ semidirect factor is generated by $t_{e^-}$.  
\end{remark}

\subsection{Unbranched mapping tori}\label{subsec:unbranched-mapping-tori}
Now we discuss how to characterize unbranched mapping tori. The first two lemmas are used in the proof of Theorem \ref{thm:main}.

\begin{lem}\label{lem:valence}
Let $\Delta$ be the graph of groups decomposition of $\maptor$ from Proposition \ref{prop:split-over-tori}, where $\phi$ is a linearly growing UPG. Let $v\in\vertices(\Delta)$ be black and let $\tilde v\in T$ be a lift of $v$ to the Bass-Serre tree $T$.  Then $v$ has the same valence in $\Delta$ as $\tilde v$ has in $T$.
\end{lem}

\begin{proof}
Let $e$ be an edge of $\Delta$ incident to $v$; by Proposition \ref{prop:split-over-tori}.\eqref{item:linear-bipartite}, $e^-=v$ and the edge map $\maptor_e\to \maptor_v$ is an isomorphism.  Hence $e$ has a unique lift to $T$ incident to $\tilde v$, so the projection $T\to\Delta$ restricts to an homeomorphism from the star of $\tilde v$ to the star of $v$, as required.
\end{proof}

\begin{lem}\label{lem:linear-unbranched}
Let $\phi\in\Out(\F)$ be a linearly-growing UPG and let $\Delta$ be as in Proposition \ref{prop:split-over-tori}.  Then $\maptor$ is unbranched if and only if every black $b\in\vertices(\Delta)$ has valence  $2$.
\end{lem}

\begin{proof}
Color the vertices of $T$ according to the colors of their images in $\Delta$.  Valence $1$ black vertices are impossible by minimality of the action on $T$.  Let $b\in\vertices(\Delta)$ be black and let $e_1,e_2,e_3$ be distinct incident edges.  Let $\tilde b,\tilde e_i$ be lifts of $b,e_i$ to the Bass-Serre tree $T$ with each $\tilde e_i$ incident to $\tilde b$.  Let $\maptor_i=\stabilizer_\maptor(\tilde v_i)$, where $\tilde v_i$ is the white vertex at which $\tilde e_i$ terminates.  By Proposition \ref{prop:split-over-tori}, $\maptor_i\cong \F_i\times \langle t_i\rangle$ is a block.  So for $i\neq j$, the subgroup $\maptor_i\cap \maptor_j\leq \stabilizer_\maptor(\tilde b)$ is abelian and hence of infinite index in $\maptor_i$ and $\maptor_j$.  On the other hand, $\maptor_1\cap \maptor_2\cap \maptor_3=\stabilizer_\maptor(\tilde b)\cong\integers^2$, which is impossible if $\maptor$ is unbranched.

Now suppose black vertices in $\Delta$ (and hence in $T$, by Lemma \ref{lem:valence}) have valence $2$.  Let $B_1,B_2,B_3\leq \maptor$ be blocks in $\maptor$ such that $B_i\cap B_j$ has infinite index in $B_i$ for $i\neq j$.  Proposition \ref{prop:split-over-tori}.\eqref{item:linear-blocks} implies that each $B_i$ fixes a unique vertex $v_i\in T$, and $v_i$ is a white vertex.

Let $H=[v_1,v_2]\cup[v_2,v_3]\cup[v_3,v_1]\subset T$, which is a subtree of $T$, fixed pointwise by $B_1\cap B_2\cap B_3$ and having at most three leaves.  So there is at most one vertex $m\in H$ of valence more than $2$.  If such an $m$ exists, then it must be white and distinct from $v_1,v_2,v_3$, so $B_1\cap B_2\cap B_3$ is cyclic by Proposition \ref{prop:split-over-tori}.\eqref{item:linear-cylic-intersection}.  If there is no such $m$, then (say) $v_2\in[v_1,v_3]$.  If $v_2\not\in\{v_1,v_3\}$, then Proposition \ref{prop:split-over-tori}.\eqref{item:linear-cylic-intersection} implies $B_1\cap B_2\cap B_3$ is virtually cyclic.  Otherwise (say) $v_1=v_2$, so by maximality of blocks, $B_1=B_2$, contradicting that $[B_1:B_1\cap B_2]=\infty$.
\end{proof}

\section{Unbranched blocks from coarse medians}\label{sec:coarse-median-implies-unbranched}

\begin{defn}[Definition 4.2 \cite{MunroPetyt}]\label{2RBF} An $n-$dimensional \emph{richly branching flat} ($n$--\emph{RBF}) is a piecewise linear space $\mathcal R$ consisting of a base flat $\mathcal B$, which is an isometric copy of $\mathbb R^n$,
and half-flats $\mathbb R^{n-1} \times [0, \infty)$ glued to $\mathcal B$ the following way: choose $n$ linearly independent vectors $\{v_i\}^n_{i=0}$ in $\mathcal B$ and and choose a coarsely dense subset  $P_i\subset \mathbb R$, for all $i$. For each $i$, glue the boundary of a half flat to $\mathcal B$ along each affine subspace $pv_i+v^{\bot}_i$, where $p\in P_i$. 
\end{defn}

Since $\maptor$ has cohomological dimension $2$, Corollary 3.4 in \cite{MunroPetyt} says that, if $\maptor$ is coarse median, then it is coarse median of rank at most $2$, from which Theorem 4.3 of \cite{MunroPetyt} then implies that, if $\maptor$ is coarse median, then there is no quasi-isometric embedding of a $2$--RBF in $\maptor$.  This gives the implication \eqref{item:coarse-median}$\Rightarrow$\eqref{item:no-2-RBF} in  Theorem \ref{thm:main}, and will also be used in the proof of  following lemma. 

\begin{lem}[\eqref{item:coarse-median}$\implies$\eqref{item:unbranched-blocks}]\label{lem:coarse-median-implies-unbranched-blocks}
If $\maptor$ admits a coarse median, then $\maptor$ is unbranched.  In fact, if $\maptor$ is not unbranched, then $\maptor$ has a quasi-isometrically embedded $2$--RBF.
\end{lem}

\begin{proof}
Suppose $\Phi$ has polynomial growth rate $d\geq 0$ and that $\maptor$ is coarse median.  Since it suffices to prove that a finite index subgroup of $\maptor$ is unbranched, by Lemma \ref{lem:unbranched-index}, and since being coarse median is a QI invariant, assume that $\Phi$ is a UPG element \cite{BFH:kolchin}.  

\textbf{Linear growth.}  Suppose that $d\leq 1$ and $\maptor$ is not unbranched.  We will construct a quasi-isometrically embedded $2$--RBF in $\maptor$.  We will freely use the fact that abelian subgroups of $\maptor$ are quasi-isometrically embedded \cite[Cor. 4.5]{Button:aspects}.  

Consider the acylindrical splitting from Proposition \ref{prop:split-over-tori}.  By Lemma \ref{lem:linear-unbranched}, the underlying graph $\Delta$ has a black vertex $b$ of valence at least $3$.  Let $e_1,e_2,e_3$ be distinct edges of $\Delta$ with $e_i^-=b$ for all $i$.  Recall that $\maptor_b$ is generated by commuting elements $r_b,t_b$, and for each $i$, the edge map on $\langle r_b\rangle\oplus\langle t_b\rangle$ sends $r_b$ to $p_{e_i}$ and $t_b$ to $p_{e_i}^{k(e_i)}t_{e_i^+}$.  Let $L_i=\langle r_b^{-k(e_i)}t_b\rangle\leq \maptor_b$.  

Fix some maximal cyclic subgroup $\langle f_i\rangle$ of $F_{e_i^+}$ that intersects $\langle p_{e_i}\rangle$ trivially.  Since $t_{e_i^+}$ is central in $\maptor_{e_i^+}$, the subgroup $\maptor_{e_i^+}$ contains an $\integers^2$ subgroup $M_i=\langle f_i,t_{e_i^+}\rangle$ intersecting $\langle p_{e_i},t_{e_i^+}\rangle$ along $\langle t_{e_i^+}\rangle$. Let $M_i^+$ be the quasi-isometrically embedded subspace of $M_i$ consisting of elements of the form $f_i^at_{e_i^+}^b$ with $a,b\in\integers$ and $a\geq 0$.  

We now construct a $2$--RBF in $\maptor$ as follows.  The base quasiflat is $\maptor_b$.  For $i\in\{1,2,3\}$, and each $n\in\integers$, attach $p_{e_i}^nM_i^+$ to $\maptor_b$ by identifying $p_{e_i}^n\langle t_{e_i^+}\rangle$ with $r_b^n L_i$ via $r_b^n(r_b^{-k(e_i)}t_b)^m\mapsto p_{e_i}^nt_{e_i^+}^m$.  The inclusions of the base quasiflat and the various translates of the $M_i$ into $\maptor$ induce a quasi-isometric embedding of the $2$--RBF, by construction, since, for $i\neq j$, the quasilines $L_i$ and $L_j$ cannot be coarsely parallel, by Proposition \ref{prop:split-over-tori}.\eqref{item:linear-finite-index}. 

\textbf{Superlinear polynomial growth.}  Suppose that $d>1$ and that $\maptor$ is not unbranched.  By Proposition \ref{prop:cyclic-hierarchy}, $\maptor$ acts acylindrically on a tree $T$ with cyclic edge stabilizers, so by Proposition \ref{prop:unbranched-cyclic-splitting}, some vertex stabilizer $\maptor_v$ (the mapping torus of an autormorphism of polynomial growth rate $d_v\leq d-1$) fails to be unbranched.  By induction on $d$, with the base case $d= 1$ done above, there is a $2$--RBF in $\maptor_v$.  Since $\maptor_v$ is quasi-isometrically embedded in $\maptor$, by \cite[Lem. 4.1]{Mutanguha}, there is a quasi-isometrically embedded $2$--RBF in $\maptor$.

\textbf{Conclusion in the polynomial case.}  Thus far, we have shown that if $\Phi$ has polynomial growth and $\maptor$ is not unbranched, then $\maptor$ contains a $2$--RBF. 
 This contradicts the assumption that $\maptor$ admits a coarse median, by combining \cite[Thm. 3.6, Thm. 4.3]{MunroPetyt} with the observation that $\maptor$ has geometric dimension $2$.  

\textbf{Exponential growth.}  If $\Phi$ has exponential growth, then Proposition \ref{prop:relhyper} provides a relatively hyperbolic structure; we use the notation from that proposition.  If $\maptor$ has branching blocks, then by Proposition \ref{prop:relhyp-unbranched}, some peripheral subgroup $\maptor_i$ has branching blocks, and since $\maptor_i$ is the mapping torus of a polynomial growth automorphism, the previous part of the proof produces a $2$--RBF $F\subset \maptor_i$ that is quasi-isometrically embedded.  But $\maptor_i\hookrightarrow\maptor$ is a quasi-isometric embedding, so $F$ is a $2$--RBF in $\maptor$.   But since $\maptor$ has geometric dimension $2$, applying \cite{MunroPetyt} again shows that $\maptor$ cannot have a coarse median.  This completes the proof.
\end{proof}

\section{Hierarchically hyperbolic structures in the unbranched case}\label{sec:build-HHG}
In this section, we prove:

\begin{lem}[$\eqref{item:unbranched-blocks}\implies\eqref{item:virtual-hhg}$]\label{lem:unbranched-implies-HHG}
Suppose that $\phi\in\Out(F)$ has the property that $\maptor=\F\rtimes_\phi\integers$ is unbranched.  Then $\maptor$ has a finite-index subgroup $\maptor'$ that is a colorable hierarchically hyperbolic group.  In particular, $\maptor$ admits a coarse median of finite rank.
\end{lem}

\begin{proof}
Being coarse median is a quasi-isometry invariant property, and hierarchically hyperbolic groups are coarse median of finite rank \cite[Thm. 7.3]{HHSII}, so the first part of the lemma implies the second. Next, we claim that $\maptor$ is virtually a colorable hierarchically hyperbolic group.  We divide into cases according to the growth rate of $\phi$.

\textbf{Linear growth.} If $\phi$ has polynomial growth rate at most $1$, then $\maptor'$ is virtually a colorable HHG, by Lemma \ref{lem:unbranched-blocks-linear}, which is an application of \cite{HRSS}.

\textbf{Superlinear polynomial growth.}  Proposition \ref{prop:polynomial-growth-is-virtually-acceptable} below implies that $\maptor$ is virtually a colorable hierarchically hyperbolic group in the case where $\phi$ has polynomial growth rate at least $1$.  (This is proved by induction below, with the base case being the linear growth case.)

\textbf{Exponential growth.}  By Proposition \ref{prop:relhyper}, $\maptor$ is hyperbolic relative to a finite collection $\{\maptor_i\}_{i\in I}$ of mapping tori of polynomially-growing automorphisms.  By the preceding case, for each $i\in I$ there is a finite-index subgroup $\maptor'_i\leq \maptor_i$ admitting a hierarchically hyperbolic group structure $(\maptor'_i,\mathfrak S_i)$.  By Lemma \ref{lem:separability}, there is a finite-index normal subgroup $\maptor'\leq\maptor$ such that $\maptor'\cap \maptor_i\leq \maptor_i'$ for all $i\in I$.  For each $i$, let $\maptor''_i=\maptor'\cap \maptor_i$, and note that $\maptor_i''$ inherits a hierarchically hyperbolic group structure from $\maptor_i'$.  On the other hand, $\maptor'$ is hyperbolic relative to a finite collection of subgroups, each conjugate in $\maptor$ to some $\maptor''_i$.  So, by \cite[Thm. 9.1]{HHSII}, $\maptor'$ is hierarchically hyperbolic, as claimed.  Finally, since the HHG structure  $(\maptor',\mathfrak S)$ constructed in the proof of \cite[Thm. 9.1]{HHSII} does not introduce any orthogonality, one verifies easily that colorability of $(\maptor',\mathfrak S)$ is inherited from colorability of the $(\maptor_i,\mathfrak S_i)$.
\end{proof}

We will handle the polynomial case using the virtual cyclic hierarchy from Proposition \ref{prop:cyclic-hierarchy}, where the terminal vertex groups are mapping tori of linearly growing automorphisms, which are hierarchically hyperbolic by \cite{MunroPetyt}.  The following abstracts the properties that we need to verify at each stage of our induction on the length of the hierarchy.

\begin{defn}[Acceptable graph of groups]\label{defn:acceptable-cyclic-HHG-graph}
Let $\Omega$ be a finite graph and let $\mathcal G$ be a graph of groups with underlying graph $\Omega$.  For each $v\in\vertices(\Omega)$, let $G_v$ be the corresponding vertex group, and for each (oriented) edge $e$, let $G_e$ be the corresponding edge group.  Let $e^\pm$ be the endpoints of $e$, and let $\iota_e^\pm:G_e\to G_{e^\pm}$ be the edge-monomorphisms.  We call the graph of groups $(\Omega,\{G_v\}_v,\{G_e\}_e,\{\iota_e^\pm\})$ \emph{acceptable} if all of the following hold:
\begin{enumerate}
    \item For each $v\in\vertices(\mathcal G)$, there is a hierarchically hyperbolic group structure $(G_v,\mathfrak S_v)$.  For each $U\in\mathfrak S_v$, let $\pi_U:G_v\to \mathcal CU$ be the associated projection to a hyperbolic space.  Let $S_v\in\mathfrak S_v$ be the unique $\nest$--maximal element.\label{item:vertx-hhg}

    \item For each edge $e$, the group $G_e$ is infinite cyclic, generated by an element $g_e$.\label{item:edge-cyclic}

        \item For each $e$, the element $\iota^+_e(g_e)\in G_{e^+}$ is loxodromic on the hyperbolic space $\mathcal CS_{e^+}$.\label{item:edge-tree-hyperbolic}

    \item For each $e$, consider the element $z_e=\iota^-_e(g_e)\in G_{e^-}$.  Then $\langle z_e\rangle$ is hierarchically quasiconvex in $(G_{e^-},\mathfrak S_{e^-})$ and, more strongly, there exists a unique $U_e\in\mathfrak S_{e^-}$ such that $z_e$ is loxodromic on $\mathcal CU_e$, and $\mathcal CU_e$ is quasi-isometric to $\reals$ and $\nest$--minimal in $\mathfrak S_{e^-}$.\label{item:edge-hqc}

    \item Each $\iota^+_e(g_e)$ generates a maximal elementary subgroup of $G_{e^+}$.\label{item:edge-maximal}

    \item For each $v\in\vertices(\Omega)$, the following holds.  Let $U\in\mathfrak S_v$.  Then $\stabilizer_{G_v}(U)$ acts coboundedly on the \emph{standard product region} $P_U$ in $G_v$.\label{item:stabiliser-cobounded}
\end{enumerate}
Moreover, letting $T$ be the Bass-Serre tree, we require $G$ to act $2$--acylindrically on $T$
\end{defn}

\begin{remark}\label{rem:product-regions}
Standard product regions in hierarchically hyperbolic groups are introduced in \cite[Sec. 5]{HHSII} and a more detailed treatment is given in, for instance, \cite[Sec. 17]{CHRK}.  We just need the following facts.  Given an HHG $(G,\mathfrak F)$, to each $U\in\mathfrak F$ there is an associated subset $P_U\subseteq G$, the \emph{standard product region}.  For each $g\in G$, we have $P_{gU}=gP_U$, and in particular $\stabilizer_G(U)$ acts on $P_U$.  This action need not be cobounded (see e.g. \cite[Rem. 8.25]{HHSII} but in practice it often is.  Finally, $P_U$ is hierarchically quasiconvex (with HQC parameters independent of $U$) and for all $W\in\mathfrak F$, the projection $\pi_W:G\to\mathcal CW$ has the following properties with respect to $P_U$: 
\begin{itemize}
    \item $\pi_W(P_U)$ coarsely coincides with $\pi_W(G)\asymp\mathcal CW$ when $W\nest U$ or $W\orth U$, and
    \item $\pi_W(P_U)$ has uniformly bounded diameter for all other $W$.
\end{itemize}
Some other results on product regions are used below; references will be given as needed.
\end{remark}

\begin{defn}[Very HQC element]\label{defn:very-HQC-elements}
Let $(G,\mathfrak S)$ be a hierarchically hyperbolic group.  Recall from \cite[Sec. 6]{DHSbound} that for all infinite-order $g\in G$, the set $\bigset(g)$ of $U\in\mathfrak S$ such that $\pi_U(\langle g\rangle)$ is unbounded is nonempty and consists of pairwise orthogonal elements.

An element $g\in G$ is \emph{very HQC} if one of the following holds:
\begin{itemize}
    \item $\bigset(g)=\{S\}$, where $S\in\mathfrak S$ is the unique $\nest$--maximal element, or
    \item $\bigset(g)=\{U\}$, where $U\in\mathfrak S$ is $\nest$--minimal and $\mathcal CU$ is quasi-isometric to $\reals$.
\end{itemize}
In the former case, we say $g$ is \emph{top-level} and in the latter, $g$ is \emph{bottom-level}.
\end{defn}

\begin{lem}[Combination lemma]\label{lem:combination-theorem}
Let $(\Omega,\{G_v\}_v,\{G_e\}_e,\{\iota_e^\pm\})$ be an acceptable graph of groups, with total group $G$ and Bass-Serre tree $T$.  Then $G$ admits a hierarchically hyperbolic structure $(G,\mathfrak S)$ with the following properties:
\begin{itemize}
    \item $g\in G$ is a top-level very HQC element if and only if $g$ is loxodromic on $T$.
    \item For all $v\in \vertices(\Omega)$, the element $g\in G_v$ is bottom-level very HQC in $(G,\mathfrak S)$ if and only if it is bottom-level very HQC in $(G_v,\mathfrak S_v)$.
    \item For each $U\in \mathfrak S$, the group $\stabilizer_G(U)$ acts on the product region $P_U$ coboundedly.
\end{itemize}
Moreover, if each $(G_v,\mathfrak S_v)$ is colorable, then so is $(G,\mathfrak S)$.
\end{lem}

\begin{proof}
For each vertex $v\in\vertices(\Omega)$, by hypothesis we have a hierarchically hyperbolic group $(G_v,\mathfrak S_v)$.  For each $e\in\edges(\Omega)$, let $g_e$ generate the infinite cyclic group $G_e$, and let $z^\pm_e=\iota^\pm_e(g_e)$.  Let $\mathcal E^\pm_v$ be the set of $e\in\edges(\Omega)$ such that $e^\pm=v$.  

We now modify each $\mathfrak S_v$ as follows. Consider the finite collection $\mathcal H=\{\langle z_e^+\rangle:e\in\mathcal E_v^+\}$ of subgroups of $G_v$.  By Definition \ref{defn:acceptable-cyclic-HHG-graph}, each $z_e^+$ is loxodromic on the hyperbolic space $\mathcal CS_v$ corresponding to the $\nest$--maximal element $S_v$ of $\mathfrak S_v$.  Hence we can apply \cite[Prop. 6.14]{HHSIII} --- which applies since the edge groups $\langle z_e^+\rangle$ are maximal elementary subgroups that are top-level HQC --- to obtain a new HHG structure $(G_v,\mathfrak S'_v)$ which is identical to $(G_v,\mathfrak S_v)$, except that $\mathcal CS_v$ is replaced by the hyperbolic space $\mathcal C'S_v$ obtained by coning off each $G$--translate of each $\langle z_e^+\rangle\cdot \pi_{S_v}(1)$, and adding to $\mathfrak S_v$, for each coset $a\langle z_e^+\rangle$, a new element that is nested in $S_v$ and transverse to everything else. In summary:
\begin{center}
\emph{$(G_v,\mathfrak S_v')$ is an HHG structure where $\iota^\pm_e(g_e)$ is bottom-level very HQC for all $e\in\mathcal E^+_v\cup\mathcal E^-_v$.}  
\end{center}

(Note that this was already true of $e\in \mathcal E^-_v$ for the HHG structure $\mathfrak S_v$, because of Definition \ref{defn:acceptable-cyclic-HHG-graph}, and we did not change the bigset of such $g_e$ when applying \cite[Prop. 6.14]{HHSIII}.)

Now we will apply the combination theorem for hierarchically hyperbolic groups, namely \cite[Cor. 8.24]{HHSII}, whose hypotheses we check as follows.

\textbf{Hierarchical quasiconvexity.} The hierarchical quasiconvexity of edge groups in vertex groups, which is hypothesis (1) of \cite[Cor. 8.24]{HHSII}, holds because each $g_e$ maps to a bottom-level very HQC element in each of its incident vertex groups, equipped with the HHG structures $\mathfrak S_{e^\pm}'$.  (Note that $\iota_e^\pm$ is a uniform-quality quasi-isometric embedding since cyclic subgroups of HHGs are undistorted, by \cite[Thm. 7.1]{DHSbound} and \cite[Thm. 3.1]{DHScorr}.)

\textbf{Fullness.}  The requirement that each $\iota_e^\pm$ be a \emph{full hieromorphism} is why we passed to $\mathfrak S_v'$ for each $v$.  This requirement is hypothesis (2) of \cite[Cor. 8.24]{HHSII}.  First, we equip each $G_e$ with the trivial HHG structure $(G_e,\mathfrak S_e)$, where $\mathfrak S_e$ has a single element $S_e$ with $\mathcal CS_e$ the obvious Cayley graph of $\langle g_e\rangle$.  Now, in $\mathfrak S_{e^\pm}'$, the element $\iota^\pm_e(g_e)=z_e^\pm$ has a unique element in its bigset, which we denote $U_e^\pm$, and $U_e^\pm$ is $\nest$--minimal, and $\mathcal CU_e^\pm$ is a quasiline on which $\langle z_e^\pm\rangle$ acts loxodromically.  The action of $\langle z_e^\pm\rangle$ on $\mathcal CU^\pm_e$ defines an equivariant map $\mathcal CS_e\to \mathcal CU^\pm_e$ making $\iota_e^\pm$ an equivariant hieromorphism.  The \emph{fullness} condition (see \cite[Defn. 8.1]{HHSII}) is satisfied since $\mathcal CS_e\to \mathcal CU^\pm_e$ is a quasi-isometry (the coarse surjectivity uses that $\mathcal CU_e^\pm$ is a quasiline by Definition \ref{defn:very-HQC-elements}) and since $U_e^\pm$ is $\nest$--minimal in $\mathfrak S_{e^\pm}'$.

\textbf{Pairwise-orthogonal sets.}  We need to check hypothesis (5) of \cite[Cor. 8.24]{HHSII} (see also \cite[Rem. 8.25]{HHSII}), which says that for each $v\in\vertices(\Omega)$, the action of $G_v$ on the set $\mathfrak S_v'$  has the property that there are finitely many orbits of subsets $\mathcal U$ that consist of pairwise orthogonal elements.  We deduce this from the assumption in Definition \ref{defn:acceptable-cyclic-HHG-graph} that stabilizers of elements of $\mathfrak S_v$ act on product regions coboundedly, as follows.

The assumption implies that each $\stabilizer_{G_v}(U),\ U\in\mathfrak S'_v$, acts coboundedly on the associated $P_U$. Indeed, this holds for $U\in\mathfrak S_{v}$ by hypothesis; the new elements $U\in\mathfrak S_v'-\mathfrak S_v$ are not orthogonal to anything, and they are $\nest$--minimal, so $P_U$ coarsely coincides with the quasi-axis of some $z_e^+$, so $\langle z_e^+\rangle$ acts on $P_U$ coboundedly.  Now apply Lemma \ref{lem:cocompact-product} to each $(G_v,\mathfrak S_v')$.

\textbf{Bounded supports.}  Hypothesis (3) of \cite[Cor. 8.24]{HHSII}, bounded supports, follows from the assumption that the $G$--action on the Bass-Serre tree $T$ is acylindrical.  Indeed, acylindricity provides $n<\infty$ such that if $v,w\in \vertices(T)$ are at distance at least $n$, then $|\stabilizer_G(v)\cap\stabilizer_G(w)|<\infty$ (in our setting, the $2$--acylindricity hypothesis allows us to take $n=3$).  On the other hand, from Definition \ref{defn:acceptable-cyclic-HHG-graph} and \cite[Defn. 8.5]{HHSII}, the proof of \cite[Lem. 8.20]{HHSII} implies the following: if $v,w\in\vertices(T)$ belong to a common support, then the corresponding vertex spaces (viewed as subsets of $G$) have unbounded coarse intersection.  Since these vertex spaces are $G$--cosets of vertex groups, and $\vertices(\Omega)$ is finite, it follows that $\stabilizer_G(v)\cap\stabilizer_G(w)$ is infinite cyclic, and hence $d_T(v,w)\leq n$, as required.

\textbf{Additional hypotheses.}  The combination theorem \cite[Cor. 8.24]{HHSII} has two additional hypotheses, called (4) and (6) in \cite{HHSII}.  However, as observed by Berlai and Robbio in \cite[Rem. 4.7]{BerRob}, hypothesis (4) of \cite[Cor. 8.24]{HHSII} is unnecessary (the authors of \cite{BerRob} explain the slight change to the proof needed to eliminate that hypothesis).  The remaining hypothesis is (6).  Here, again, Berlai and Robbio come to our aid: by \cite[Thm. 4.9]{BerRob}, our edge-maps satisfy hypothesis (6) of \cite[Cor. 8.24]{HHSII} (which Berlai and Robbio note as a consequence of their theorem in \cite[Rem. 5.1]{BerRob}) once we observe that the trivial HHG structures $(G_e,\mathfrak S_e),\ e\in\edges(\Omega)$, are \emph{concrete} in the sense of \cite[Defn. 3.10]{BerRob}, which holds vacuously since $\mathfrak S_e$ has a single element.  

\textbf{The HHG structure.}  We now apply \cite[8.24]{HHSII} to obtain an HHG structure $(G,\mathfrak S)$.  

\textbf{Very HQC elements.}  The proof of \cite[8.24]{HHSII} (specifically \cite[Defn. 8.14]{HHSII}) says that the unique $\nest$--maximal element of $\mathfrak S$ has associated hyperbolic space equal to the Bass-Serre tree $T$, so $g\in G$ is top-level very HQC if and only if $g$ is loxodromic on $T$.  Now, any $U\in\mathfrak S-\{T\}$ for which $\mathcal CU$ is unbounded has the property that $U\in\mathfrak S_v'$ for at least one $v\in\vertices(T)$, and the hyperbolic space $\mathcal CU$ is inherited from the HHG structure $\mathfrak S_v'$, which implies the claim about bottom-level very HQC elements.

\textbf{Persistence of colorability.}  $\mathfrak S=\mathfrak S_0\sqcup\mathfrak K$, where $\mathfrak S_0$ consists of $T$, together with a set of equivalence classes $[V]$, where $V\in\bigcup_{v\in \vertices(T)}\mathfrak S_v$ and the equivalence relation is from \cite[Defn. 8.5]{HHSII}.  The set $\mathfrak K$ comes from Definition 8.11 in \cite{HHSII}.  Now, $\{T\}$ and $\mathfrak S_0$ and $\mathfrak K$ are all $G$--invariant, and by the proof of \cite[Lem. 8.13]{HHSII}, no two elements of $\mathfrak K$ are orthogonal.  Hence it suffices to show that elements of $\mathfrak S_0$ in the same $G$--orbit are non-orthogonal.  If $[W],[V]\in\mathfrak S_0$ and $[W]\orth[V]$, then by definition, there exists $v\in\vertices(T)$ and representatives $W_v,V_v\in\mathfrak S_v$ of the two classes, with $W_v\orth V_v$ in $\mathfrak S_v$.  By colorability of $(G_v,\mathfrak S_v)$, the elements $W_v,V_v$ are in different $G_v$--orbits.  Now, suppose toward a contradiction that $g\in G$ has the property that $g[W]=[V]$.  

Let $T_{[V]},T_{[W]}\subset T$ be the support trees of $[V],[W]$ (Definition 8.5 in \cite{HHSII}).  Note that $T_{[W]}=gT_{[V]}$ and hence $v\in T_{[V]}\cap gT_{[V]}$.

We will show below that $gv=v$.  Now, $V_v$ is the unique representative of $[V]$ in $\mathfrak S_v'$, and $W_v$ is the unique representative of $[W]=g[V]$ in $\mathfrak S'_v$.  So if $gv=v$, then $gV_v$ is the unique representative of $g[V]$ in $\mathfrak S'_{gv}=\mathfrak S'_v$, so $gV_v=W_v$, which contradicts the above conclusion that $V_v,W_v$ are in different $G_v$--orbits. So it indeed suffices to show $gv=v$.

Suppose that $T_{[V]}=\{v\}$. Then $T_{[W]}=\{gv\}$, so since $v\in T_{[V]}\cap T_{[W]}$, we must have $gv=v$, as required.  The other option is that $T_{[V]}$ is nontrivial, which is equivalent to the statement that the equivalence class $[V]$ has more than one element.  By construction, this implies that the representative $V_v\in\mathfrak S_v'$ is $\nest$--minimal, and since $V_v\orth W_v$ in $\mathfrak S_v'$, we have the following for all $u\in T_{[V]}$ adjacent to $v$: in $\mathfrak S_u'$, the representative $V_u$ of $[V]$ is $\nest$--minimal and not in $\mathfrak S_u$ (i.e. it is one of the new quasilines added to make the edge-hieromorphisms full).  In particular, $V_u$ is not orthogonal to any element of $\mathfrak S_u'$.  Now, since $\mathcal CW_v$ is unbounded and $V_v\orth W_v$, there are infinitely many vertices in $T_{[V]}$ incident to $v$, so by $2$--acylindricity, $u$ is a leaf of $T_{[V]}$. We have shown that $T_{[V]}$ is a star whose central vertex $v$ is the unique vertex in $T_{[V]}$ where the representative $V_v$ is orthogonal to some other domain in the HHG structure at that vertex.  Hence, $T_{[W]}=gT_{[V]}$ is also a star, centred at $gv$, and $W_{gv}$ is orthogonal to $V_{gv}$.  Thus the central vertices of the stars $T_{[V]}$ and $T_{[W]}$ coincide, i.e. $w=gv=v$.  

Thus, in either case, $gv=v$, and this contradicts colorability of $(G_v,\mathfrak S_v')$ as noted above.  This proves that $(G,\mathfrak S)$ is colorable.

\textbf{Cobounded stabilizers.}  We have to check that $\stabilizer_G(U)$ acts coboundedly on $P_U$ for all $U\in\mathfrak S$, where $P_U$ is the standard product region in $G$.  The argument has three cases, provided by the proof of \cite[8.24]{HHSII} (specifically, by Definition 8.11 in \cite{HHSII}):
\begin{enumerate}
    \item If $U\in\mathfrak S$ is the $\nest$--maximal element, then $P_U=G$ and $\stabilizer_G(U)=G$, and the left multiplication action of $G$ on itself is cobounded, so we are done.

    \item \label{item:cobounded-equiv-class} If $U$ is an equivalence class, i.e. $U=[U_v]$ where $U_v\in\mathfrak S_v$ for some vertex $v$ of $T$, then $\stabilizer_G(U)$ must stabilize the subtree $T_{[U_v]}$ on which $[U_v]$ is supported, so by bounded supports, $\stabilizer_G(U)\leq \stabilizer_G(v)$ for some choice of $v\in T_{[U_v]}$, namely any vertex $v$ in $T_{[U_v]}$ that is fixed by $\stabilizer_G(U)$. Specifically, either $T_{[U_v]}=\{v\}$, or $T_{[U_v]}$ is an infinite star with central vertex $v$.  In either case, let $P^v_U:=P_{U_v}$ be the standard product region in $G_v$ associated to $U_v$ in the HHG $(G_v,\mathfrak S_v)$.

    Note that $\stabilizer_G(U)=\stabilizer_{G_v}(U_v)$.  We claim that $P_U$ is contained in a regular neighbourhood in $G$ of $P^v_{U}$.  Since $\stabilizer_{G_v}(U_v)$ acts coboundedly on $P_{U}^v$ by hypothesis, it will follow from this that $\stabilizer_G(U)$ acts coboundedly on $P_U$.  For each $u\neq v$ in $T_U$, the vertices $u,v$ are adjacent and the representative $U_u$ of $U$ in $\mathfrak S_u'$ is $\nest$--minimal and not orthogonal to any element of $\mathfrak S_u'$, so $P_{U}^u\subset G_u$ is a quasiline at uniformly bounded Hausdorff distance from the edge space in which $G_v,G_u$ coarsely intersect.  On the $G_v$ side, this quasiline is contained in $P_{U}^v$, so we are done.  If no such $u$ exists, then either $P^v_U=P_U$, or $P_U$ consists of $P^v_U$ together with some quasilines corresponding to $\nest$--minimal elements of $\mathfrak S_v'$ that are nested in $U_v$ and whose support trees are stars centred at $v$.  In the latter case, we again conclude that $P_U$ intersects the adjacent vertex spaces in uniform neighbourhoods of edge spaces, so we are done.
    
        \item The final possibility is that $U=K^\orth_m([U_1],\ldots,[U_m])$, where the notation is as in \cite[Defn. 8.11]{HHSII}. From that definition, we obtain a set $[U_1],\ldots,[U_m]\in\mathfrak S$ such that $[U_i]\orth [U_j]$ whenever $i\neq j$, and each $[U_i]$ is an equivalence class in $\bigcup_v\mathfrak S_v$.  As noted in \cite[Defn. 8.11]{HHSII}, the $\sim$--class $[V]$ is nested in $U$ if and only if $[V]\orth[U_i]$ for all $i$, and $[V]\orth U$ if and only if $[V]\nest [U_i]$ for some $i$.  By the proof of \cite[Lem. 8.13]{HHSII}, no other element of $\mathfrak S$ is orthogonal to $U$. It is possible that there are $U'=K^\orth_{m'}([U_1'],\ldots,[U_{m'}'])$ with $U'\nest U$, but by \cite[Defn. 8.14]{HHSII}, for such $U'$, the space $\mathcal CU'$ is a single point.  Hence Lemma \ref{lem:cocompact-product}.\eqref{item:product-coarse-intersect} implies that $P_U$ uniformly coarsely coincides with $\bigcap_i\neb_r(P_{[U_i]})$ for some uniform $r$.  Lemma \ref{lem:cocompact-product}.\eqref{item:product-orthogonal-stabiliser} and item \eqref{item:cobounded-equiv-class} (just above) then imply that $\bigcap_i\stabilizer_G([U_i])\leq \stabilizer_G(U)$ acts coboundedly on $P_U$.
\end{enumerate}
We have now verified all parts of the statement of the lemma.
\end{proof}

The last part of the following lemma was observed without proof in \cite{HHSII}, and the other parts are routine HHS facts that we recall here for completeness.  We emphasise the standing assumption in this paper that HHS/G structures have coarsely surjective $\pi_U$ maps (see Definition \ref{defn:HHS}).

\begin{lem}\label{lem:cocompact-product}
Let $(G,\mathfrak F)$ be an HHG.  There exists a constant $R$ such that the following statements hold for any $\mathcal U\subseteq \mathfrak F$ such that $U\orth V$ for all distinct $U,V\in\mathcal U$: 

\begin{enumerate}[(i)]
    \item For any $r\geq R$, the intersection $P=\bigcap_{U\in\mathcal U}\neb_r(P_U)$ is hierarchically quasiconvex with parameters depending only on $r$ and the HHS structure, and for all $W\in\mathfrak F$, either $\pi_W(P)$ has diameter bounded in terms of $r$ and the HHS structure, or $\pi_W(P)\asymp \mathcal CV$ (uniformly). The latter holds if and only if either $W\nest U$ for some $U\in\mathcal U$, or $W\orth U$ for all $U\in\mathcal U$. \label{item:product-coarse-intersect}
    \item Suppose that $\stabilizer_G(V)$ acts on $P_V$ coboundedly for all $V\in\mathcal U$.  Then $\bigcap_{U\in\mathcal U}\stabilizer_G(U)$ acts on $P$ coboundedly.\label{item:product-orthogonal-stabiliser}

    \item  Suppose that $\stabilizer_G(V)$ acts on $P_V$ coboundedly for all $V\in\mathfrak F$. Then $\mathcal U$ belongs to one of finitely many $G$--orbits in $2^{\mathfrak F}$.\label{item:product-finitely-many-orthogonal-sets}
\end{enumerate}
\end{lem}

\begin{proof}
We work with a fixed word metric on $G$ associated to a finite generating set, so $(G,\mathfrak F)$ is a hierarchically hyperbolic space.

By \cite[Lem. 2.1]{HHSII}, $|\mathcal U|\leq \chi$, where $\chi\in\naturals$ depends only on the HHS structure.  Moreover, by the partial realization axiom \cite[Defn. 1.1]{HHSII}, we have $d_G(P_U,P_V)\leq s$ for all $U,V\in\mathcal U$, where $s<\infty$ depends only on the HHS structure.  So by, for instance, \cite[Thm. 3.5]{haettel_coarse_2020}, there exists $x\in G$ such that $d_G(P_U,x)\leq s'$ for all $U\in\mathcal U$, where $s'=s'(s,\chi)$ ultimately depends only on the HHS structure; in particular it is independent of $\mathcal U$.

Item \eqref{item:product-coarse-intersect} is then, for instance, an inductive application of \cite[Lem. 4.10, Lem. 1.20]{HHS:quasiflats}.

Next, we prove \eqref{item:product-orthogonal-stabiliser}.  By assumption, for any domain $V\in\mathcal U$, the subgroup $\stabilizer_G(V)$ acts uniformly coboundedly on $P_V$ (uniformity comes from the finite number of orbits in $\mathfrak F$), so there is a function $m:\reals\to\reals$ such that for all $r>0$, any $r$--ball in $G$ intersects  at most $m(r)$ standard product regions associated to elements of $G\cdot\mathcal U$.

To show that $\bigcap_{U\in\mathcal U}\stabilizer_G(U)$ acts on $P$ coboundedly, fix $U$ and let $Q=\bigcap_{V\in\mathcal U-\{U\}}\neb_r(P_V)$.  By induction, $H:=\bigcap_{V\in\mathcal U-\{U\}}\stabilizer_G(P_V)$ acts on $Q$ coboundedly (in the base case, $Q$ is a neighbourhood of one standard product region, and the claim holds by assumption).  If $B$ is a ball of arbitrary radius $r'$ centred at $y\in Q$, then for any $h\in H$ such that $hP$ intersects $B$, the product region $hP_U$ intersects $\neb_r(B)$ and thus there are at most $m(r)<\infty$ possibilities for $hP_U$, and thus finitely many possibilities for $h\neb_r(P_U)\cap Q=hP$.  It follows from the standard ``pigeonhole principle for cobounded actions'' (see e.g. \cite[Lem. 2.3]{HagenSusse}, or the related \cite[Prop. 7.2]{HruskaRuane}) that $\stabilizer_G(P)\cap H$ acts on $P$ coboundedly.  Now, $\stabilizer_G(P)\cap H$ permutes the translates of $\neb_r(P_U)$ that contain $P$, and there are only finitely many of these, so $\stabilizer_G(U)\cap H$ has finite index in $\stabilizer_G(P)\cap H$ and therefore $\stabilizer_G(U)\cap H=\bigcap_{V\in\mathcal U}\stabilizer_G(V)$ acts on $P$ coboundedly, proving \eqref{item:product-orthogonal-stabiliser}.

Finally, we prove item \eqref{item:product-finitely-many-orthogonal-sets}.  Here we have the stronger assumption that $\stabilizer_G(P_V)$ acts on $P_V$ coboundedly for all $V\in\mathfrak F$ (not just $V$ in the specific set $\mathcal U$).  As above, we obtain a function $m'$ such that any $r$--ball intersects at most $m'$ product regions (associated to any domains, not just domains in $\mathcal U$). Let $\mathcal W$ be the set of $W\in \mathfrak F$ such that $P_W$ intersects the $s'$--ball in $G$ about $1$.  We have shown that $\mathcal U\subset x\mathcal W$, so $\mathcal U$ belongs to the same orbit as some subset of $\mathcal W$.  But there are only $2^{|\mathcal W|}\leq 2^{m'(s')}$ such subsets and hence finitely many possibilities for the orbit $G\cdot\mathcal U$.
\end{proof}

\begin{prop}\label{prop:polynomial-growth-is-virtually-acceptable}
Suppose that $\maptor=\F\rtimes_\phi\integers$, where $\phi$ has polynomial growth rate $d\geq 2$ and $\maptor$ is unbranched.  Then, up to replacing $\maptor$ with a finite-index subgroup, $\maptor$ splits as a finite graph of groups that is acceptable, and hence is virtually a colorable HHG.
\end{prop}

\begin{proof}
We will show that the cyclic hierarchy (Proposition \ref{prop:cyclic-hierarchy}) gives us the required acceptable graph of groups decomposition as in Definition \ref{defn:acceptable-cyclic-HHG-graph}. Using \cite[Corollary 5.7.6]{BFH-00} we may replace $\phi$ with some power if necessary and assume that $\phi$ is UPG.

Item  \eqref{item:edge-cyclic} from Definition \ref{defn:acceptable-cyclic-HHG-graph} is a direct consequence of the topmost edges decomposition as mentioned in the paragraph before Proposition \ref{prop:cyclic-hierarchy}. 

If $e$ is an edge in this graph of group decomposition, then $\maptor_e = \langle t_e \rangle$ is an infinite cyclic group and $\maptor_{e+} = F_{e^+} \rtimes \langle t_{e^+} \rangle$, where the automorphism corresponding to $t_{e^+}$ has at least linear growth on $F_{e^+}$ (the notation is as in Remark \ref{rem:cyclic-hierarchy-suffix}).  A similar description holds for $\maptor_{e^-}$. In the positive direction, $t_e$ is identified with $wt_{e^+}$, for some $w\in F_{e^+}$.  Hence the action of $\iota^+(t_e)$ on the Bass-Serre tree of $\maptor_{e^+}$ is loxodromic, satisfying items \eqref{item:edge-tree-hyperbolic}, \eqref{item:edge-maximal} of the definition.  (See again Remark \ref{rem:cyclic-hierarchy-suffix}.)

In the opposite direction, the identification happens as $t_e \mapsto t_{e^-}$. To see why this is HQC in $\maptor_{e^-}$ (and why the latter is  an HHG, as required by item \eqref{item:vertx-hhg}), we need to use an inductive argument by first establishing items \eqref{item:vertx-hhg} and \eqref{item:edge-hqc} for the linear growth case, and this is where we will use the \emph{unbranched blocks} hypothesis.

By Lemma \ref{lem:unbranched-blocks-linear}, if $G_{e^-}$ is the mapping torus of an automorphism of polynomial growth rate $0$ or $1$, then it is an HHG and the element $t_{e^-}$ is bottom-level very HQC.

We can apply Lemma \ref{lem:combination-theorem} to conclude the same for $d\geq 2$ inductively once we verify \eqref{item:stabiliser-cobounded} below. This then completes verification of item \eqref{item:edge-hqc} of Definition \ref{defn:acceptable-cyclic-HHG-graph}. 

To check item \eqref{item:stabiliser-cobounded}, we have to identify the standard product regions.  In the linear case, this is again Lemma \ref{lem:unbranched-blocks-linear}.  Lemma \ref{lem:combination-theorem} verifies \eqref{item:stabiliser-cobounded} in the inductive step, which completes the proof. This completes verification of item \eqref{item:stabiliser-cobounded}, and we are done. 
\end{proof}

\begin{lem}\label{lem:unbranched-blocks-linear}
If $\phi$ has polynomial growth of degree at most $1$, and  $\maptor=\F\rtimes_\phi\integers$  is unbranched, then $\maptor$ is virtually a colorable hierarchically hyperbolic group.  Hence it is quasi-isometric to a finite-dimensional CAT(0) cube complex and is in particular coarse median of finite rank.  

Moreover, if $\phi$ is a linear UPG element, there is an HHG structure $(\maptor,\mathfrak S)$ with the following properties for the graph of groups decomposition from Proposition \ref{prop:split-over-tori}:
\begin{itemize}
    \item for each $v\in\vertices(\Delta)$, the element $t_v$ is bottom-level very HQC in $(\maptor,\mathfrak S)$, and
    \item for each $U\in\mathfrak S$, the subgroup $\stabilizer_\maptor(U)$ acts on  $P_U$ coboundedly.
\end{itemize}
\end{lem}

\begin{proof}
Replacing $\maptor$ by a finite-index subgroup, we assume $\phi$ is a UPG automorphism \cite{BFH:kolchin}.  Hence there is a splitting as in Proposition \ref{prop:split-over-tori}, and the assumption that $\maptor$ is unbranched enables one to apply Lemma \ref{lem:linear-unbranched} (along with Proposition \ref{prop:split-over-tori}) to see that this graph of groups is an  \emph{admissible graph of groups} as defined in \cite{CrokeKleiner} (see also \cite[Defn. 2.13]{HRSS}).  The verification of admissibility is exactly the same as in the proof of \cite[Prop. 5.4]{MunroPetyt}: the graph of groups used in \emph{loc. cit.} is the splitting from Proposition \ref{prop:split-over-tori}, except \cite{MunroPetyt} ignores the black vertices (which is fine, since they have valence $2$ and black vertex groups coincide with the incident edge groups by Proposition \ref{prop:split-over-tori}.\eqref{item:linear-bipartite}).  The key point in the definition of admissibility (distinct edges incident to a common white vertex correspond to incommensurable subgroups in the vertex group) is exactly Proposition \ref{prop:split-over-tori}.\eqref{item:linear-cylic-intersection}.  The colorable HHG structure comes from admissibility and \cite[Thm. 4]{HRSS}.  As noted in \cite{MunroPetyt}, this gives a quasi-isometry to a finite-dimensional CAT(0) cube complex, by \cite{Pet}.\footnote{This argument --- make an admissible splitting and then use \cite{HRSS} --- differs from the proof of  Proposition 5.4 of \cite{MunroPetyt} only in the hypotheses: Munro-Petyt's assumption is not that $\maptor$ is unbranched, but that $\phi$ is represented by a suitable relative train track map where each Nielsen cycle supports $\leq 1$ linear strata.} 

The preceding argument works if $\phi$ has growth $0$, i.e. the splitting from Proposition \ref{prop:split-over-tori} is trivial, i.e. $\Delta$ is a single vertex (black or white according to the rank of $\F$).  

\textbf{HQC elements.}  We now check the bottom-level very HQC assertion by examining the HHG structure from \cite{HRSS}.  This HHG structure comes from using Lemma 4.2 in \cite{HRSS},  which chooses appropriate infinite generating sets for the vertex groups  $F' \times \langle t_v \rangle$ in the Proposition \ref{prop:split-over-tori} decomposition to make them quasi-isometric to a line, and this quasiline is a $\nest$--minimal domain in $\mathfrak S$. The center of a vertex group is $\langle t_v\rangle$, which, by \cite[Lemma 4.2]{HRSS}, acts with unbounded orbits on the quasiline and has bounded projection to every other domain.  This shows that $\langle t_v \rangle$ is HQC in the linearly growing case and the bigset has a unique element, so $t_v$ is bottom-level very HQC.

\textbf{Product regions.} Finally, we prove the statement about standard product regions by inspecting the various numbered \emph{types} of domains from Figure 3 of \cite{HRSS}. \emph{Type (3)} corresponds to the standard product regions given by edge groups of type $\integers \times \integers$ and their translates. \emph{Type (7)} corresponds to the standard product region given by vertex groups of type $F'\times \integers$ and their translates. In either case, the standard product regions are translates of Cayley graphs of finitely generated subgroups, so the action is always cobounded. The other types which appear in \cite[Figure 3]{HRSS}, are orthogonal to one of these two types of domains, and do not create additional product regions.
\end{proof}

\subsection{From virtual hierarchical hyperbolicity to hierarchical hyperbolicity}\label{subsec:remove-virtual}
We now analyze the preceding construction in order to pass from a virtual HHG structure to an HHG structure in the unbranched case.  In the linear case, the main point is that the virtual splitting from Proposition \ref{prop:split-over-tori} is canonical enough that the virtual HHG structure constructed above admits an action by the finite-index supergroup, and hence gives an HHG structure there.  In the superlinear case, we use results from \cite{BFH:kolchin} in a similar way.

\begin{prop}\label{prop:remove-virtual}
If $\maptor$ is unbranched, then it is a colorable HHG.
\end{prop}

\begin{proof}
We first consider the case where $\phi$ has polynomial growth. Fix $k>0$ such that $\phi^k$ is a UPG element (necessarily also of growth rate $d$) and let $$\maptor'=\langle F,t^k\rangle_\maptor$$
be the corresponding finite-index subgroup of $\maptor$, so that $\maptor'\cong F\rtimes_{\phi^k}\integers$.  Since $\maptor$ is unbranched, $\maptor'$ is also unbranched, by Lemma \ref{lem:unbranched-index}.  Hence Lemma \ref{lem:unbranched-blocks-linear} or Proposition \ref{prop:polynomial-growth-is-virtually-acceptable} give a colorable hierarchically hyperbolic structure $(\maptor',\mathfrak F)$.  Indeed, in the proofs of those statements, the only passage to a finite-index subgroup came from taking a power of $\phi$. Note that $\maptor'$ is normal in $\maptor$ and $\{1,t,\ldots,t^{k-1}\}$ is a complete set of coset representatives for $\maptor'$ in $\maptor$.

\textbf{Finite order case.}  If $d=0$, then $\phi$ has finite order.  Hence Nielsen realisation provides a left-invariant proper metric on $F$ so that $\phi$ is an isometry.  Hence $\maptor$ acts by isometries on $F\times \reals$, with $t$ acting hyperbolically on $\reals$ and with bounded orbits on $F$, and $F$ acts trivially on $\reals$.  This yields an HHG structure by, for instance, \cite[Prop. 8.27]{HHSII}.

\textbf{Linear case.}  Suppose that $d\leq 1$.  Consider the (bipartite) graph of groups decomposition of $\maptor'$ given by Proposition \ref{prop:split-over-tori}.  We adopt the notation from that proposition; in particular, let $T$ be the Bass-Serre tree.  

By \cite[Prop. 5.2.2]{AndrewMartino:splitting} and \cite[Prop. 3.1.4]{AndrewMartino:splitting}, the $\maptor'$--action on $T$ extends to a $\maptor$--action on $T$.  This action preserves the decomposition into black and white vertices.  Thus $\maptor$ itself admits a splitting as a finite graph of groups satisfying the conclusion of Proposition \ref{prop:split-over-tori}, with the following change.  Each white vertex stabilizer has the form $\maptor_v=F_v\rtimes_{\phi_v}\langle t_v^i\rangle$ for some $i\in\{1,\ldots,k\}$, where $\phi_v$ is a finite-order automorphism whose order divides $k$.  So, $\langle t_v^k\rangle$ is central in $\maptor_v$, and $\maptor_v/\langle\langle t_v^k\rangle\rangle\cong F_v\rtimes C$, where $C$ is finite cyclic.  So $\maptor_v$ is a central extension of a nonelementary virtually free group.  Since $\maptor$ is unbranched, Lemma \ref{lem:unbranched-index} and Lemma \ref{lem:linear-unbranched} imply that black vertices have valence two, so for convenience we ``unsubdivide'' $T$ to remove the black vertices and regard them as midpoints of edges.

In order to make an HHG structure, we have to be slightly careful because \cite{HRSS} is written for graphs of groups that are \emph{admissible} in the sense of \cite[Defn. 2.13]{HRSS}, and the proper powers of $t_v$ above may complicate this.  We instead use a more general version of the \cite{HRSS} construction, due to Mangioni \cite{Mangioni:Short-1}. Specifically:

\begin{claim}\label{claim:blowup-materials}
The $\maptor$--action on $T$, and the finite-index subgroup $\maptor'$, define \emph{blowup materials} in the sense of \cite[Defn. 3.9]{Mangioni:Short-1}.
\end{claim}
\renewcommand{\qedsymbol}{$\blacksquare$}
\begin{proof}[Proof of Claim \ref{claim:blowup-materials}]
Since $T$ is a nontrivial tree, it satisfies the requirements of a \emph{support graph} in \cite[Defn. 3.9.(1)]{Mangioni:Short-1}: it is nontrivial, connected, is triangle-free and square-free, and $\maptor$ acts cocompactly since the graph of groups in Proposition \ref{prop:split-over-tori} is finite.  Since $T$ is a tree, the weak hyperbolicity condition from \cite[Defn. 3.9.(4)]{Mangioni:Short-1} holds. 

Next, as noted above, for each vertex $v\in T$, the stabilizer $\maptor_v$ fits into an exact sequence
$$1\to \langle t^{k}_v\rangle \to \maptor_v\to H_v\to 1,$$
where $H_v$ is commensurable with the nonelementary free group $F_v$. Moreover, $\langle t^{k}_v\rangle$ is contained in the centre of $\maptor_v$ and hence fixes each edge of $T$ incident to $v$.  This verifies \cite[Defn. 3.9.(2)]{Mangioni:Short-1}.

If $v,w\in T$ are adjacent vertices, joined by an edge $e$, then the edge-stabiliser $\maptor_e$ contains $\langle t_v^{k},t_w^{k}\rangle$ as a finite-index subgroup, by Proposition \ref{prop:split-over-tori}.  This verifies \cite[Defn. 3.9.(3)]{Mangioni:Short-1}.

For \cite[Defn. 3.9.(5),(6)]{Mangioni:Short-1}, we use the HHG structure on $\maptor'$ coming from its admissible graph of groups structure using the action on $T$ (see Lemma \ref{lem:linear-unbranched}) as follows.

For each $v$, the subgroup $\maptor'_v=\maptor'\cap \maptor_v$ has finite index, and is normal, in $\maptor_v$, and $t_v^{k}\leq \maptor'_v$.  Since the action of $\maptor'$ on $T$ splits it as an admissible graph of groups as in Lemma \ref{lem:unbranched-blocks-linear}, we again have that, by \cite[Lem. 4.2]{HRSS}, there is a homogeneous quasimorphism $\phi_v:\maptor'_v\to \reals$ such that $\phi_v(t_v^{k})\neq 0$ and $\phi_v(t_w^{k})=0$ whenever $w$ is adjacent to $v$.  (See also \cite[Sec. 2.3.3]{Mangioni:Short-1}.)  This verifies \cite[Defn. 3.9.(5)]{Mangioni:Short-1}. 

We finally check \cite[Defn. 3.9.(6)]{Mangioni:Short-1}.  Since we are assuming that $\maptor$ is unbranched, the same is true of $\maptor'$, by Lemma \ref{lem:unbranched-index}.  Hence, by Lemma \ref{lem:unbranched-blocks-linear}, $\maptor$ admits an HHG structure in which the edge groups are hierarchically quasiconvex, and, by the proof of the same lemma, each $\maptor_v$ is a standard product region and hence hierarchically quasiconvex.

Let $v_1,\ldots,v_k$ be a set of $\maptor$--orbit representatives of vertices in $T$.  Let $\iota:\maptor'\to\maptor$ be the inclusion.  Let $\mathfrak p:\maptor\to \maptor'$ be a quasi-inverse for $\iota$, which exists because of finiteness of the index.  For each $i$, let $\gate_i:\maptor'\to2^{\maptor'_{v_i}}$ be the coarse gate map provided by hierarchical quasiconvexity (see \cite[Defn. 5.4, Lem. 5.5]{HHSII}).  Let $j:2^{\maptor'}\to 2^{\maptor}$ be induced by $\iota$.  Then the map $\gate_i':=j\circ \gate'_i\circ \mathfrak p:\maptor\to 2^{\maptor}$ is a coarsely lipschitz coarse retraction, since $\gate_i'$ is. In \cite[Defn. 3.9.(6)]{Mangioni:Short-1}, there are two additional conditions on the coarse retraction, and these are satisfied by $\gate_i'$ for the following two reasons: first, in any hierarchically hyperbolic space, the gate map to a fixed standard product region sends other product regions to their coarse intersections with the fixed one (see e.g. \cite[Lem. 1.20, Lem. 1.27]{HHS:quasiflats}), and second, vertex spaces in $T$ corresponding to vertices at distance $2$ must coarsely intersect along a coset of the centre of the vertex-group lying between them in $T$. We have now checked all parts of \cite[Defn. 3.9.(6)]{Mangioni:Short-1}.
\end{proof}
\renewcommand{\qedsymbol}{$\Box$}

Theorem 3.10 of \cite{Mangioni:Short-1} and Claim \ref{claim:blowup-materials} imply that $\maptor$ is a hierarchically hyperbolic group.  (The aforementioned theorem strengthens the HHG result from \cite{HRSS}.)  Moreover, by the same theorem, there is a \emph{short} HHG structure $(\maptor,\mathfrak S)$, in the sense of \cite[Sec. 2.2]{Mangioni:Short-1}.  So \cite[Lemma 2.13]{Mangioni:Short-1} implies that stabilisers of standard product regions act coboundedly, and \cite[Axiom C]{Mangioni:Short-1} shows that each $\langle t_v\rangle$ is bottom-level very HQC.  Since $\maptor$ preserves the natural $2$--coloring of $\vertices(T)$, it follows from \cite[Rem. 3.16]{Mangioni:Short-1} that the HHG structure is colorable.

\textbf{Superlinear case.}  Suppose that $d\geq 2$.  Let $s=t^k$, so that $\maptor'=\langle F,s\rangle_\maptor$ is naturally isomorphic to the UPG mapping torus $F\rtimes_{\phi^k}\langle s\rangle$, and $\maptor/\maptor'=\{t^i\maptor':0\leq i\leq k-1\}$.

Apply Proposition \ref{prop:cyclic-hierarchy} (see also \cite[Lem 5.2]{KudlinskaValiunas}) to $\maptor'$ to obtain a cocompact, $2$--acylindrical action of $\maptor'$ on a tree $T$ such that $\maptor'\backslash T:=\Delta$ is the underlying graph of the cyclic splitting from Proposition \ref{prop:cyclic-hierarchy}.  For each vertex $v$ of $\Delta$, the vertex group $\maptor'_v:=\langle F_v,s_v\rangle_{\maptor'}$, where $F_v$ is a free factor of $F$ and $s_v=f_vt^k$ for some $f_v\in F$.  For each edge $e$, the edge group is $\maptor'_e=\langle s_e\rangle$ where $s_e=f_et^k$ for some $f_e\in F$.  Moreover, $\maptor'_v$ is naturally isomorphic to $F_v\rtimes_{\phi_v}\langle s_v\rangle$ where $\phi_v$ has polynomial growth rate at most $d-1$.  This includes the  growth--$0$ cases in Remark \ref{rem:cyclic-hierarchy-suffix}.

Exactly as in the proof of Proposition \ref{prop:polynomial-growth-is-virtually-acceptable}, the splitting of $\maptor'$ over $\Delta$ is an acceptable graph of groups, and hence there is an HHG structure $(\maptor',\mathfrak S)$ with the following properties:
\begin{enumerate}[(a)]
    \item For each $v\in \vertices(\Delta)$, there is an HHG structure $(\maptor'_v,\mathfrak S_v)$ in which stabilisers of product regions act coboundedly.
    \item For each edge $e$, the element $\iota^+_e(s_e)$ is loxodromic on the $\nest$--maximal hyperbolic space $\mathcal CS_{e^+}$ in $\mathfrak S_{e^+}$ and $\iota^-_e(s_e)$ is bottom-level very HQC in $(\maptor'_{e^-},\mathfrak S_{e^-})$.  
    \item More strongly, these HHG structures on vertex spaces make the graph of groups splitting of $\maptor'$ an acceptable graph of groups.
    \item The HHG structure $\mathfrak S$ is the one provided by Lemma \ref{lem:combination-theorem}.
\end{enumerate}

To extend the HHG structure over $\maptor$, we need to extend the action on $T$ over $\maptor$.  The following claim is known to experts, and the way to assemble it from results in the literature (specifically \cite{BFH:kolchin}), was kindly outlined for us by Naomi Andrew\footnote{Any inaccuracies are obviously the responsibilities of the present authors, though.}:

\begin{claim}\label{claim:naomi-monika-jp}
The action of $\maptor'$ on $T$ extends to a $2$--acylindrical action of $\maptor$ on $T$ such that:
\begin{itemize}
    \item For each $\tilde v\in\vertices(T)$, there exists $\tau_{\tilde v}\in Ft$ such that $s_{\tilde v}\in\langle \tau_{\tilde v}\rangle$ and $\maptor_{\tilde v}:=\stabilizer_{\maptor}(\tilde v)=\langle \maptor'_{\tilde v},\tau_{\tilde v}\rangle_{\maptor}$.
    \item For each edge $\tilde e$ of $T$, the stabiliser $\maptor_{\tilde e}:=\stabilizer_{\maptor}(\tilde e)$ is infinite cyclic, generated by some element $\tau_{\tilde e}$, and, if $e$ is the image of $\tilde e$ under the map $T\to \Delta$, then $s_e$ is conjugate in $\maptor$ to a positive power of $\tau_e$.
\end{itemize}
\end{claim}
\renewcommand{\qedsymbol}{$\blacksquare$}
\begin{proof}[Proof of Claim \ref{claim:naomi-monika-jp}]
Recall that $\phi^k\in\Out(F)$ is a UPG outer automorphism of polynomial growth rate $d\geq 2$.  We first recall the construction of the $\maptor'$--action on $T$, following \cite[Lem. 5.2]{KudlinskaValiunas}: let $f:B\to B$ be an improved relative train track representative for $\phi^k$.  Let $e_1,\ldots,e_m$ be the edges of $B$ with polynomial growth rate $d$ under iteration of $f$.  Let $B'$ be the subgraph of $B$ obtained by removing the interior of each edge $e_i$.  Then $\Delta$ is the graph obtained from $B$ by collapsing each component of $B'$ to a vertex.  So, $F=\pi_1B$ naturally splits as a graph of groups with underlying graph $\Delta$ and trivial edge groups.  As explained in \cite{KudlinskaValiunas}, $T$ is the Bass-Serre tree of this decomposition of $F$, and $f$ induces an automorphism $s:T\to T$, resulting in the $\maptor'$--action.

We now observe that applying \cite[Lem. 4.33]{BFH:kolchin} to the UPG automorphism $\mathcal O:=\phi^k$ and the free minimal $F$--tree $\widetilde B$ (the universal cover of $B$), the limiting tree $T\mathcal O^\infty$ produced by \cite{BFH:kolchin} is $F$--equivariantly isometric to $T$.  Indeed, from \cite[Defn. 4.27, Lem. 4.33]{BFH:kolchin}, we get that $T\mathcal O^\infty$ is the Bass-Serre tree of the splitting of $F$ with underlying graph obtained from $B$ by collapsing all edges except $e_1,\ldots,e_m$, exactly as in the construction of $T$.

We wish to show that the $T$--splitting of $F$ is actually $\phi$--invariant.  Let $\alpha:F\to\Isom(\widetilde B)$ be the action of $F$ by deck transformations (representing a point in the outer space of $F$) and consider the action $\alpha \Phi$, where $\Phi\in\Aut(F)$ represents the outer automorphism $\phi$ (for instance, $\Phi$ could be conjugation by $t$ in $\maptor$).  Let $X$ be the quotient of $\widetilde B$ by the action $\alpha\Phi$, and let $h:B\to X$ be the based homotopy equivalence that takes vertices to vertices and induces $\Phi$.  As in \cite[Defn. 4.27]{BFH:kolchin}, we then define, for any path $\sigma$ in $B$, the quantity:
$$L_B(\sigma)=\lim_{n\to\infty}\frac{1}{n^d}\|[h(f^n(\sigma))]\|_X,$$
where $[\bullet]$ denotes the tightening of a path in $X$ and $\|\bullet\|_X$ is path length in $X$.  As in \cite[Sec. 4]{BFH:kolchin}, we obtain an $F$--action on a limiting tree $T'$ provided by applying \cite[Lem. 4.33]{BFH:kolchin} with inputs $f:B\to B$ and $L_B$ defined as above using $X$.  We claim $L_B(e)=0$ for exactly those edges $e$ of $B$ that are not among the edges $e_1,\ldots,e_m$.  Indeed, let $e$ be an edge of $B$.  Then $f(e)=eu$, where $u$ is the suffix from Theorem \ref{thm:RTT}.  From the splitting
$$[f^n(e)]=eu[f(u)][f^2(u)]\cdots[f^{n-1}(u)]$$
and \cite[Prop. 4.21]{BFH:kolchin}, we deduce that $\|[h(f^n(e))]\|_X$ grows like $n^{d_e}$, where $d_e$ is the polynomial growth rate of $e$ under iteration of $f$.  This implies that $L_B(e)>0$ if and only if $e=e_i$ for some $i$.  Hence, by \cite[Lem. 4.33]{BFH:kolchin}, applied to $f:B\to B$ and the function $L_B$ defined using $X$, we see that $T$ was independent of the choice of $F$--action on $\widetilde B$.  Using \cite[Thm. 3.7]{CullerMorgan} then shows that the original action of $F$ on $T$, and the action twisted by $\Phi$, differ by an equivariant isometry $T\to T$.  This shows that the action of $\maptor'$ on $T$ extends to an action of $\maptor$; the descriptions of vertex and edge stabilisers in $\maptor$ follow from the corresponding descriptions for $\maptor'$ and finiteness of the index, as does $2$--acylindricity.   
\end{proof}
\renewcommand{\qedsymbol}{$\Box$}

A vertex $\tilde v$ of $T$ is \emph{cyclic} if $\maptor_{\tilde v}=\langle \tau_{\tilde v}\rangle$.  Cyclic vertices are lifts of vertices in $\Delta$ all of whose incident edges are outgoing.  Cyclic vertices can arise since $f(e)$ has a suffix but not a prefix for each growing edge $e$. This prevents $\maptor$ from inverting edges in $T$.  More precisely:

\begin{claim}\label{claim:no-inversions}
The action of $\maptor$ on $T$ is without inversions.
\end{claim}
\renewcommand{\qedsymbol}{$\blacksquare$}
\begin{proof}
Consider an edge $\tilde e$ of $T$.   Suppose that the action of $\maptor$ inverts $\tilde e$, so  $\tau_{\tilde e}(\tilde e^{\pm})=\tilde e^{\mp}$.

Let $c:\maptor\to\maptor$ be the inner automorphism given by conjugation by $\tau_{\tilde e}$.  Then $c(\maptor_{\tilde e^\pm})=\maptor_{\tilde e^\mp}$.  Now, let $t_{\tilde e^\pm}$ be the image of $\tau_{\tilde e}^2$ in $\maptor_{\tilde e^\pm}$.  Then $c(t_{\tilde e^\pm})=t_{\tilde e^\mp}$.  

Since $c$ preserves the finite-index subgroup $\maptor'$, it thus interchanges $\stabilizer_{\maptor'}(\tilde e^{\pm})$, taking the image $\langle s_{\tilde e^-}\rangle$ of $\langle s_{\tilde e}\rangle$ to the image $\langle s_{\tilde e^+}\rangle$ in $\stabilizer_{\maptor'}(\tilde e^+)$ of the same edge group. 

Let $T_{\tilde e^\pm}$ be the Bass-Serre tree of either the topmost edges decomposition of $\stabilizer_{\maptor'}(\tilde e^\pm)$ (in the case where $\phi_{\tilde e^\pm}$ is superlinear) or the decomposition from Proposition \ref{prop:split-over-tori} (in the linear case).  By Remark \ref{rem:cyclic-hierarchy-suffix}, $s_{\tilde e^+}$ acts hyperbolically on $T_{\tilde e^+}$.  By the construction of the topmost edges decomposition using IRTTs (so, edges are sent to paths with suffixes but not prefixes), $s_{\tilde e^-}$ is elliptic on $T_{\tilde e^-}$.  Now, $\tau_{\tilde e}$ conjugates the action of $\stabilizer_{\maptor'}(\tilde e^+)$ on $T_{\tilde e^+}$ to an action of $\stabilizer_{\maptor'}(\tilde e^-)$ on $T_{\tilde e^+}$ making $s_{\tilde e^-}$ hyperbolic. But the proof of Claim \ref{claim:naomi-monika-jp} shows that $\tau_{\tilde e}$ preserves the topmost edges splitting $T_{\tilde e^-}$, where $s_{\tilde e^-}$ is elliptic.  This is a contradiction.  
\end{proof}
\renewcommand{\qedsymbol}{$\Box$}

Before continuing the proof, here are two brief digressions about the preceding claim that we feel may be of interest.  First, one can also probably prove the previous claim using the fact that $c:\maptor'_{\tilde e^\pm}\to\maptor'_{\tilde e^\mp}$ is a quasi-isometry taking a quasi-axis for $t_{\tilde e^\mp}$ to a quasi-axis for $t_{\tilde e^\pm}$, and reach a contradiction since those two quasi-axes have polynomial divergence functions (in the vertex groups) of different growth rates, using arguments from \cite{Macura:detour}.  

Also, in the preceding argument, the presence of cyclic vertices implicitly played a role, via the fact that $s_{\tilde e^-}$ is elliptic on $T_{\tilde e^-}$.  If our relative train tracks were allowed to have prefixes and suffixes, i.e. $f(e)=veu$, then we could indeed have an inversion, and we would need to subdivide $T$.  So the preceding claim can be taken as saying that the fact that $f(e)=eu$ means that the subdivision has already happened, and manifests in cyclic vertices.    

We now return to the proof of Proposition \ref{prop:remove-virtual}.

We work with the $\maptor$--action on $T$ from Claim \ref{claim:naomi-monika-jp}. So, we have decomposed $\maptor$ as a finite graph $\maptor\backslash T:=\Delta_1$ of groups.  We now verify the conditions from Definition \ref{defn:acceptable-cyclic-HHG-graph} for $\Delta_1$ and $\maptor$.

First, for each $v\in\vertices(\Delta_1)$, let $\tilde v$ be a lift of $v$ to $T$ so that the vertex group $\maptor_v=\maptor_{\tilde v}$ and let $\mathfrak S_v=\mathfrak S_{\tilde v}$.  By induction on the polynomial growth rate, we can assume that the HHG structure $(\maptor_v',\mathfrak S_v)$ above was obtained by restricting to $\maptor_v'$ a colorable HHG structure $(\maptor_v,\mathfrak S_v)$ on the entire vertex group $\maptor_v$. In the base case, where $\maptor_v$ is a linear mapping torus, this holds because we showed above that there is an HHG structure $(\maptor_v,\mathfrak S_v)$ with the required properties.  This verifies Item \eqref{item:vertx-hhg} from Definition \ref{defn:acceptable-cyclic-HHG-graph}.  

Second, the edge groups in the splitting of $\maptor$ over $\Delta_1$ are conjugates of the various $\maptor_{\tilde e}=\langle \tau_{\tilde e}\rangle$, verifying Definition \ref{defn:acceptable-cyclic-HHG-graph}.\eqref{item:edge-cyclic}.  

Third, let $e\in\edges(\Delta_1)$.  Let $\tilde e$ be the lift of $e$ to $T$ with $\maptor_e=\maptor_{\tilde e}$.  Let $\tilde e^\pm$ be the endpoints of $\tilde e$, so $\maptor_e$ is generated by $\tau_e$, and the image of $\tau_e$ in $\maptor_{\tilde e^+}$ has $s_{\tilde e}$ as a positive power, and the above properties of $(\maptor'_{\tilde e^+},\mathfrak S_{\tilde e^+})$ therefore ensure that the image of $\tau_e$ is loxodromic on $\mathcal CS_{\tilde e^+}$, and its image in $\maptor_{\tilde e^-}$ is bottom-level very HQC since this is true of the image of $s_{\tilde e}$, which is a power of $\tau_{\tilde e}$.  This verifies Definition \ref{defn:acceptable-cyclic-HHG-graph}.\eqref{item:edge-tree-hyperbolic},\eqref{item:edge-hqc}. 

For each $v$ and each standard product region $P_U$ in $(\maptor_v,\mathfrak S_v)$, our choice of HHG structure $(\maptor_v',\mathfrak S_v)$ implies that $\stabilizer_{\maptor'_v}(U)$ already acts on $P_U$ coboundedly, so the finite-index overgroup $\stabilizer_{\maptor_v}(U)$ does as well.  This gives item \eqref{item:stabiliser-cobounded}.

We now verify item \eqref{item:edge-maximal}, which says that for all edges $e$ of $T$, the image $\tau_{e^+}$ of $\tau_e$ in $\stabilizer_\maptor(e^+)$ is a maximal cyclic subgroup of $\stabilizer_\maptor(e^+)$.  The edge $e$ is the image of an edge $\tilde e$ of $T$ that is stabilised by $\tau_e$.  Note that $s_e\in\langle \tau_e\rangle$.  The quotient $\maptor\to\integers$ with kernel $F$ sends $t$ to $1$, and $s_e$ to $k$, and $\tau_e$ to some positive divisor of $k$.  

We aim to show that $\tau_e^+$ generates a maximal cyclic subgroup of $\stabilizer_\maptor(\tilde e^+)$.  Let $p\in\stabilizer_\maptor(\tilde e^+)$ generate the elementary closure in $\stabilizer_\maptor(\tilde e^+)$ of $s_{\tilde e^+}$, which is virtually cyclic since $s_{\tilde e^+}$ is loxodromic on $\mathcal CS_{\tilde e^+}$, and hence cyclic since $\maptor$ is torsion-free.  Note that $\langle p\rangle$ contains $\langle s_{\tilde e^+}\rangle$ as a finite-index subgroup.  If $p\tilde e=\tilde e$, then item \eqref{item:edge-maximal} holds, so suppose not.  Then $\tilde e,p\tilde e$ are distinct edges of $T$ that both terminate at $\tilde e^+$.  Now, $s_{\tilde e^+}=p^m$ for some $m\neq 0$, so $s_{\tilde e^+}=ps_{\tilde e^+}p^{-1}$ stabilises both $\tilde e$ and $p\tilde e$, contradicting, for example, the proof of \cite[Lem. 5.2.(iii)]{KudlinskaValiunas}.
%
%
%
%
%
To conclude that the splitting $\Delta_1$ of $\maptor$ is an acceptable graph of groups decomposition, it remains to recall that the action of $\maptor$ on $T$ is $2$--acylindrical by Claim \ref{claim:naomi-monika-jp}.

The proof of Proposition \ref{prop:polynomial-growth-is-virtually-acceptable} now goes through verbatim: we apply Lemma \ref{lem:combination-theorem} to obtain a colorable HHG structure on $\maptor$ satisfying the properties needed for the induction (and restricting on $\maptor'$ to the HHG structure one would obtain by applying Lemma \ref{lem:combination-theorem} to the original acceptable graph of groups decomposition of $\maptor'$).

\textbf{Exponential case.}  Apply Propositions \ref{prop:relhyper} and \ref{prop:relhyp-unbranched}, and \cite[Thm. 9.1]{HHSII}.
\end{proof}

\section{Unbranching and excessive linearity}\label{subsec:excessive-linearity}
This section is only used in the proof of Theorem \ref{exceptCor}, and is the place where we work with CT maps.  The proof of Theorem \ref{thm:main} is independent of this subsection.

\begin{defn}\label{defn:fixed-subgroup}
Let $f:\GG\to\GG$ be a CT map fixing all vertices of $\GG$.  For each vertex $x\in\GG$, regard $f:(\GG,x)\to(\GG,x)$ and let $f_*\in \Aut(\pi_1(\GG,x))$ be the induced automorphism of the fundamental group.  Let $\Fix_x(f)\leq \pi_1(\GG,x)$ be the subgroup consisting of elements fixed by $f_*$.  Note that $g\in\pi_1(\GG,x)$ belongs to $\Fix_x(f)$ if and only if $g$ is the path-homotopy class of a closed Nielsen path based at $x$.
\end{defn}

\begin{defn}\label{defn:ct-cycle-vertex}
Let $f:\GG\to\GG$ be a CT map fixing all vertices of $\GG$.  For each Nielsen cycle $u$, let $E(u)$ be the set of linear strata supported on $u$, and let $T(u)$ be the set of vertices $x\in\GG$ such that $u$ passes through $x$ and $\rank(\Fix_x(f))>1$.
\end{defn}

\begin{defn}[Excessive linearity]\label{defn:excessive-linearity}
Let $f:\GG\to\GG$ be a CT map representing a UPG outer automorphism of at most linear growth.  We say that $f$ has \emph{excessive linearity} if $|E(u)\cup T(u)|\geq 3$ for some Nielsen cycle $u$ in $\GG$.
\end{defn}

\begin{remark}\label{rem:principal-lifts}
    A reader seasoned with the train-track theory would have perhaps noticed that what we are doing in Definition \ref{defn:ct-cycle-vertex} is essentially hunting for principals lifts of $\phi$ to $\Aut(\F)$, by choosing basepoints at vertices of $u$ which are principal. 
    
    To see why, use \cite[Fact 1.37]{HM-20} to conclude that under the condition that $\phi$ is UPG and every vertex of $\GG$ is fixed, all vertices of $\GG$ must be principal. Next we recall that a lift of a linearly growing $\phi$ to $\Aut(\F)$ is \emph{principal} if and only if the associated fixed subgroup has rank at least $2$ (see \cite[p. 40, last paragraph, and Remark 3.3]{FeighHandel:recognition}).  \\
    In fact, in the proof of the direction (2) $\implies$ (1) of Theorem \ref{lem:unbranched-implies-one-Nielsen}, the white vertices which are constructed are precisely the fixed subgroups of these principal lifts. 
\end{remark}

\begin{remark}\label{rem:excessive-nielsen-paths}
    In light of the above remark and looking back at Definition \ref{defn:excessive-linearity}, it is perhaps worth noticing that \emph{excessive linearity} is not about ``too many linear edges'' but rather about too many closed Nielsen paths glued at the vertices of linear edges and their suffixes. Equivalently, excessive linearity is about existence of too many principal lifts (giving large fixed subgroups) with basepoints chosen on the suffixes of linear edges.  Hence we chose to formulate the definition in terms of fixed subgroups, but still retained the phrase \emph{linearity} to stress that this property is something special to linearly growing outer automorphisms. 
\end{remark}
\begin{remark}\label{rem:rich-linearity}
Excessive linearity is closely related to the \emph{rich linearity} condition from \cite[Defn. 5.11]{MunroPetyt}, which, in the context of a linearly-growing CT map, reduces to the following: there is a Nielsen cycle $u$ supporting at least two linear strata, and if there are exactly two such strata, then there is a Nielsen path $p$ joining some vertex of $u$ to the initial vertex of a linear edge. Excessive linearity does not imply rich linearity, by Example \ref{exmp:not-rich}.
\end{remark}

\begin{exmp}[Excessive versus rich linearity]\label{exmp:not-rich}
Let $\GG$ be the following graph: there are three vertices, $x,y_1,y_2$.  There is a (directed) loop $A$ based at $x$, and loops $B_i$ based at $y_i$ for $i\in\{1,2\}$.  There is a directed edge $C_i$ from $y_i$ to $x$.  The map $f:\GG\to\GG$ fixes each vertex, fixes the edges $A,B_1,B_2$ pointwise, and sends $C_i$ to the path $C_iA^{p_i}$, where $p_1,p_2$ are distinct positive integers.  Let $\F=\pi_1\GG$ and let $\phi\in\Out(\F)$ be represented by $f$.  Then the mapping torus $\maptor$ of $\phi$ is unbranched.  One way to see this is to follow the construction from \cite{DahmaniTouikan} to compute the graph of groups $\Delta$ from Proposition \ref{prop:split-over-tori}.  In this case, $\Delta$ has a single black vertex (corresponding to the conjugacy class in $\F$ represented by $A$), two white vertices, and two edges.  So unbranching follows from Lemma \ref{lem:linear-unbranched}.  This $f$ does not have excessive (or rich) linearity.  See the the left side of Figure \ref{fig:excessive}.  The only Nielsen cycle supporting linear strata is $A$, and $E(A)=\{C_1,C_2\}$.  The only vertex traversed by $A$ is $x$.  To see that excessive linearity fails, we just have to check that $\Fix_x(f)$ is cyclic.  Suppose that $N$ is a closed Nielsen path based at $x$.  Either $N$ is path-homotopic to a power of $A$, or its splitting into fixed edges and indecomposable Nielsen paths contains some path $(C_iA^kC_i^{-1})^{\pm1}$.  This is impossible since $N$ starts and ends at $x$.

Now modify $\GG$ by adding a second loop $P$ at the vertex $x$, and extend  $f$ by requiring it to fix $P$.  Define $\phi$ as before, using the new $f$, and let $\maptor$ be the mapping torus.  This graph is contained in the picture at right in Figure \ref{fig:excessive}. 

First, observe that $f$ has excessive linearity.  Indeed, $E(A)$ has two elements, $C_1$ and $C_2$.  Meanwhile, $A$ and $P$ generate a nonabelian free subgroup of $\Fix_x(f)$, so $|T(x)|\geq 1$, so Definition \ref{defn:excessive-linearity} applies.

Second, $f$ does \emph{not} have rich linearity.  Indeed, the only Nielsen cycle supporting linear strata is $A$, and it supports exactly two, $C_1$ and $C_2$.  Let $\tau$ be a path that starts at $x$ and ends at the initial point of a linear edge; without loss of generality, $\tau$ ends at $y_1$.  Since Nielsen paths split as concatenations of fixed edges and paths of the form $C_iA^nC_i^{-1}$, each linear edge occurs an even number of times in a Nielsen path.  Since the endpoints of $\tau$ are on opposite sides of the separating edge $C_1$, it follows that $\tau$ is not a Nielsen path.  Hence $A$ does not have a \emph{nearby source} in the sense of \cite[Sec. 5]{MunroPetyt}, so $f$ does not have rich linearity.

Finally, since $f$ has excessive linearity, $\maptor$ is not unbranched, as we are about to show.  One can also see this by hand.  Indeed, the graph $\Delta$ from Proposition \ref{prop:split-over-tori} now has an additional white vertex, corresponding to the free subgroup generated by $A$ and $P$, incident to the (still unique) black vertex.  The $2$--complex in the picture at right in Figure \ref{fig:excessive} has fundamental group $\F$, and the union of the cylinders in the picture is a subcomplex corresponding to the cyclic subgroup associated to the black vertex of $\Delta$.  The three white vertices correspond to the three wedges of two circles based at the vertices in the complex.
\end{exmp}

\begin{figure}
\setlength{\unitlength}{0.01\linewidth}
\begin{picture}(80,40)
\put(0,8){\includegraphics[width=0.8\textwidth]{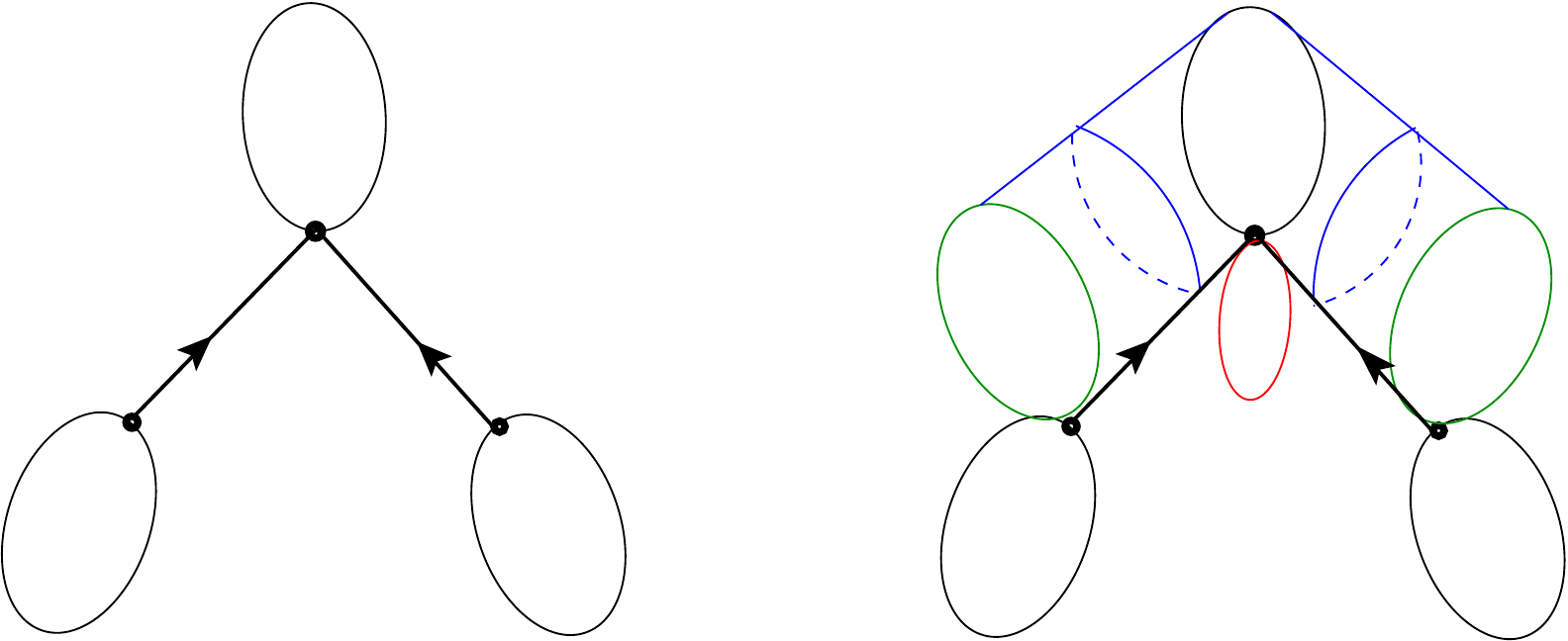}}
\put(15,31){$A$}
\put(15,25){$x$}
\put(62.5,31){$A$}
\put(7,24){$C_1$}
\put(8.5,18){$y_1$}
\put(21,18){$y_2$}
\put(23,24){$C_2$}
\put(4.5,6){$B_1$}
\put(25,6){$B_2$}
\put(62.5,16.5){\textcolor{red}{$P$}}
\end{picture}

\caption{Illustration of Example \ref{exmp:not-rich}. The example at left has unbranched mapping torus and does not have excessive, or rich, linearity.  The example at right has excessive linearity but does not have rich linearity.  In the picture at right, the graph $\GG$ consists of $A,B_1,B_2,C_1,C_2,P$, and we have also added the cylinders from the \cite{DahmaniTouikan} construction, which allows one to represent $\phi$ as a Dehn twist.  At the level of the CT maps, the only difference between the left and right pictures is the invariant loop $P$.}

\label{fig:excessive}
\end{figure}

The following theorem is the main technical tool towards proving Theorem \ref{exceptCor}.
\begin{lem:unbranched-implies-one-Nielsen}
Suppose $\phi\in\Out(\F)$ is a rotationless UPG automorphism of at most linear growth.  Then the following statements are equivalent: 
\begin{enumerate}
    \item $\maptor$ is not unbranched. 
    \item Any CT map $f:\mathbb G\to\mathbb G$ representing $\phi$ has excessive linearity.

\end{enumerate}
\end{lem:unbranched-implies-one-Nielsen}
Before the proof, we summarize the graph of groups construction of \cite{DahmaniTouikan} and the subsequent construction of cylinder complexes. 
This is the same graph of groups described in Proposition \ref{prop:split-over-tori}.  In the proof of Theorem \ref{thm:main}, we just use the statement of that proposition.  It is for Theorem \ref{lem:unbranched-implies-one-Nielsen}, and hence Theorem \ref{exceptCor}, that we need the following discussion.

\subsubsection{Summary of the \cite{DahmaniTouikan} construction.}\label{subsubsec:DT-construction-recall} Let $f:\GG\to\GG$ be a CT map representing $\phi$.  Let $\GG'_0$ be the subgraph consisting of all vertices together with all edges that are fixed by $f$.  For each linearly-growing edge $e_i$, let $u_i$ be the Nielsen cycle such that $f(e_i)=e_i\cdot u_i^{p_i}$ for some $p_i\in\mathbb Z-\{0\}$.
We view each linearly-growing edge $e_i$ as being oriented so that its terminal vertex is the initial vertex of $u_i$.  Sometimes, we refer to a linear edge $e$ with no subscript, in which case $u_e$ is the cycle, and $p_e$ is the integer, such that $f(e)=e \cdot u_e^{p_e}$.

We will first recall how the proof of \cite[Thm. 2.7]{DahmaniTouikan} produces, starting with the arbitrary CT map $f$, a graph of groups decomposition of $\F$ and hence of $\maptor$.  It is not yet the graph $\Delta$ of groups from Proposition \ref{prop:split-over-tori}, but rather the intermediate graph of groups that is called $\mathbb X$ in \cite{DahmaniTouikan}.  Our notation differs from \cite{DahmaniTouikan}, to be consistent with the rest of this paper.

\begin{defn}\label{defn:DT-S_k}
Let $S_1$ be the set of linearly growing edges $e_i$ such that $u_i$ is a path in $\GG'_0$, and let $\GG'_1$ be the subgraph consisting of $\GG'_0$ together with all edges in $S_1$. Inductively, for $k\ge 2$, let $S_k$ be the set of edges $e_i$ such that $u_i$ is a path in $\GG'_{k-1}$ that traverses at least one edge in $S_{k-1}$.  Let $\GG'_k$ be the union of $\GG'_{k-1}$ together with the edges in $S_k$.
\end{defn}

The vertex set $V$ of $\mathbb X$ has a vertex for each component of the fixed subgraph $\Fix(f)$ of $\GG$.  Vertices of $\GG$ are fixed by $f$, so each vertex of $\GG$ is contained in a subgraph corresponding to a vertex in $V$.  

\begin{defn}\label{defn:DT-initial-vertex-spaces}
For $\mathfrak v\in V$, let $K_{\mathfrak v}$ be the corresponding connected fixed subgraph of $\GG$.
\end{defn}

Vertices $\mathfrak v,\mathfrak w\in V$ are joined by a directed edge $(\mathfrak v,\mathfrak w)$ in $\mathbb X$ if there is a linear edge $e$ of $\GG$ whose initial vertex is in $K_{\mathfrak v}$ and whose terminal vertex is in $K_{\mathfrak w}$.  We denote the edge of $\mathbb X$ corresponding to the linear edge $e$ of $\mathbb G$ by $\hat e$.

\begin{defn}\label{defn:DT-X_v}
For each $\mathfrak v\in V$, the vertex space $X_{\mathfrak v}$ is the connected graph consisting of $K_{\mathfrak v}$, together with some additional edges that we call \emph{new loops}.  Specifically, for each linear edge $e$ whose initial vertex $v$ is in $K_{\mathfrak v}$, add a $1$--cycle $\mu_e$ starting and ending at $v$.
\end{defn}

For each edge $\hat e$ of $\mathbb X$, the edge space is a circle.  The total space $Z$ of the graph of spaces is obtained from $\bigsqcup_{\mathfrak v\in V}X_{\mathfrak v}$ as follows.  For each edge $\hat e=(\mathfrak v,\mathfrak w)$, we add a cylinder $S^1\times [-1,1]$.  The circle $S^1\times \{-1\}$ is mapped to $X_v$ by identifying it with the new loop $\mu_e$.  The circle $S^1\times\{1\}$ is glued using an immersed cycle $\hat u_e\to X_{\mathfrak w}$ that is chosen in the proof of \cite[Thm. 2.7]{DahmaniTouikan} to correspond to the Nielsen cycle $u_e\to \GG$ in the following sense.

First, note that the inclusions $K_{\mathfrak s}\to X_{\mathfrak s},\ \mathfrak s\in V$ induce an embedding $\GG\hookrightarrow Z$ sending each linear edge $e$ to a segment of the form $\{\theta\}\times[-1,1]$ in the cylinder $S^1\times [-1,1]$ that corresponds to $\hat e$.  Hence we view $\GG$ as a subspace of $Z$.  Dahmani-Touikan show that the original Nielsen cycle $u_e$ in $\GG$ is homotopic in $Z$ to a cycle $\hat u_e$ lying in the vertex space $X_{\mathfrak w}$ at which $\hat e$ terminates.

Their construction inducts on the stratification $\GG_1'\subset\GG_2'\subset\cdots\subset\GG_n'=\GG$: they first add cylinders for the $e\in S_1$, for which $u_e$ is already a cycle in $K_{\mathfrak w}$.  If $e\in S_2$, then $u_e$ splits as a concatenation of fixed edges and indecomposable Nielsen paths of the form $au_a^ma^{-1}$, where $a\in S_1$ (\cite{DahmaniTouikan}, Lemma 2.8).  The cylinders attached at the previous stage therefore allow them to homotope $u_e$ to a cycle in a vertex space.  They continue this procedure inductively.

This completes the description of $Z$.  Dahmani and Touikan explicitly construct a homotopy inverse for the inclusion $\GG\hookrightarrow Z$ by iteratively collapsing free faces coming from the new loops.  This deformation retraction of $Z$ onto $\GG$ simultaneously homotopes each $\hat u_e$ to $u_e$.

\subsubsection{Characterizing unbranching}  Next, Dahmani-Touikan  pass from the $\F$--action on the Bass-Serre tree $T$ of the above splitting to a tree of cylinders.  The tree of cylinders is the Bass-Serre tree $\widetilde \Delta$ of the splitting from Proposition \ref{prop:split-over-tori} (called $T_c$ in \cite{DahmaniTouikan}).
As noted in \cite{DahmaniTouikan}, the white vertices of $\Delta$ correspond to the $\mathfrak v\in V$ such that $\rank(\pi_1X_v)\geq 2$.  The black vertices correspond to \emph{(JSJ) cylinders} in $T$. For each Nielsen cycle $u_e$, there is a black vertex $b_e$, and every black vertex has the form $b_e$ for at least one $u_e$.  There is an associated space $Z(u_e)$, a \emph{cylinder complex}, defined as follows.

Let $\widetilde Z$ be the universal cover of $Z$.  Let $C_e\leq \F$ be a cyclic subgroup generated by an element of the conjugacy class represented by $u_e$.  There is a finite subtree $\Omega_e$ of $T$ (the $C_e$--JSJ cylinder) which is maximal with the property that $C_e$ fixes $\Omega_e$ pointwise.
Let $\widetilde Z(u_e)$ be the preimage of $\Omega_e$ under the projection $\widetilde Z\to T$.  Let $Z(u_e)=C_e\backslash \widetilde Z(u_e)$ be the quotient by the deck transformation action.  The inclusion $\widetilde Z(u_e)\to\widetilde Z$ descends to a map $Z(u_e)\to Z$ with the following properties.

\begin{itemize}
    \item $Z(u_e)$ is a finite tree of spaces whose vertex spaces are cycle graphs (circle divided into edges and vertices) with (possibly trivial) trees attached, and edge spaces are (topological) circles.  The underlying tree is $\Omega_e$ and its edges are oriented.

    \item There is a locally injective combinatorial map $Z(u_e)\to Z$ that sends the open cylinders corresponding to the edges of $\Omega_e$ homeomorphically to edge-cylinders in $Z$, preserving the orientation of the underlying edge.  It sends each vertex space in $Z(u_e)$ via an immersion to a vertex space of $Z$.  If $C$ is a cylinder in $Z(u_e)$ that is sent to the edge-cylinder in $Z$ corresponding to an edge $\hat a=(\mathfrak v,\mathfrak w)$, then the initial vertex space in $C$ is sent to the new loop $\mu_a$ in $X_{\mathfrak v}$, and the terminal vertex space is sent to the cycle $\hat u_a$ in $X_{\mathfrak w}$.

    \item Some cylinder in $Z(u_e)$ is mapped to the cylinder in $Z$ corresponding to the edge $\hat e$ of $\mathbb X$ (and hence to the edge $e$ of $\GG$).  In particular, $\Omega_e$ has at least one edge.
\end{itemize}
Note that $Z(u_e)$ deformation retracts to any of its vertex spaces, it has fundamental group $C_e$, and the map $Z(u_e)\to Z$ induces the inclusion $C_e\hookrightarrow\pi_1Z\cong \F$.

\begin{defn}\label{defn:DT-v-omega}
For each vertex $\omega\in\Omega_e$, let $\mathfrak v(\omega)$ be the vertex in $\mathbb X$ such that the $\omega$--vertex space in $Z(u_e)$ is sent to the vertex space $X_{\mathfrak v(\omega)}$.  
\end{defn}

The tree of cylinders construction shows that $\Delta$ has a black vertex $b_e$ for each $Z(u_e)$, and $b_e$ has one incident edge in $\Delta$ for each $\omega\in \vertices(\Omega_e)$ such that $\rank(\pi_1(X_{\mathfrak v(\omega)}))>1$.

\begin{lem}\label{lem:omega-e-nielsen}
Let $\ell\subset \GG\subset Z$ be a linear edge of $\GG$.  Then $\ell$ is in the image of the map $\Omega_e\hookrightarrow T\to \mathbb X$ if and only if $u_e=u_{\ell}$ (as unbased immersed cycles).
\end{lem}

\begin{proof}
Let $\tilde e$ be the lift of $\hat e$ to $T$ that is fixed by $C_e$.  Let $\ell$ be a linear edge.  Then $\ell$ is in the image of $\Omega_e\hookrightarrow T\to \mathbb X$ if and only if there is a lift $\tilde \ell$ of $\hat \ell$ such that $\stabilizer_{\F}(\tilde \ell)=C_e$.  This holds if and only if $C_{\ell}$ can be chosen in its conjugacy class so that $C_e=C_{\ell}$.  But the conjugacy class of $C_{\ell}$ is represented by a homotopy class of closed paths containing a unique immersed cycle, namely $u_{\ell}$, so $C_e$ and $C_{\ell}$ are conjugate if and only if $u_e=u_{\ell}$. 
\end{proof}

The preceding lemma immediately yields:

\begin{cor}\label{cor:omega-e-support}
For each linear edge $e$, the set of linear edges $e_i$ lying in the image of $\Omega_e\to \mathbb X$ is exactly the set of linear edges $e_i$ supported on the Nielsen cycle $u_e$.
\end{cor}

\begin{defn} A \emph{leaf} $\omega \in \Omega_e$ is a valence one vertex. 
It is \emph{outgoing} (resp. \emph{incoming}) if the edge in $\Omega_e$ incident to $\omega$ starts (resp. terminates) at $\omega$.
\end{defn}

Outgoing leaves in $\Omega_e$ correspond to new loops in $Z$, in the following sense.  Let  $\omega\in \Omega_e$ be an outgoing leaf and let $a$ be the linear edge of $\GG$ such that this edge in $\Omega_e$ maps to $\hat a$.  Then the $\omega$--vertex space in $Z(u_e)$ maps to the new loop $\mu_a$ in $X_{\mathfrak v(\omega)}$.  

\begin{lem}\label{lem:omega-leaves-white-vertices}
If $\omega\in\Omega_e$ is a leaf, then $X_{\mathfrak v(\omega)}$ corresponds to a white vertex of $\Delta$. 
\end{lem}

\begin{proof}
Viewing $\Omega_e$ as the $C_e$--JSJ cylinder, we see that $\stabilizer_{\F}(\omega)$ must contain $C_e$ as an infinite-index subgroup, since otherwise we could enlarge the JSJ cylinder $\Omega_e$.  Hence multiple cylinders in $T$ contain $\omega$, so it corresponds to a white vertex.
\end{proof}

\begin{cor}\label{cor:delta-valence-bound}
Let $b_e$ be a black vertex of $\Delta$.  Then the valence of $b_e$ in $\Delta$ is equal to the number of leaves of $\Omega_e$, plus the number of non-leaf vertices $\omega$ such that $\rank(X_{\mathfrak v(\omega)})>1$.
\end{cor}

\begin{proof}
By Lemma \ref{lem:valence}, the valence of $b_e$ in $\Delta$ is the same as the valence in $\widetilde{\Delta}$ of the vertex $\tilde b_e$ stabilized by $C_e$.  This in turn coincides with the number of vertices $\omega\in\Omega_e$ such that $X_{\mathfrak v(\omega)}$ corresponds to a white vertex of $\Delta$.  Lemma \ref{lem:omega-leaves-white-vertices} implies that each leaf $\omega$ of $\Omega_e$ has this property, so each leaf contributes one to the valence of $b_e$.  A non-leaf vertex $\omega$ contributes $1$ if and only if $\rank(X_{\mathfrak v(\omega)})>1$.
\end{proof}

\begin{remark}
The corollary is illustrated by Example \ref{exmp:not-rich} and Figure \ref{fig:excessive} (right), where the cylinder subtree is $\Omega_{C_1}=\Omega_{C_2}$, the associated cylinder complex is the union of the two cylinders in the picture, and the three white vertices correspond to the two leaves along with a third white vertex for the free group generated by $A$ and $P$. 
\end{remark}

\begin{lem}\label{lem:omega-no-saddle}
If $\omega\in\Omega_e$ is a non-leaf vertex, then all edges of $\Omega_e$ incident to $\omega$ are incoming.
\end{lem}

\begin{proof}
Let $a,b$ be distinct linear edges supported on $u_e$ and having endpoints on $K_{\mathfrak v(\omega)}$, which exist by Corollary \ref{cor:omega-e-support} and the assumption that $\omega$ is not a leaf.  Since $u_e=u_a=u_b$, the Nielsen path $u_b$ cannot traverse $a$, and vice versa, so $\hat u_b$ cannot traverse $\mu_a$, and vice versa.  On the other hand, the vertex space of $Z(u_e)$ corresponding to $\omega$ is a graph whose fundamental group is carried by a cycle $\mu_a$ or $\hat u_a$, according to the orientation of $a$.  Since $\mu_a\neq \mu_b$, it follows that $a$ and $b$ correspond to edges of $\Omega$ that are incoming at $\omega$.
\end{proof}

\begin{proof}[Proof of Theorem \ref{lem:unbranched-implies-one-Nielsen}]
By Lemma \ref{lem:linear-unbranched}, $\maptor$ fails to be unbranched if and only if there is a black vertex $b$ of valence at least three.  Since black vertices correspond to cylinder complexes, Corollary \ref{cor:delta-valence-bound} and Lemma \ref{lem:omega-no-saddle} imply that this is equivalent to the existence of a linear edge $e$ of $\GG$ such that $\Omega_e$ has a single non-leaf vertex $\omega$ of valence at least $2$ in $\Omega_e$, and either $\omega$ has valence at least $3$, or $\rank(X_{\mathfrak v(\omega)})>1$.  Let $u=u_e$ be the Nielsen cycle in $\GG$ supporting $e$.  By Corollary \ref{cor:omega-e-support}, this is equivalent to the statement that the set $E(u)$ of linear edges supported on $u$ has cardinality at least $2$, and either $|E(u)|\geq 3$ or $\rank(X_{\mathfrak v(\omega)})>1$.  Finally, let $x\in\GG$ be a vertex traversed by $u$.  Recall that the inclusion $(\GG,x)\to(Z,x)$ induces an isomorphism on fundamental groups, and in fact that there is a deformation retraction $Z\to \GG$ formed by successively collapsing free faces corresponding to new loops.  In particular, if $\hat p$ is an immersed closed path in $X_{\mathfrak v(\omega)}$ representing a nontrivial element of $\pi_1(X_{\mathfrak v(\omega)},x)$, then $\hat p$ is homotopic (rel endpoints) in $Z$ to a Nielsen path $p$, and hence $\rank(\Fix_x(f))\geq \rank(\pi_1(X_{\mathfrak v(\omega)},x))\geq 2$.  It follows that $|E(u)\cup T(u)|\geq |E(u)\cup\{x\}|\geq 3$.

Conversely, suppose that $E(u)$ has exactly two elements $a,b$, and $\rank(X_{\mathfrak v(\omega)})\leq 1$.  We claim that $|T(u)|=\emptyset$.  We first claim that $X_{\mathfrak v(\omega)}$ consists of fixed edges.  If not, then there is a linear edge $c$ such that $X_{\mathfrak v(\omega)}$ contains the new loop $\mu_c$, so since $\rank(X_{\mathfrak v(\omega)})\leq 1$, each of $\hat u_a,\hat u_b$ (which are immersed cycles in $X_{\mathfrak v(\omega)}$ corresponding to maximal cyclic subgroups) coincides with $\mu_c^{\pm}$.  This means that $a$ and $b$ have the same terminal vertex and, by examining the homotopy induced by collapsing free faces, $u=cu_c^{\pm 1}c^{-1}$, contradicting that $u$ is an immersed cycle.  Hence $X_{\mathfrak v(\omega)}=K_{\mathfrak v(\omega)}$, and in particular $\hat u_a=\hat u_b=u$ is a Nielsen path consisting of fixed edges.  For any vertex $x$ of $u$, no linear edge has initial vertex $x$ (that would lead to a new loop in $X_{\mathfrak v(\omega)}$) and no linear edge other than $a$ and/or $b$ terminates at $x$.  Hence, exactly as in the first part of Example \ref{exmp:not-rich}, $\rank(\Fix_x(f))<2$.  Thus $T(u)=\emptyset$, as claimed. 

Now we can summarise.  First, suppose that $\maptor$ is not unbranched.  As noted above, this implies there is a Nielsen cycle $u$ supporting at least two linear edges, and either $|E(u)|\geq 3$, or $|E(u)|=2$ and $u$ traverses a vertex $x$ with $\rank(\Fix_x(f))\geq 2$, so $x\in T(u)$.  In either case, we have verified excessive linearity, so \eqref{maptorunbranched} implies \eqref{onelinearedgepercycle}.  Conversely, suppose that excessive linearity holds, and choose a Nielsen cycle $u$ with $|E(u)\cup T(u)|\geq 3$.  As noted above, either $\maptor$ is unbranched, or for all choices of $u$ with $|E(u)|\geq 2$, we have $|E(u)|=2$, but the $u$ witnessing excessive linearity sastisfies $u=u_e$ where the non-leaf vertex $\omega\in\Omega_e$ satisfies $\rank(X_{\mathfrak v(\omega)})> 1$ (using the preceding paragraph).  As noted above, this also implies that $\maptor$ is not unbranched, so we have shown \eqref{onelinearedgepercycle} implies \eqref{maptorunbranched}.
\end{proof}

\section{Proofs of Theorem \ref{thm:main} and Theorem \ref{exceptCor}}\label{sec:proof}
\begin{proof}[Proof of Theorem \ref{thm:main}] The theorem is now a consequence of the preceding lemmas.  Indeed, \eqref{item:coarse-median} implies \eqref{item:unbranched-blocks} by Lemma \ref{lem:coarse-median-implies-unbranched-blocks}, and \eqref{item:unbranched-blocks} implies \eqref{item:virtual-hhg} and hence \eqref{item:coarse-median} by Lemma \ref{lem:unbranched-implies-HHG}.  Assertion \eqref{item:virtual-hhg} implies \eqref{item:QI-to-cube-complex} by \cite[Thm. B]{Pet}, which implies \eqref{item:coarse-median} as an immediate consequence of the definition of a coarse median space \cite{Bowditch:coarse-median}.  In the proof of Lemma \ref{lem:coarse-median-implies-unbranched-blocks}, we showed that, if $\maptor$ is not unbranched, then there is a quasi-isometrically embedded $2$--RBF in $\maptor$, i.e. \eqref{item:no-2-RBF} implies \eqref{item:unbranched-blocks}, while in the same proof, we observed that the results in \cite{MunroPetyt} imply that \eqref{item:coarse-median} implies \eqref{item:no-2-RBF}.  The equivalence of hierarchical hyperbolicity (\ref{item:actual-hhg}) and virtual hierarchical hyperbolicity (\ref{item:virtual-hhg}) is given by Proposition \ref{prop:remove-virtual}.  Thus all of the enumerated parts of the theorem are equivalent.  The coarse median rank of $\maptor$ (if it is coarse median) is finite, by \cite[Thm. 7.3]{HHSII}.
\end{proof}

Finally, we characterize (virtual) HHGs among free by cyclic groups in terms of restriction of $\phi$ to invariant free factors where the rate of growth of conjugacy classes is at most linear. For the definition of \emph{excessive linearity},  see Definition  \ref{defn:excessive-linearity}. 

\begin{exceptCor}\label{cor:excep}
Let $\phi\in\Out(F)$ have infinite order. Then the following are equivalent, where $k\geq 1$ is a constant depending only on $\rank(\F)$:
\begin{enumerate}
    \item $\F\rtimes_{\phi}\integers$ is a colorable hierarchically hyperbolic group.\label{item:col-hhg}
    \item Given any $\phi^k-$invariant free factor $[F^1]$ such that ${{\phi^k}\mid}_{F^1}$ has at most linear growth, no CT map representing this restriction has excessive linearity. \label{item:excessive-linearity}
\end{enumerate}
\end{exceptCor}

\begin{proof} We consider three cases, according to the growth rate of $\phi$.  There exists $k\geq 1$ such that $\phi^k$ is rotationless and hence admits CT representatives and, moreover, if $\phi$ is polynomially growing, then $\phi^k$ is UPG \cite[Lem. 4.42]{FeighHandel:recognition},\cite[Thm. 4.28]{FeighHandel:recognition}, \cite[Cor. 5.7.6]{BFH-00}.  By replacing $\phi$ with $\phi^k$, we therefore assume that $\phi$ itself already has these properties. Observe that Theorem \ref{thm:main} implies it is sufficient to prove that \eqref{item:excessive-linearity} is equivalent to $\maptor$ being unbranched, and that by Lemma \ref{lem:unbranched-index}, there is therefore no loss of generality in passing to the power.  

\textbf{Case 1:} Assume $\phi$ has at most linear growth and let $F^1 = \F$. Theorem \ref{lem:unbranched-implies-one-Nielsen} proves $(\ref{item:col-hhg}) \Leftrightarrow (\ref{item:excessive-linearity})$ for the linear case. 

\textbf{Case 2:} Assume that $\phi$ has quadratic growth. Let $F^1$ be a free factor of $\F$ whose conjugacy class $[F^1]$ is $\phi-$invariant and has the property that $\phi\mid_{F^1}$ has linear growth. Let $f_1: Y_1\to Y_1$ be some CT map representing $\phi\mid_{F^1}$, and assume, towards a contradiction, that there is excessive linearity in $Y_1$. Theorem \ref{lem:unbranched-implies-one-Nielsen} implies that $\maptor_1 = F^1 \rtimes_{\phi} \integers$ is not unbranched.  Using the cyclic hierarchy splitting  Proposition \ref{prop:cyclic-hierarchy} we see that some terminal vertex group is not unbranched. Proposition \ref{prop:unbranched-cyclic-splitting} then gives us that $\maptor$ must not be unbranched. This contradicts our hypothesis that $F\rtimes_\phi\integers$ is a colorable HHG.  Therefore we conclude that, for the restriction of $\phi$ to any linearly growing free factor, there are no CT maps where we can witness excessive linearity, proving $\eqref{item:col-hhg}\implies\eqref{item:excessive-linearity}$ for the higher order polynomial growth case.

Suppose next that (\ref{item:excessive-linearity}) holds in the quadratic case and call the mapping torus of $\phi$ as $\maptor$. By the cyclic hierarchy decomposition (Proposition  \ref{prop:cyclic-hierarchy}) we have a graph of groups decompositon of $\maptor$ that satisfies the following conditions: 
\begin{itemize}
    \item Each vertex group is isomorphic to $\integers$. 
    \item Each vertex group is a mapping torus corresponding to a restriction of $\phi$ to an invariant free factor, and the growth of the restriction is at most linear. 
    \item The $\maptor-$action on the Bass-Serre is acylindrical. 
\end{itemize}
Theorem \ref{lem:unbranched-implies-one-Nielsen} implies that each vertex group is unbranched. Invoking Proposition \ref{prop:unbranched-cyclic-splitting}, we get that $\maptor$ is unbranched. Theorem \ref{thm:main} proves $(\ref{item:excessive-linearity}) \implies  (\ref{item:col-hhg})$ for the quadratic growth case. 

\textbf{Case 3:} Assume $\phi$ has polynomial growth of order $\geq 3$. The direction $(\ref{item:col-hhg}) \implies (\ref{item:excessive-linearity})$ has already been handled while proving it for the quadratic case. For the converse direction, proceed by induction replicating the steps of case (2) and using the  appropriate version of the cyclic hierarchy decomposition. In summary, we have shown that when $\phi$ is UPG and rotationless, $(\ref{item:col-hhg}) \Leftrightarrow (\ref{item:excessive-linearity})$.

\textbf{Case 4:} We finally consider the case where $\phi$ has exponential growth. We first show \eqref{item:col-hhg} $\implies$ \eqref{item:no-exceptional}. Proposition \ref{prop:relhyper} and Proposition \ref{prop:relhyp-unbranched} together imply that all the peripheral subgroups from Proposition \ref{prop:relhyper} must be unbranched. Our proof of for the polynomially growing cases completes the proof in this direction. 
    
We now prove \eqref{item:no-exceptional} $\implies$ \eqref{item:col-hhg}. Using the polynomially growing case, we get that each peripheral subgroup is unbranched. Using Proposition \ref{prop:relhyp-unbranched} now completes the proof.
\end{proof}

\section{Examples}\label{sec:exs}

\begin{example}[Quadratic growth HHG]\label{exmp:quad}
Consider the automorphism 
\[ a\mapsto a,\,\,\, b\mapsto ba,\,\,\,\, c\mapsto cb \]
and the CT  representative $f$ on the standard rose with filtration
\[\emptyset \subset \{a \}\subset \{a,b\}\subset \{a,b,c\}.\]
Its cyclic hierarchy decomposition consists of a vertex stabilizer $H= \langle a, b\rangle \rtimes \langle t_v\rangle $ and edge stabilizer $\langle t_e \rangle $. 
Then the mapping torus
\[\maptor=\langle H, c^{-1}\co  c^{-1}t_ec=bt_v, t_e=t_v \rangle\]
is virtually an HHG for any of the following reasons.
\begin{enumerate}

\item The CT representative $f$ does not have excessive linearity (see Definition \ref{defn:excessive-linearity}). The given CT has only one linear stratum, $H_2=\{b\}$. Excessive linearity requires more than one linear stratum by definition.
So $\maptor$ is an HHG by Theorem \ref{thm:main}.

\item The mapping torus $\maptor$ is unbranched (see Definition \ref{defn:block-unbranched}). There is only one block up to conjugacy, namely $\langle a, bab^{-1}\rangle \times \langle t \rangle$, and one can use this to directly verify that $\maptor$ has unbranched blocks and apply Theorem \ref{thm:main} to get that $\maptor$ is an HHG.

\item The mapping torus $\maptor$ is (virtually) an acceptable graph of groups. Because $\maptor$ has quadratic growth and unbranched blocks, Proposition \ref{prop:polynomial-growth-is-virtually-acceptable} gives that $\maptor$ virtually splits as an acceptable graph of groups and is virtually HHG.
\end{enumerate}
\end{example}

Example \ref{exmp:quad} acts geometrically on a CAT(0) space.  Indeed, $H=\langle a,t_v\mid [a,t_v],b^{-1}t_vb=at_v\rangle$.  Let $\langle a,t_v\rangle_H$ act on $\reals^2$ by sending $at_v$ and $t_v$ to the translations $(1,0)$ and $(0,1)$ respectively.  Let $X=\langle a,t_v\rangle_H\backslash \reals^2$, which is a (metric) torus with fundamental group identified with $\langle a,t_v\rangle_H$ in such a way that $at_v$ and $t_v$ are represented by distinct closed geodesics of length $1$.  Let $B$ be a cylinder $S^1\times [0,1]$, for some $L\geq 0$, and form a space $Y$ by attaching the circles $S^1\times\{0\},S^1\times \{1\}$ using these geodesics.  By construction, $\pi_1Y=H$, and by \cite[II.11.21]{BH:book}, the universal cover of $Y$ admits a CAT(0) metric in which $t_v$ and $b$ both have translation length $1$.  Another application of \cite[II.11.21]{BH:book} shows that $\maptor$ is CAT(0).

\begin{example}[Linear growth non-HHG]\label{exmp:gersten}
Consider the automorphism 
\[ a\mapsto a,\,\,\, b\mapsto ba,\,\,\,\, c\mapsto ca^2 \]
for which the mapping torus $\maptor$ is not CAT(0) \cite{Gersten}. Consider the CT representative $f$ on the standard rose with the filtration as in the previous example.
In this example $H_1=\{a\}$ where $a$ is a Nielsen cycle, and $H_2=\{b\}$ and $H_3=\{c\}$ are linear strata supported by $a$. Then the mapping torus
$$ \maptor = \langle a,b,c,t : tat^{-1} =a, tbt^{-1} = ba, tct^{-1}=ca^2 \rangle$$
is not an HHG, for any one of the following reasons.

\begin{enumerate}

\item The CT representative $f$ has excessive linearity (see Definition \ref{defn:excessive-linearity}) because the Nielsen cycle $a$ supports exactly two linear strata $b$ and $c$, while the lift corresponding to the (only) vertex of \emph{a} has fixed subgroup of rank $\geq 2$ since it contains $a, b^{-1} a b, c^{-1} a c$. So the count from Definition \ref{defn:excessive-linearity} is exactly $3$.

\item The mapping torus $\maptor$ is not unbranched (see Definition \ref{defn:block-unbranched}). Consider the blocks
\begin{align*}
    B_1 &= \langle a, bab^{-1} \rangle \times \langle ta \rangle \\
    B_2 &= \langle a, cac^{-1} \rangle \times \langle ta^2 \rangle \\
    B_3 &= \langle a, bab^{-1}, cac^{-1} \rangle \times \langle t \rangle
\end{align*}
whose intersection is $B_1 \cap B_2 \cap B_3 = \langle a, ta^2 \rangle\cong\integers^2$, which has infinite index in all $B_i$. 
\end{enumerate}
In either circumstance, $\maptor$ is not (even virtually) an HHG by Theorem \ref{thm:main}.
\end{example}

\begin{example}[$\mathrm{CAT}(0)$ non-HHG]\label{exmp:cat(0)-non-hhg}
Consider the outer automorphism  $\phi$ given by 
\[ a\mapsto a,\,\,\, b\mapsto aba,\,\,\,\,
 c\mapsto a^2ca^2. \]
This $\maptor=\F_3\rtimes_{\phi} \mathbb Z $ appears elsewhere in the literature: It is $\mathrm{CAT}(0)$ \cite{lyman:new-cat-0} but not cocompactly cubulable \cite[Theorem 1.4]{wu-ye:questions}.  A computation reveals:
$$b(at)b^{-1}=a^{-1}t,\ \ \ \ c(a^2t)c^{-1}=a^{-2}t,$$
so the mapping torus $\maptor$ of $\phi$ splits as a double HNN extension with cyclic edge groups,  vertex group $\langle a\rangle \times \langle t\rangle$, and stable letters $b$ and $c$ performing the above conjugations.  We claim that $\maptor$ is not unbranched, and hence not an HHG, by Theorem \ref{thm:main}.  It is sufficient to verify this in a finite-index subgroup, by Lemma \ref{lem:unbranched-index}, which enables us to pass to a finite-index subgroup, where the offending blocks are easier to visualize, as follows.  The map to the underlying graph corresponds to a homomorphism $\pi:\maptor\to \F(b,c)$ given by $a\mapsto 1,t\mapsto1,b\mapsto b,c\mapsto c$.  Let $H=\langle b^3,c^3, bcb^{-1},b^{-1}cb,cbc^{-1},c^{-1}bc\rangle\leq \F(b,c)$, and let $\maptor'=\pi^{-1}(H)$, which is an index-$5$ subgroup of $\maptor$; the corresponding cover $X'$ of the graph of spaces $X$ associated to the above double HNN extension is shown in Figure \ref{fig:finite-cover}.

\begin{figure}[h]
    \centering
    \includegraphics[width=0.4\linewidth]{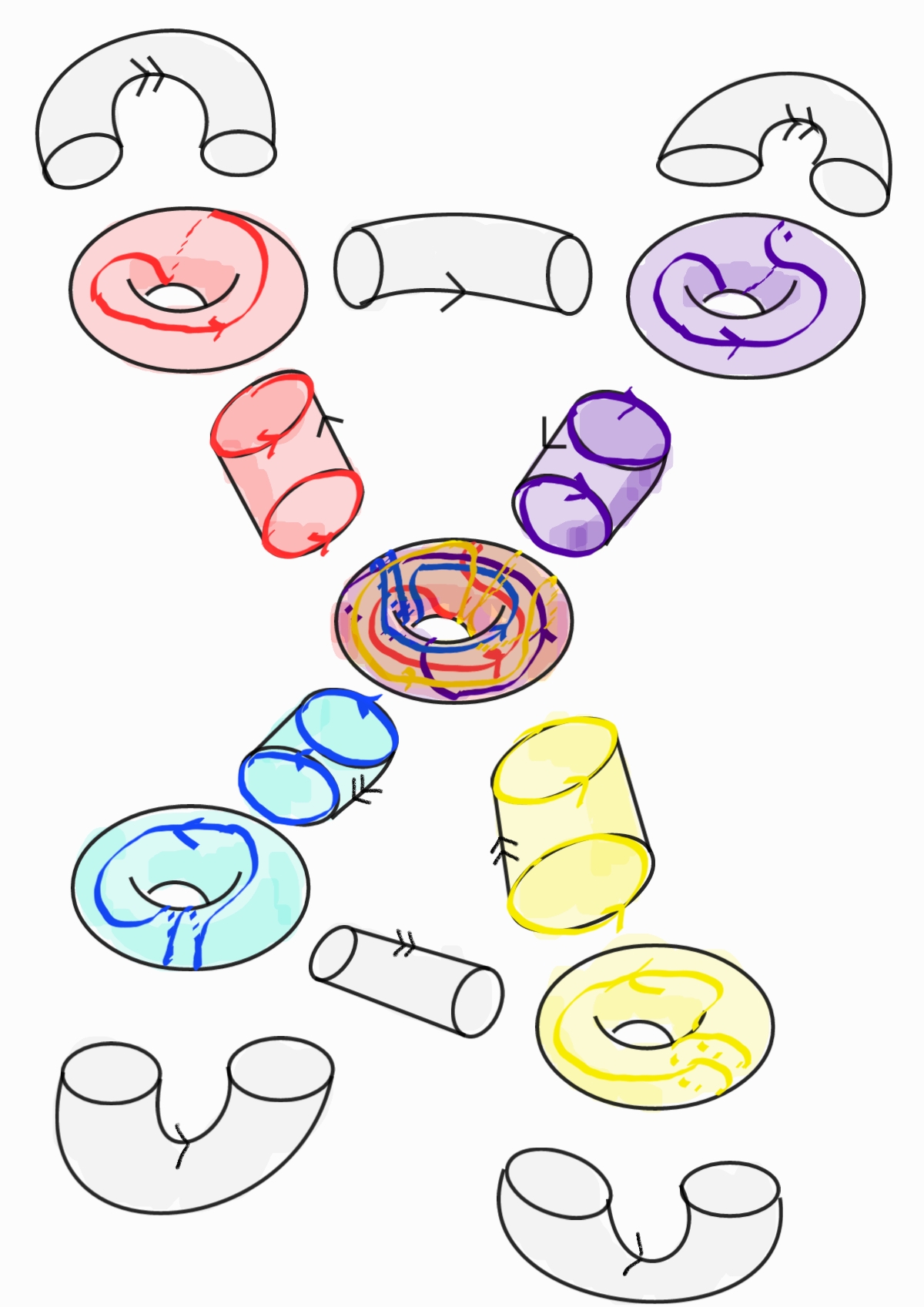}
    \caption{To see that Example \ref{exmp:cat(0)-non-hhg} is non-unbranched, it is convenient to work in this finite cover.  The colored cylinders correspond to cyclic subgroups whose centralizers are the fundamental groups of the four colored subspaces intersecting in the middle torus.  These centralizers are blocks because the incident cylinders in each torus are all attached along non-homotopic circles, as indicated by the colorful curves in the middle torus.}
    \label{fig:finite-cover}
\end{figure}

Since $a,t\in\kernel(\pi)\leq \maptor'$, the vertex space (a torus $T\subset X$ corresponding to $\langle a,t\rangle$) has five elevations $T_i,\ i\in\{0,1,2,3,4\}$ to the cover $X'\to X$, and the covering maps $T_i\to T$ all have degree one.  For $1\leq i\leq 4$, let $C_i\subset X$ consist of $T_0,T_i$, and the cylinder joining them.  The restriction of the covering map to $C_i$ induces a map $\pi_1C_i\to \maptor$ whose image is a subgroup isomorphic to $\F_2\times \integers$; let $B_i$ denote this subgroup.  Observe that $\langle a,t\rangle \leq B_1\cap B_2\cap B_3\cap B_4$.    

To conclude that $\maptor'$, and hence $\maptor$, is unbranched, it thus suffices to show that each $B_i$ is a block in $\maptor'$.  Let $g_i\in B_i$ generate the edge group corresponding to the unique edge space in $C_i$.  Since each $T_j$ contains four non-homotopic circles that are attaching circles of incident edge-cylinders, $B_i$ is the entire centralizer of $g_i$, and hence $B_i$ is a block, as required.
\end{example}

One way to produce a virtual HHG structure on a group is to show that it is virtually compact special, but the next example shows that one cannot recover all HHG structures provided by Theorem \ref{thm:main} in this way:

\begin{exmp}[HHG but not cocompactly cubulated]\label{exmp:not-via-specialness}
The group 
\[\langle a,b, t\mid tat^{-1}=a^{-1}, tbt^{-1}=a^{-1}b^{-1}\rangle \] has the form $\maptor=\F_2\rtimes_\Phi\integers$, where $\Phi$ is linearly growing.  The finite-index subgroup $\maptor'=\langle a,b,t^2\rangle_\maptor$ splits as a single HNN extension of $\integers^2$ with $\integers$ edge group, from which one can check (directly or using \cite{Wise:tubular,Woodhouse:special}) that $\maptor$ is virtually compact special, and hence a virtual HHG.  One can also check directly that the virtual splitting from Proposition \ref{prop:split-over-tori} has one black vertex, so Lemma \ref{lem:linear-unbranched} and Theorem \ref{thm:main} imply that $\maptor$ is an HHG.  On the other hand, using \cite[Lem. 7.1]{wu-ye:questions}, one verifies that $\maptor$  does not have the \emph{RFRS} property, so it is not special (compact or otherwise), just \emph{virtually} special.

However, slightly more complicated examples are HHGs but not even virtually cocompactly cubulated.  Indeed, \cite{HRSS} shows that $\maptor$ is an HHG if $\maptor$ is the fundamental group of the mapping torus of a multitwist on a (punctured) orientable surface.  On the other hand, if such a $\maptor$ is virtually cocompactly cubulated, then it is \emph{chargeless} \cite[Thm. B]{HagenPrzytycki}, and we can use the latter criterion to choose $\maptor$ with the desired properties. 

Here is an explicit example where $\maptor$ is not virtually cocompactly cubulated but it is the fundamental group of a graph manifold with boundary, and hence an HHG.  Let $S_0$ be a torus with one boundary component $b_0$, let $S_1$ be a torus with two boundary components $b_1$ and $c'$, and let $S=S_0\cup_{b_0=b_1} S_1$ be a once-punctured genus--$2$ surface; the image of $b_0,b_1$ in $S$ is a separating curve we denote $b$.  Let $\tau:S\to S$ be the Dehn twist about $b$.  Note that $\F=\pi_1S$ is a free group of rank $4$, and the induced homomorphism $\tau_*:\F\to \F$ is a linearly growing automorphism.  Let $\maptor$ be its mapping torus, which is the fundamental group of the mapping torus of $\tau$, a graph manifold $M$ with base surface $S$.  The decomposition of $\maptor$ from Proposition \ref{prop:split-over-tori} corresponds to the JSJ decomposition of $M$, which is obtained by cutting along the torus $b\times f\subset M$, where $f$ is a fiber.  The vertex spaces are $S_i\times f_i$, for $i\in\{0,1\}$, where $f_i$ is a fiber circle in the vertex space.  The edge maps $b\times f\to S_i\times f_i$ are given by $s\mapsto s_1,f\mapsto f_1$ and $s\mapsto s_0,f\mapsto f_0s_0$.  Since there is a unique edge incident to the vertex corresponding to $S_0\times f_0$, and $s_0$ vanishes in $H_1(S_0\times f_0,\integers)$, the image of $f$ in $S_0\times f_0$ represents the nontrivial homology class represented by $f_0$, so the charge of $M$ at $S_0\times f_0$ is nonvanishing (see \cite[Defn. 1.1]{HagenPrzytycki}).  By \cite[Thm. B]{HagenPrzytycki}, $\maptor$ is not virtually cocompactly cubulated. (However, since $M$ is a graph manifold with boundary, it has a nonpositively-curved Riemannian metric, and hence $\maptor$ is a (non-cubical) CAT(0) group.)  The existence of examples like this was pointed out in \cite{HRSS} but explicit ones were not given.
\end{exmp}

\bibliographystyle{alpha}
\bibliography{biblio}

@article{BFH-00,
    AUTHOR = {Bestvina, Mladen and Feighn, Mark and Handel, Michael},
     TITLE = {The {T}its alternative for {${\rm Out}(F_n)$}. {I}. {D}ynamics
              of exponentially-growing automorphisms},
   JOURNAL = {Ann. of Math. (2)},
  FJOURNAL = {Annals of Mathematics. Second Series},
    VOLUME = {151},
      YEAR = {2000},
    NUMBER = {2},
     PAGES = {517--623},
      ISSN = {0003-486X,1939-8980},
   MRCLASS = {20E36 (20E05 57M07)},
  MRNUMBER = {1765705},
MRREVIEWER = {K.\ Vogtmann},
       DOI = {10.2307/121043},
       URL = {https://doi.org/10.2307/121043},
}

@article {Bowditch:coarse-median,
    AUTHOR = {Bowditch, Brian H.},
     TITLE = {Coarse median spaces and groups},
   JOURNAL = {Pacific J. Math.},
  FJOURNAL = {Pacific Journal of Mathematics},
    VOLUME = {261},
      YEAR = {2013},
    NUMBER = {1},
     PAGES = {53--93},
      ISSN = {0030-8730,1945-5844},
   MRCLASS = {20F65},
  MRNUMBER = {3037559},
MRREVIEWER = {Vassilis\ Metaftsis},
       DOI = {10.2140/pjm.2013.261.53},
       URL = {https://doi.org/10.2140/pjm.2013.261.53},
}

@article {BOWtight,
    AUTHOR = {Bowditch, Brian H.},
     TITLE = {Tight geodesics in the curve complex},
   JOURNAL = {Invent. Math.},
  FJOURNAL = {Inventiones Mathematicae},
    VOLUME = {171},
      YEAR = {2008},
    NUMBER = {2},
     PAGES = {281--300},
      ISSN = {0020-9910},
     CODEN = {INVMBH},
   MRCLASS = {57M50 (20F65)},
  MRNUMBER = {2367021 (2008m:57040)},
MRREVIEWER = {Jason A. Behrstock},
       DOI = {10.1007/s00222-007-0081-y},
       URL = {http://dx.doi.org/10.1007/s00222-007-0081-y},
}

@article {RST,
    AUTHOR = {Russell, Jacob and Spriano, Davide and Tran, Hung Cong},
     TITLE = {Convexity in hierarchically hyperbolic spaces},
   JOURNAL = {Algebr. Geom. Topol.},
  FJOURNAL = {Algebraic \& Geometric Topology},
    VOLUME = {23},
      YEAR = {2023},
    NUMBER = {3},
     PAGES = {1167--1248},
      ISSN = {1472-2747,1472-2739},
   MRCLASS = {20F65 (20F67)},
  MRNUMBER = {4598806},
       DOI = {10.2140/agt.2023.23.1167},
       URL = {https://doi.org/10.2140/agt.2023.23.1167},
}

@article {RST-MLTG,
    AUTHOR = {Russell, Jacob and Spriano, Davide and Tran, Hung Cong},
     TITLE = {The local-to-global property for {M}orse quasi-geodesics},
   JOURNAL = {Math. Z.},
  FJOURNAL = {Mathematische Zeitschrift},
    VOLUME = {300},
      YEAR = {2022},
    NUMBER = {2},
     PAGES = {1557--1602},
      ISSN = {0025-5874,1432-1823},
   MRCLASS = {53C23 (57K20)},
  MRNUMBER = {4363788},
MRREVIEWER = {Josef\ Mike\v s},
       DOI = {10.1007/s00209-021-02811-w},
       URL = {https://doi.org/10.1007/s00209-021-02811-w},
}

@article {HHSIII,
    AUTHOR = {Behrstock, Jason and Hagen, Mark F. and Sisto, Alessandro},
     TITLE = {Asymptotic dimension and small-cancellation for hierarchically
              hyperbolic spaces and groups},
   JOURNAL = {Proc. Lond. Math. Soc. (3)},
  FJOURNAL = {Proceedings of the London Mathematical Society. Third Series},
    VOLUME = {114},
      YEAR = {2017},
    NUMBER = {5},
     PAGES = {890--926},
      ISSN = {0024-6115,1460-244X},
   MRCLASS = {20F65 (20F67 57M07)},
  MRNUMBER = {3653249},
MRREVIEWER = {Mahan\ Mj},
       DOI = {10.1112/plms.12026},
       URL = {https://doi.org/10.1112/plms.12026},
}

@conference{Banff,
    author = {Carolyn Abbott and Jason Behrstock and Jacob Russell},
    booktitle = {Advances in Hierarchical Hyperbolicity},
    title = {Problem List},
    url={https://www.wescac.net/HHG_Question_Session-dec-2024.pdf},
    year = {2024},
}

@article {AB,
    AUTHOR = {Alonso, Juan M. and Bridson, Martin R.},
     TITLE = {Semihyperbolic groups},
   JOURNAL = {Proc. London Math. Soc. (3)},
  FJOURNAL = {Proceedings of the London Mathematical Society. Third Series},
    VOLUME = {70},
      YEAR = {1995},
    NUMBER = {1},
     PAGES = {56--114},
      ISSN = {0024-6115,1460-244X},
   MRCLASS = {20F32 (20F10 53C23 57M07)},
  MRNUMBER = {1300841},
MRREVIEWER = {Athanase\ Papadopoulos},
       DOI = {10.1112/plms/s3-70.1.56},
       URL = {https://doi.org/10.1112/plms/s3-70.1.56},
}

@article {Hagen:colour,
    AUTHOR = {Hagen, Mark},
     TITLE = {Non-colorable hierarchically hyperbolic groups},
   JOURNAL = {Internat. J. Algebra Comput.},
  FJOURNAL = {International Journal of Algebra and Computation},
    VOLUME = {33},
      YEAR = {2023},
    NUMBER = {2},
     PAGES = {337--350},
      ISSN = {0218-1967,1793-6500},
   MRCLASS = {20F65 (20F67)},
  MRNUMBER = {4581212},
MRREVIEWER = {Jiawen\ Zhang},
       DOI = {10.1142/S0218196723500170},
       URL = {https://doi.org/10.1142/S0218196723500170},
}

@article {CCGHO:helly,
    AUTHOR = {Chalopin, J\'er\'emie and Chepoi, Victor and Genevois, Anthony
              and Hirai, Hiroshi and Osajda, Damian},
     TITLE = {Helly groups},
   JOURNAL = {Geom. Topol.},
  FJOURNAL = {Geometry \& Topology},
    VOLUME = {29},
      YEAR = {2025},
    NUMBER = {1},
     PAGES = {1--70},
      ISSN = {1465-3060,1364-0380},
   MRCLASS = {20F65 (05E18 20F06 20F67)},
  MRNUMBER = {4846637},
       DOI = {10.2140/gt.2025.29.1},
       URL = {https://doi.org/10.2140/gt.2025.29.1},
}

@article {NibloReeves:biautomatic,
    AUTHOR = {Niblo, G. A. and Reeves, L. D.},
     TITLE = {The geometry of cube complexes and the complexity of their
              fundamental groups},
   JOURNAL = {Topology},
  FJOURNAL = {Topology. An International Journal of Mathematics},
    VOLUME = {37},
      YEAR = {1998},
    NUMBER = {3},
     PAGES = {621--633},
      ISSN = {0040-9383},
   MRCLASS = {20F32 (20F55 57M50)},
  MRNUMBER = {1604899},
MRREVIEWER = {Darryl\ McCullough},
       DOI = {10.1016/S0040-9383(97)00018-9},
       URL = {https://doi.org/10.1016/S0040-9383(97)00018-9},
}

@article{CRHK,
  title={Real cubings and asymptotic cones of hierarchically hyperbolic groups},
  author={Casals-Ruiz, Montserrat and Hagen, Mark and Kazachkov, Ilya},
  journal={Preprint available at https://www. wescac. net/cones. html},
  pages={1--313},
  year={2024}
}

@article {HagenPrzytycki,
    AUTHOR = {Hagen, Mark F. and Przytycki, Piotr},
     TITLE = {Cocompactly cubulated graph manifolds},
   JOURNAL = {Israel J. Math.},
  FJOURNAL = {Israel Journal of Mathematics},
    VOLUME = {207},
      YEAR = {2015},
    NUMBER = {1},
     PAGES = {377--394},
      ISSN = {0021-2172,1565-8511},
   MRCLASS = {57M50 (57M05)},
  MRNUMBER = {3358051},
MRREVIEWER = {Yi\ Liu},
       DOI = {10.1007/s11856-015-1177-5},
       URL = {https://doi.org/10.1007/s11856-015-1177-5},
}

@article{Mangioni:Short-1,
  title={Short hierarchically hyperbolic groups I: uncountably many coarse median structures},
  author={Mangioni, Giorgio},
  journal={arXiv preprint arXiv:2410.09232},
  year={2024}
}

@article {CullerMorgan,
    AUTHOR = {Culler, Marc and Morgan, John W.},
     TITLE = {Group actions on {${\bf R}$}-trees},
   JOURNAL = {Proc. London Math. Soc. (3)},
  FJOURNAL = {Proceedings of the London Mathematical Society. Third Series},
    VOLUME = {55},
      YEAR = {1987},
    NUMBER = {3},
     PAGES = {571--604},
      ISSN = {0024-6115,1460-244X},
   MRCLASS = {20F32 (57N10)},
  MRNUMBER = {907233},
MRREVIEWER = {Roger\ C.\ Alperin},
       DOI = {10.1112/plms/s3-55.3.571},
       URL = {https://doi-org.bris.idm.oclc.org/10.1112/plms/s3-55.3.571},
}

@article{DMS2,
  title={Asymptotically {CAT}(0) metrics, Z-structures, and the Farrell-Jones Conjecture},
  author={Durham, Matthew Gentry and Minsky, Yair and Sisto, Alessandro},
  journal={arXiv preprint arXiv:2504.17048},
  year={2025},
}

@article{Durham:cubes,
  title={Cubulating infinity in hierarchically hyperbolic spaces},
  author={Durham, Matthew Gentry},
  journal={arXiv preprint arXiv:2308.13689},
  year={2023}
}

@article {DMS,
    AUTHOR = {Durham, Matthew G. and Minsky, Yair N. and Sisto, Alessandro},
     TITLE = {Stable cubulations, bicombings, and barycenters},
   JOURNAL = {Geom. Topol.},
  FJOURNAL = {Geometry \& Topology},
    VOLUME = {27},
      YEAR = {2023},
    NUMBER = {6},
     PAGES = {2383--2478},
      ISSN = {1465-3060,1364-0380},
   MRCLASS = {20F65 (57K20)},
  MRNUMBER = {4634751},
MRREVIEWER = {Jiawen\ Zhang},
       DOI = {10.2140/gt.2023.27.2383},
       URL = {https://doi.org/10.2140/gt.2023.27.2383},
}

@article {HHSI,
    AUTHOR = {Behrstock, Jason and Hagen, Mark F. and Sisto, Alessandro},
     TITLE = {Hierarchically hyperbolic spaces, {I}: {C}urve complexes for
              cubical groups},
   JOURNAL = {Geom. Topol.},
  FJOURNAL = {Geometry \& Topology},
    VOLUME = {21},
      YEAR = {2017},
    NUMBER = {3},
     PAGES = {1731--1804},
      ISSN = {1465-3060,1364-0380},
   MRCLASS = {20F36 (20F55 20F65)},
  MRNUMBER = {3650081},
MRREVIEWER = {Nadia\ Benakli},
       DOI = {10.2140/gt.2017.21.1731},
       URL = {https://doi.org/10.2140/gt.2017.21.1731},
}

@article {Bowditch:invariance,
    AUTHOR = {Bowditch, Brian H.},
     TITLE = {Invariance of coarse median spaces under relative
              hyperbolicity},
   JOURNAL = {Math. Proc. Cambridge Philos. Soc.},
  FJOURNAL = {Mathematical Proceedings of the Cambridge Philosophical
              Society},
    VOLUME = {154},
      YEAR = {2013},
    NUMBER = {1},
     PAGES = {85--95},
      ISSN = {0305-0041,1469-8064},
   MRCLASS = {20F67},
  MRNUMBER = {3002585},
       DOI = {10.1017/S0305004112000382},
       URL = {https://doi.org/10.1017/S0305004112000382},
}

@misc{Pet,
    title = {Mapping class groups are quasicubical},
    author = {Petyt, Harry},
    year={2024},
    howpublished = {arXiv:2112.10681 [math.MG]},
    note= {\emph{To appear in Amer. J. Math}},
}

@article {HRSS,
   author = {Hagen, Mark and Russell, Jacob and Sisto, Alessandro and Spriano, Davide},
     title = {Equivariant hierarchically hyperbolic structures for 3-manifold groups via quasimorphisms},
     journal = {Annales de l'Institut Fourier},
     publisher = {Association des Annales de l{\textquoteright}institut Fourier},
     year = {2024},
     doi = {10.5802/aif.3654},
     language = {en},
     note = {Online first},
}

@article {HHSII,
    AUTHOR = {Behrstock, Jason and Hagen, Mark and Sisto, Alessandro},
     TITLE = {Hierarchically hyperbolic spaces {II}: {C}ombination theorems and the distance formula},
   JOURNAL = {Pacific J. Math.},
  FJOURNAL = {Pacific Journal of Mathematics},
    VOLUME = {299},
      YEAR = {2019},
    NUMBER = {2},
     PAGES = {257--338},
      ISSN = {0030-8730,1945-5844},
   MRCLASS = {20F36 (20F65 20F67)},
  MRNUMBER = {3956144},
MRREVIEWER = {Jiming\ Ma},
       DOI = {10.2140/pjm.2019.299.257},
    }

@article{MunroPetyt,
  title={Coarse obstructions to cocompact cubulation},
  author={Munro, Zachary and Petyt, Harry},
  journal={arXiv preprint arXiv:2407.09275},
  year={2024}
}

@article {CrokeKleiner,
    AUTHOR = {Croke, C. B. and Kleiner, B.},
     TITLE = {The geodesic flow of a nonpositively curved graph manifold},
   JOURNAL = {Geom. Funct. Anal.},
  FJOURNAL = {Geometric and Functional Analysis},
    VOLUME = {12},
      YEAR = {2002},
    NUMBER = {3},
     PAGES = {479--545},
      ISSN = {1016-443X,1420-8970},
   MRCLASS = {53C24 (37D40 53C23 53D25)},
  MRNUMBER = {1924370},
MRREVIEWER = {David\ Michael\ Fisher},
       DOI = {10.1007/s00039-002-8255-7},
       URL = {https://doi.org/10.1007/s00039-002-8255-7},
}

@article {CohenLustig,
    AUTHOR = {Cohen, Marshall M. and Lustig, Martin},
     TITLE = {Very small group actions on {${\bf R}$}-trees and {D}ehn twist
              automorphisms},
   JOURNAL = {Topology},
  FJOURNAL = {Topology. An International Journal of Mathematics},
    VOLUME = {34},
      YEAR = {1995},
    NUMBER = {3},
     PAGES = {575--617},
      ISSN = {0040-9383},
   MRCLASS = {20F32 (20E05 20E08 20F28)},
  MRNUMBER = {1341810},
MRREVIEWER = {Gilbert\ Levitt},
       DOI = {10.1016/0040-9383(94)00038-M},
       URL = {https://doi.org/10.1016/0040-9383(94)00038-M},
}

@article {HaglundWise:special,
    AUTHOR = {Haglund, Fr\'{e}d\'{e}ric and Wise, Daniel T.},
     TITLE = {Special cube complexes},
   JOURNAL = {Geom. Funct. Anal.},
  FJOURNAL = {Geometric and Functional Analysis},
    VOLUME = {17},
      YEAR = {2008},
    NUMBER = {5},
     PAGES = {1551--1620},
      ISSN = {1016-443X,1420-8970},
   MRCLASS = {20F36 (20F55 20F67)},
  MRNUMBER = {2377497},
MRREVIEWER = {Patrick\ Bahls},
       DOI = {10.1007/s00039-007-0629-4},
       URL = {https://doi.org/10.1007/s00039-007-0629-4},
}

@article {Wise:tubular,
    AUTHOR = {Wise, Daniel T.},
     TITLE = {Cubular tubular groups},
   JOURNAL = {Trans. Amer. Math. Soc.},
  FJOURNAL = {Transactions of the American Mathematical Society},
    VOLUME = {366},
      YEAR = {2014},
    NUMBER = {10},
     PAGES = {5503--5521},
      ISSN = {0002-9947,1088-6850},
   MRCLASS = {20F67 (20E06 20F65)},
  MRNUMBER = {3240932},
MRREVIEWER = {Michael\ L.\ Mihalik},
       DOI = {10.1090/S0002-9947-2014-06065-0},
       URL = {https://doi.org/10.1090/S0002-9947-2014-06065-0},
}

@article {Ghosh,
    AUTHOR = {Ghosh, Pritam},
     TITLE = {Relative hyperbolicity of free-by-cyclic extensions},
   JOURNAL = {Compos. Math.},
  FJOURNAL = {Compositio Mathematica},
    VOLUME = {159},
      YEAR = {2023},
    NUMBER = {1},
     PAGES = {153--183},
      ISSN = {0010-437X,1570-5846},
   MRCLASS = {20F65 (20E05 20F67 57M07)},
  MRNUMBER = {4541452},
MRREVIEWER = {Anthony\ Genevois},
       DOI = {10.1112/S0010437X22007813},
       URL = {https://doi.org/10.1112/S0010437X22007813},
}

@article{AndrewHughesKudlinska,
  title={Torsion homology growth of polynomially growing free-by-cyclic groups},
  author={Andrew, Naomi and Hughes, Sam and Kudlinska, Monika},
  journal={Rocky Mountain Journal of Mathematics},
  volume={54},
  number={4},
  pages={933--941},
  year={2024},
  publisher={Rocky Mountain Mathematics Consortium Tempe, AZ, USA}
}

@article {BGGGS,
    AUTHOR = {Beeker, Benjamin and Cordes, Matthew and Gardam, Giles and
              Gupta, Radhika and Stark, Emily},
     TITLE = {Cannon-{T}hurston maps for {${\rm CAT}(0)$} groups with
              isolated flats},
   JOURNAL = {Math. Ann.},
  FJOURNAL = {Mathematische Annalen},
    VOLUME = {384},
      YEAR = {2022},
    NUMBER = {1-2},
     PAGES = {963--987},
      ISSN = {0025-5831,1432-1807},
   MRCLASS = {20F65 (20E07 20F67 57K31 57M07)},
  MRNUMBER = {4476223},
MRREVIEWER = {Jingyin\ Huang},
       DOI = {10.1007/s00208-021-02245-z},
       URL = {https://doi.org/10.1007/s00208-021-02245-z},
}

@article{BowRH,

    AUTHOR = {Bowditch, B. H.},
     TITLE = {Relatively hyperbolic groups},
   JOURNAL = {Internat. J. Algebra Comput.},
  FJOURNAL = {International Journal of Algebra and Computation},
    VOLUME = {22},
      YEAR = {2012},
    NUMBER = {3},
     PAGES = {1250016, 66},
      ISSN = {0218-1967,1793-6500},
   MRCLASS = {20F67 (20F65)},
  MRNUMBER = {2922380},
MRREVIEWER = {R\'{e}mi\ Bernard\ Coulon},
       DOI = {10.1142/S0218196712500166},
       URL = {https://doi.org/10.1142/S0218196712500166},
}

@article{Gersten,
  title={The automorphism group of a free group is not a {CAT}(0) group},
  author={Gersten, Stephen M},
  journal={Proceedings of the American Mathematical Society},
  volume={121},
  number={4},
  pages={999--1002},
  year={1994}
}

@article {Button:aspects,
    AUTHOR = {Button, Jack Oliver},
     TITLE = {Aspects of non positive curvature for linear groups with no
              infinite order unipotents},
   JOURNAL = {Groups Geom. Dyn.},
  FJOURNAL = {Groups, Geometry, and Dynamics},
    VOLUME = {13},
      YEAR = {2019},
    NUMBER = {1},
     PAGES = {277--292},
      ISSN = {1661-7207,1661-7215},
   MRCLASS = {20F67 (20F65)},
  MRNUMBER = {3900771},
MRREVIEWER = {Vassilis\ Metaftsis},
       DOI = {10.4171/GGD/484},
       URL = {https://doi.org/10.4171/GGD/484},
}

@article {BDM,
    AUTHOR = {Behrstock, Jason and Dru{\c{t}}u, Cornelia and Mosher, Lee},
     TITLE = {Thick metric spaces, relative hyperbolicity, and
              quasi-isometric rigidity},
   JOURNAL = {Math. Ann.},
  FJOURNAL = {Mathematische Annalen},
    VOLUME = {344},
      YEAR = {2009},
    NUMBER = {3},
     PAGES = {543--595},
      ISSN = {0025-5831,1432-1807},
   MRCLASS = {20F67 (57M07 57M50)},
  MRNUMBER = {2501302},
       DOI = {10.1007/s00208-008-0317-1},
       URL = {https://doi.org/10.1007/s00208-008-0317-1},
}

@article{Mutanguha,
  title={On polynomial free-by-cyclic groups},
  author={Mutanguha, Jean Pierre},
    journal={Preprint},
url={https://mutanguha.com/pdfs/hierarchy.pdf},
  year={2024}
}

@article{DahmaniTouikan,
  title={Unipotent linear suspensions of free groups},
  author={Dahmani, Fran{\c{c}}ois and Touikan, Nicholas},
  journal={arXiv preprint arXiv:2305.11274},
  year={2023}
}

@article{KudlinskaValiunas,
  title={Free-by-cyclic groups are equationally Noetherian},
  author={Kudlinska, Monika and Valiunas, Motiejus},
  journal={arXiv preprint arXiv:2407.08809},
  year={2024}
}

@article {Haettel:lattices,
    AUTHOR = {Haettel, Thomas},
     TITLE = {Higher rank lattices are not coarse median},
   JOURNAL = {Algebr. Geom. Topol.},
  FJOURNAL = {Algebraic \& Geometric Topology},
    VOLUME = {16},
      YEAR = {2016},
    NUMBER = {5},
     PAGES = {2895--2910},
      ISSN = {1472-2747,1472-2739},
   MRCLASS = {20F65 (22E40 51E24 53C35)},
  MRNUMBER = {3572353},
MRREVIEWER = {Igor\ Belegradek},
       DOI = {10.2140/agt.2016.16.2895},
       URL = {https://doi.org/10.2140/agt.2016.16.2895},
}

@article {FeighHandel:recognition,
    AUTHOR = {Feighn, Mark and Handel, Michael},
     TITLE = {The recognition theorem for {${\rm Out}(F_n)$}},
   JOURNAL = {Groups Geom. Dyn.},
  FJOURNAL = {Groups, Geometry, and Dynamics},
    VOLUME = {5},
      YEAR = {2011},
    NUMBER = {1},
     PAGES = {39--106},
      ISSN = {1661-7207,1661-7215},
   MRCLASS = {20E05 (20E36)},
  MRNUMBER = {2763779},
MRREVIEWER = {S.\ Andreadakis},
       DOI = {10.4171/GGD/116},
       URL = {https://doi.org/10.4171/GGD/116},
}

@article {HruskaRuane,
    AUTHOR = {Hruska, G. Christopher and Ruane, Kim},
     TITLE = {Connectedness properties and splittings of groups with
              isolated flats},
   JOURNAL = {Algebr. Geom. Topol.},
  FJOURNAL = {Algebraic \& Geometric Topology},
    VOLUME = {21},
      YEAR = {2021},
    NUMBER = {2},
     PAGES = {755--799},
      ISSN = {1472-2747,1472-2739},
   MRCLASS = {20E08 (20F67)},
  MRNUMBER = {4250516},
MRREVIEWER = {Denis\ E.\ Serbin},
       DOI = {10.2140/agt.2021.21.755},
       URL = {https://doi.org/10.2140/agt.2021.21.755},
}

@article{HM-20,

    AUTHOR = {Handel, Michael and Mosher, Lee},
     TITLE = {Subgroup decomposition in {${\rm Out}(F_n)$}},
   JOURNAL = {Mem. Amer. Math. Soc.},
  FJOURNAL = {Memoirs of the American Mathematical Society},
    VOLUME = {264},
      YEAR = {2020},
    NUMBER = {1280},
     PAGES = {vii+276},
      ISSN = {0065-9266,1947-6221},
      ISBN = {978-1-4704-4113-5; 978-1-4704-5802-7},
   MRCLASS = {20F28 (20E05 20F65 57M07)},
  MRNUMBER = {4089372},
MRREVIEWER = {Thomas\ Koberda},
       DOI = {10.1090/memo/1280},
       URL = {https://doi.org/10.1090/memo/1280},
}

@article {BFH:kolchin,
    AUTHOR = {Bestvina, Mladen and Feighn, Mark and Handel, Michael},
     TITLE = {The {T}its alternative for {${\rm Out}(F_n)$}. {II}. {A}
              {K}olchin type theorem},
   JOURNAL = {Ann. of Math. (2)},
  FJOURNAL = {Annals of Mathematics. Second Series},
    VOLUME = {161},
      YEAR = {2005},
    NUMBER = {1},
     PAGES = {1--59},
      ISSN = {0003-486X,1939-8980},
   MRCLASS = {20E36 (20E05 57M07)},
  MRNUMBER = {2150382},
MRREVIEWER = {S.\ Andreadakis},
       DOI = {10.4007/annals.2005.161.1},
       URL = {https://doi.org/10.4007/annals.2005.161.1},
}

@article{BradyBridson,
  title={On the absence of biautomaticity in certain graphs of abelian groups},
  author={Brady, N and Bridson, MR},
  journal={preprint},
  year={1998}
}

@article {BestvinaHandel:ttm,
    AUTHOR = {Bestvina, Mladen and Handel, Michael},
     TITLE = {Train tracks and automorphisms of free groups},
   JOURNAL = {Ann. of Math. (2)},
  FJOURNAL = {Annals of Mathematics. Second Series},
    VOLUME = {135},
      YEAR = {1992},
    NUMBER = {1},
     PAGES = {1--51},
      ISSN = {0003-486X,1939-8980},
   MRCLASS = {20E05 (20F28 20F32 20F34 57M07 57M15)},
  MRNUMBER = {1147956},
MRREVIEWER = {Darryl\ McCullough},
       DOI = {10.2307/2946562},
       URL = {https://doi.org/10.2307/2946562},
}

@article {AndrewMartino:splitting,
    AUTHOR = {Andrew, Naomi and Martino, Armando},
     TITLE = {Free-by-cyclic groups, automorphisms and actions on nearly
              canonical trees},
   JOURNAL = {J. Algebra},
  FJOURNAL = {Journal of Algebra},
    VOLUME = {604},
      YEAR = {2022},
     PAGES = {451--495},
      ISSN = {0021-8693,1090-266X},
   MRCLASS = {20E08 (20E05 20E36 20F28 20F65)},
  MRNUMBER = {4414821},
MRREVIEWER = {Edgar\ A.\ Bering, IV},
       DOI = {10.1016/j.jalgebra.2022.03.033},
       URL = {https://doi.org/10.1016/j.jalgebra.2022.03.033},
}

@article{Br-00,
	 AUTHOR = {Brinkmann, P.},
     TITLE = {Hyperbolic automorphisms of free groups},
   JOURNAL = {Geom. Funct. Anal.},
  FJOURNAL = {Geometric and Functional Analysis},
    VOLUME = {10},
      YEAR = {2000},
    NUMBER = {5},
     PAGES = {1071--1089},
      ISSN = {1016-443X,1420-8970},
   MRCLASS = {20F65 (20E36 57M07)},
  MRNUMBER = {1800064},
MRREVIEWER = {Katalin\ A.\ Bencsath},
       DOI = {10.1007/PL00001647},
       URL = {https://doi.org/10.1007/PL00001647},
}

@article {Macura:detour,
    AUTHOR = {Macura, Nata\v{s}a},
     TITLE = {Detour functions and quasi-isometries},
   JOURNAL = {Q. J. Math.},
  FJOURNAL = {The Quarterly Journal of Mathematics},
    VOLUME = {53},
      YEAR = {2002},
    NUMBER = {2},
     PAGES = {207--239},
      ISSN = {0033-5606,1464-3847},
   MRCLASS = {20F65 (20F69)},
  MRNUMBER = {1909513},
MRREVIEWER = {Athanase\ Papadopoulos},
       DOI = {10.1093/qjmath/53.2.207},
       URL = {https://doi.org/10.1093/qjmath/53.2.207},
}

@book {BH:book,
    AUTHOR = {Bridson, Martin R. and Haefliger, Andr\'e},
     TITLE = {Metric spaces of non-positive curvature},
    SERIES = {Grundlehren der mathematischen Wissenschaften [Fundamental
              Principles of Mathematical Sciences]},
    VOLUME = {319},
 PUBLISHER = {Springer-Verlag, Berlin},
      YEAR = {1999},
     PAGES = {xxii+643},
      ISBN = {3-540-64324-9},
   MRCLASS = {53C23 (20F65 53C70 57M07)},
  MRNUMBER = {1744486},
MRREVIEWER = {Athanase\ Papadopoulos},
       DOI = {10.1007/978-3-662-12494-9},
       URL = {https://doi.org/10.1007/978-3-662-12494-9},
}

@article{Bongiovanni:Veech,
  title={Extensions of finitely generated Veech groups},
  author={Bongiovanni, Eliot},
  journal={arXiv preprint arXiv:2406.11090},
  year={2024}
}

@article{KudPetyt,
  title={Largest acylindrical actions of free-by-cyclic groups},
  author={Kudlinska, Monika and Petyt, Harry},
  journal={arXiv preprint arXiv:2512.05293},
  year={2025}
}

@article {Hagen:thickness,
    AUTHOR = {Hagen, Mark},
     TITLE = {A remark on thickness of free-by-cyclic groups},
   JOURNAL = {Illinois J. Math.},
  FJOURNAL = {Illinois Journal of Mathematics},
    VOLUME = {63},
      YEAR = {2019},
    NUMBER = {4},
     PAGES = {633--643},
      ISSN = {0019-2082,1945-6581},
   MRCLASS = {20F65},
  MRNUMBER = {4032818},
MRREVIEWER = {Enric\ Ventura Capell},
       DOI = {10.1215/00192082-7917878},
       URL = {https://doi.org/10.1215/00192082-7917878},
}

@article {Button:tubular-cubes,
    AUTHOR = {Button, Jack},
     TITLE = {Tubular free by cyclic groups act freely on {CAT}(0) cube complexes},
   JOURNAL = {Canad. Math. Bull.},
  FJOURNAL = {Canadian Mathematical Bulletin. Bulletin Canadien de
              Math\'{e}matiques},
    VOLUME = {60},
      YEAR = {2017},
    NUMBER = {1},
     PAGES = {54--62},
      ISSN = {0008-4395,1496-4287},
   MRCLASS = {20F65 (20E08 20E22 20F67)},
  MRNUMBER = {3612098},
MRREVIEWER = {Mark\ F.\ Hagen},
       DOI = {10.4153/CMB-2016-074-0},
       URL = {https://doi.org/10.4153/CMB-2016-074-0},
}

@article {DHSbound,
    AUTHOR = {Durham, Matthew Gentry and Hagen, Mark F. and Sisto,
              Alessandro},
     TITLE = {Boundaries and automorphisms of hierarchically hyperbolic
              spaces},
   JOURNAL = {Geom. Topol.},
  FJOURNAL = {Geometry \& Topology},
    VOLUME = {21},
      YEAR = {2017},
    NUMBER = {6},
     PAGES = {3659--3758},
      ISSN = {1465-3060,1364-0380},
   MRCLASS = {20F65 (20F67 30F60)},
  MRNUMBER = {3693574},
MRREVIEWER = {Nadia\ Benakli},
       DOI = {10.2140/gt.2017.21.3659},
       URL = {https://doi.org/10.2140/gt.2017.21.3659},
}

@article {HagenSusse,
    AUTHOR = {Hagen, Mark F. and Susse, Tim},
     TITLE = {On hierarchical hyperbolicity of cubical groups},
   JOURNAL = {Israel J. Math.},
  FJOURNAL = {Israel Journal of Mathematics},
    VOLUME = {236},
      YEAR = {2020},
    NUMBER = {1},
     PAGES = {45--89},
      ISSN = {0021-2172,1565-8511},
   MRCLASS = {20F65 (20F36 20F67 57M07)},
  MRNUMBER = {4093881},
MRREVIEWER = {Alessandro\ Sisto},
       DOI = {10.1007/s11856-020-1967-2},
       URL = {https://doi.org/10.1007/s11856-020-1967-2},
}

@article{haettel_coarse_2020,
  
 AUTHOR = {Haettel, Thomas and Hoda, Nima and Petyt, Harry},
     TITLE = {Coarse injectivity, hierarchical hyperbolicity and
              semihyperbolicity},
   JOURNAL = {Geom. Topol.},
  FJOURNAL = {Geometry \& Topology},
    VOLUME = {27},
      YEAR = {2023},
    NUMBER = {4},
     PAGES = {1587--1633},
      ISSN = {1465-3060,1364-0380},
   MRCLASS = {20F65 (20F67 51F30)},
  MRNUMBER = {4602421},
       DOI = {10.2140/gt.2023.27.1587},
       URL = {https://doi.org/10.2140/gt.2023.27.1587},
}

@article {BerRob,
    AUTHOR = {Berlai, Federico and Robbio, Bruno},
     TITLE = {A refined combination theorem for hierarchically hyperbolic
              groups},
   JOURNAL = {Groups Geom. Dyn.},
  FJOURNAL = {Groups, Geometry, and Dynamics},
    VOLUME = {14},
      YEAR = {2020},
    NUMBER = {4},
     PAGES = {1127--1203},
      ISSN = {1661-7207,1661-7215},
   MRCLASS = {20F65 (20E06 20F67)},
  MRNUMBER = {4186470},
MRREVIEWER = {Brendan\ Burns\ Healy},
       DOI = {10.4171/ggd/576},
       URL = {https://doi.org/10.4171/ggd/576},
}

@article {DHScorr,
    AUTHOR = {Durham, Matthew Gentry and Hagen, Mark F. and Sisto,
              Alessandro},
     TITLE = {Correction to the article {B}oundaries and automorphisms of
              hierarchically hyperbolic spaces},
   JOURNAL = {Geom. Topol.},
  FJOURNAL = {Geometry \& Topology},
    VOLUME = {24},
      YEAR = {2020},
    NUMBER = {2},
     PAGES = {1051--1073},
      ISSN = {1465-3060,1364-0380},
   MRCLASS = {20F65 (20F67 30F60)},
  MRNUMBER = {4153656},
       DOI = {10.2140/gt.2020.24.1051},
       URL = {https://doi.org/10.2140/gt.2020.24.1051},
}

@article {HMSxl,
    AUTHOR = {Hagen, Mark and Martin, Alexandre and Sisto, Alessandro},
     TITLE = {Extra-large type {A}rtin groups are hierarchically hyperbolic},
   JOURNAL = {Math. Ann.},
  FJOURNAL = {Mathematische Annalen},
    VOLUME = {388},
      YEAR = {2024},
    NUMBER = {1},
     PAGES = {867--938},
      ISSN = {0025-5831,1432-1807},
   MRCLASS = {20F65 (05E16 20F36 20F67)},
  MRNUMBER = {4693949},
MRREVIEWER = {Matthew\ C. B. Zaremsky},
       DOI = {10.1007/s00208-022-02523-4},
       URL = {https://doi.org/10.1007/s00208-022-02523-4},
}

@article {PetytSpriano,
    AUTHOR = {Petyt, Harry and Spriano, Davide},
     TITLE = {Unbounded domains in hierarchically hyperbolic groups},
   JOURNAL = {Groups Geom. Dyn.},
  FJOURNAL = {Groups, Geometry, and Dynamics},
    VOLUME = {17},
      YEAR = {2023},
    NUMBER = {2},
     PAGES = {479--500},
      ISSN = {1661-7207,1661-7215},
   MRCLASS = {20F65 (20E99 20F67 51F30)},
  MRNUMBER = {4584673},
MRREVIEWER = {Bruno\ P.\ Zimmermann},
       DOI = {10.4171/ggd/706},
       URL = {https://doi.org/10.4171/ggd/706},
}

@article {BridsonGroves,
    AUTHOR = {Bridson, Martin R. and Groves, Daniel},
     TITLE = {The quadratic isoperimetric inequality for mapping tori of
              free group automorphisms},
   JOURNAL = {Mem. Amer. Math. Soc.},
  FJOURNAL = {Memoirs of the American Mathematical Society},
    VOLUME = {203},
      YEAR = {2010},
    NUMBER = {955},
     PAGES = {xii+152},
      ISSN = {0065-9266,1947-6221},
      ISBN = {978-0-8218-4631-5},
   MRCLASS = {20F65 (20E36 20F06 57M07)},
  MRNUMBER = {2590896},
MRREVIEWER = {Ian\ M.\ Chiswell},
       DOI = {10.1090/S0065-9266-09-00578-X},
       URL = {https://doi.org/10.1090/S0065-9266-09-00578-X},
}

@article{HagenWise:new-polynomial,
author={Hagen, Mark and Daniel T. Wise},
title={Cubulating mapping tori of (some) polynomial free group automorphisms},
journal={Ar{X}iv 1605.07879},
pages={1--45},
year={2025},
}

@article {Osin:acyl,
    AUTHOR = {Osin, D.},
     TITLE = {Acylindrically hyperbolic groups},
   JOURNAL = {Trans. Amer. Math. Soc.},
  FJOURNAL = {Transactions of the American Mathematical Society},
    VOLUME = {368},
      YEAR = {2016},
    NUMBER = {2},
     PAGES = {851--888},
      ISSN = {0002-9947,1088-6850},
   MRCLASS = {20F67 (20F65)},
  MRNUMBER = {3430352},
MRREVIEWER = {Alessandro\ Sisto},
       DOI = {10.1090/tran/6343},
       URL = {https://doi.org/10.1090/tran/6343},
}

@article {HHS:quasiflats,
    AUTHOR = {Behrstock, Jason and Hagen, Mark F. and Sisto, Alessandro},
     TITLE = {Quasiflats in hierarchically hyperbolic spaces},
   JOURNAL = {Duke Math. J.},
  FJOURNAL = {Duke Mathematical Journal},
    VOLUME = {170},
      YEAR = {2021},
    NUMBER = {5},
     PAGES = {909--996},
      ISSN = {0012-7094,1547-7398},
   MRCLASS = {20F67 (20F36 20F55 30F60 53C23)},
  MRNUMBER = {4255047},
MRREVIEWER = {Yasushi\ Yamashita},
       DOI = {10.1215/00127094-2020-0056},
       URL = {https://doi.org/10.1215/00127094-2020-0056},
}

@article {MitraCT,
    AUTHOR = {Mitra, Mahan},
     TITLE = {Cannon-{T}hurston maps for hyperbolic group extensions},
   JOURNAL = {Topology},
  FJOURNAL = {Topology. An International Journal of Mathematics},
    VOLUME = {37},
      YEAR = {1998},
    NUMBER = {3},
     PAGES = {527--538},
      ISSN = {0040-9383},
   MRCLASS = {57M07 (20F32 57M10)},
  MRNUMBER = {1604882},
MRREVIEWER = {Sadayoshi Kojima},
       DOI = {10.1016/S0040-9383(97)00036-0},
       URL = {http://dx.doi.org/10.1016/S0040-9383(97)00036-0},
}

@article {MM2,
    AUTHOR = {Masur, H. A. and Minsky, Y. N.},
     TITLE = {Geometry of the complex of curves. {II}. {H}ierarchical
              structure},
   JOURNAL = {Geom. Funct. Anal.},
  FJOURNAL = {Geometric and Functional Analysis},
    VOLUME = {10},
      YEAR = {2000},
    NUMBER = {4},
     PAGES = {902--974},
      ISSN = {1016-443X},
     CODEN = {GFANFB},
   MRCLASS = {57M50 (30F60 32G15)},
  MRNUMBER = {1791145 (2001k:57020)},
MRREVIEWER = {Darryl McCullough},
       DOI = {10.1007/PL00001643},
       URL = {http://dx.doi.org/10.1007/PL00001643},
}

@article{CHRK,
    author = {Casals-Ruiz, Montserrat and Hagen, Mark and Kazachkov, Ilya},
    title = {Real cubings and asymptotic cones of hierarchically hyperbolic groups},
    journal = {Preprint},
url={https://www.wescac.net/cones_july_2024-public.pdf},
    year = {2025},
}

@article {Woodhouse:special,
    AUTHOR = {Woodhouse, Daniel J.},
     TITLE = {Classifying virtually special tubular groups},
   JOURNAL = {Groups Geom. Dyn.},
  FJOURNAL = {Groups, Geometry, and Dynamics},
    VOLUME = {12},
      YEAR = {2018},
    NUMBER = {2},
     PAGES = {679--702},
      ISSN = {1661-7207,1661-7215},
   MRCLASS = {20F67 (20E06 20F65)},
  MRNUMBER = {3813206},
MRREVIEWER = {Christopher\ H.\ Cashen},
       DOI = {10.4171/GGD/452},
       URL = {https://doi.org/10.4171/GGD/452},
}

@article {DahmaniLi,
    AUTHOR = {Dahmani, Fran{\c{c}}ois and Li, Ruoyu},
     TITLE = {Relative hyperbolicity for automorphisms of free products and
              free groups},
   JOURNAL = {J. Topol. Anal.},
  FJOURNAL = {Journal of Topology and Analysis},
    VOLUME = {14},
      YEAR = {2022},
    NUMBER = {1},
     PAGES = {55--92},
      ISSN = {1793-5253,1793-7167},
   MRCLASS = {20F67 (20F34 20F65)},
  MRNUMBER = {4411100},
MRREVIEWER = {Thomas\ Koberda},
       DOI = {10.1142/S1793525321500011},
       URL = {https://doi.org/10.1142/S1793525321500011},
}

@article {HagenWise:irred,
    AUTHOR = {Hagen, Mark F. and Wise, Daniel T.},
     TITLE = {Cubulating hyperbolic free-by-cyclic groups: the irreducible
              case},
   JOURNAL = {Duke Math. J.},
  FJOURNAL = {Duke Mathematical Journal},
    VOLUME = {165},
      YEAR = {2016},
    NUMBER = {9},
     PAGES = {1753--1813},
      ISSN = {0012-7094,1547-7398},
   MRCLASS = {20F65 (20E06 20F67 57M20)},
  MRNUMBER = {3513573},
MRREVIEWER = {Vassilis\ Metaftsis},
       DOI = {10.1215/00127094-3450752},
       URL = {https://doi.org/10.1215/00127094-3450752},
}

@article {HagenWise:general,
    AUTHOR = {Hagen, Mark F. and Wise, Daniel T.},
     TITLE = {Cubulating hyperbolic free-by-cyclic groups: the general case},
   JOURNAL = {Geom. Funct. Anal.},
  FJOURNAL = {Geometric and Functional Analysis},
    VOLUME = {25},
      YEAR = {2015},
    NUMBER = {1},
     PAGES = {134--179},
      ISSN = {1016-443X,1420-8970},
   MRCLASS = {20F65 (20F67 57M20)},
  MRNUMBER = {3320891},
MRREVIEWER = {Vassilis\ Metaftsis},
       DOI = {10.1007/s00039-015-0314-y},
       URL = {https://doi.org/10.1007/s00039-015-0314-y},
}

@article{DahmaniKrishnaMutanguha,
  title={Hyperbolic hyperbolic-by-cyclic groups are cubulable},
  author={Dahmani, Fran{\c{c}}ois and Krishna MS, Suraj and Mutanguha, Jean Pierre},
  journal={arXiv preprint arXiv:2306.15054},
  year={2023}
}

@phdthesis{Kudlinska:thesis,
  title={Profinite and residual properties of fibred groups},
  author={Kudlinska, Monika},
  year={2024},
  school={University of Oxford}
}

@article{lyman:new-cat-0, 
    title={Some new {CAT}(0) free-by-cyclic groups}, 
    volume={73}, 
    ISSN={0026-2285}, 
    url={http://arxiv.org/abs/1909.03097}, 
    DOI={10.1307/mmj/20205989}, 
    abstractNote={We show the existence of several new infinite families of polynomially-growing automorphisms of free groups whose mapping tori are CAT(0) free-by-cyclic groups. Such mapping tori are thick, and thus not relatively hyperbolic. These are the first families comprising infinitely many examples for each rank of the nonabelian free group; they contrast strongly with Gersten’s example of a thick free-by-cyclic group which cannot be a subgroup of a CAT(0) group.}, 
    note={arXiv:1909.03097 [math]}, 
    number={3}, 
    journal={Michigan Mathematical Journal}, 
    author={Lyman, Robert Alonzo}, 
    year={2023},
    month=7 
}

@article{wu-ye:questions, 
    title={Some questions related to free-by-cyclic groups and tubular groups}, 
    url={http://arxiv.org/abs/2504.14192}, 
    DOI={10.48550/arXiv.2504.14192}, 
    journal={arXiv:2504.14192 [math]}, 
    publisher={arXiv}, 
    author={Wu, Xiaolei and Ye, Shengkui}, 
    year={2025}, 
    month=4 
}

@article{feighn-handel:algorithmic-constructions, 
    title={Algorithmic constructions of relative train track maps and CTs}, 
    volume={12}, 
    ISSN={1661-7207, 1661-7215}, 
    DOI={10.4171/ggd/466},  
    number={3}, 
    journal={Groups, Geometry, and Dynamics}, 
    publisher={European Mathematical Society}, 
    author={Feighn, Mark and Handel, Michael}, 
    year={2018}, 
    month=8, 
    pages={1159--1238}, 
    language={en} 
}
\end{document}